\documentclass[11pt]{amsart}
\usepackage{latexsym,amscd,amssymb, graphicx, color, amsthm, bm, amsmath, cancel, enumitem,young,chessboard}  
\usepackage{tikz}
\usepackage[margin=1in]{geometry}
\usepackage{hyperref}
\usepackage[all,cmtip]{xy}
\usepackage{pgfplots}
\pgfplotsset{compat=1.18}
\newcommand*\circled[1]{\tikz[baseline=(char.base)]{
            \node[shape=circle,draw,inner sep=2pt] (char) {#1};}}

\numberwithin{equation}{section}

\newtheorem{theorem}{Theorem}[section]
\newtheorem{proposition}[theorem]{Proposition}
\newtheorem{corollary}[theorem]{Corollary}
\newtheorem{lemma}[theorem]{Lemma}
\newtheorem{conjecture}[theorem]{Conjecture}
\newtheorem{observation}[theorem]{Observation}

\newtheorem{problem}[theorem]{Problem}

\theoremstyle{definition}
\newtheorem{defn}[theorem]{Definition}

\newcommand{\Hilb}{{\mathrm{Hilb}}}

\newcommand{\type}{{\mathrm{type}}}

\newcommand{\one}{{\mathbf{1}}}

\newcommand{\rdeg}{{\mathrm{rdeg}}}
\newcommand{\cdeg}{{\mathrm{cdeg}}}

\newcommand{\lft}{{\mathrm{left}}}

\newcommand{\mult}{{\mathrm{mult}}}

\newcommand{\row}{{\mathrm{row}}}
\newcommand{\col}{{\mathrm{col}}}

\newcommand{\ttheta}{{\bm \theta}}

\newcommand{\littlesl}{{\mathfrak{sl}}}

\newcommand{\rowsum}{{\mathrm{rowsum}}}
\newcommand{\colsum}{{\mathrm{colsum}}}

\newcommand{\ddeg}{{\mathrm{ddeg}}}

\newcommand{\Eseries}{{\mathrm{E}}}

\newcommand{\spn}{{\mathrm{span}}}

\newcommand{\Frob}{{\mathrm{Frob}}}
\newcommand{\symm}{{\mathfrak{S}}}
\newcommand{\II}{{\mathbf{I}}}
\newcommand{\RRR}{{\mathbf{R}}}
\newcommand{\gr}{{\mathrm {gr}}}
\newcommand{\BBB}{{\mathcal{B}}}
\newcommand{\SSS}{{\mathcal{S}}}

\newcommand{\PPP}{{\mathcal{P}}}

\newcommand{\spl}{{\mathsf {split}}}
\newcommand{\merge}{{\mathsf {merge}}}
\newcommand{\shift}{{\mathsf {shift}}}

\newcommand{\Zpoints}{{\mathcal{Z}}}

\newcommand{\CCC}{{\mathcal{C}}}

\newcommand{\fin}{{\mathrm{fin}}}
\newcommand{\initial}{{\mathrm{in}}}
\newcommand{\Stab}{{\mathrm{Stab}}}

\newcommand{\DDD}{{\mathcal{D}}}

\newcommand{\TTT}{{\mathcal{T}}}
\newcommand{\CC}{{\mathbb{C}}}
\newcommand{\QQ}{{\mathbb{Q}}}
\newcommand{\ZZ}{{\mathbb{Z}}}
\newcommand{\PP}{{\mathbb{P}}}
\newcommand{\RR}{{\mathbb{R}}}

\newcommand{\WWW}{{\mathcal{W}}}

\newcommand{\LLL}{\mathcal{L}}
\newcommand{\FF}{\mathbb{F}}

\newcommand{\Mat}{\mathrm{Mat}}
\newcommand{\Ind}{\mathrm{Ind}}

\newcommand{\shape}{\mathrm{shape}}
\newcommand{\SSYT}{\mathrm{SSYT}}

\newcommand{\xxx}{{\mathbf{x}}}
\newcommand{\yyy}{{\mathbf{y}}}

\newcommand{\zzz}{{\mathbf{z}}}
\newcommand{\eee}{{\mathbf{e}}}

\newcommand{\ddd}{{\mathbf{d}}}

\newcommand{\mb}{{\mathfrak{mb}}}

\newcommand{\conv}{\text{conv}}

\newcommand{\image}{\text{image}}
\newcommand{\trunc}{\text{trunc}}

\newcommand{\zigzag}{\text{zigzag}}

\begin{document}

\title[Zigzags, contingency tables, and quotient rings]
{Zigzags, contingency tables, and quotient rings}

\author{Jaeseong Oh}
\address{Korea Institute for Advanced Study}
\email{jsoh@kias.re.kr}
\author{Brendon Rhoades}
\address{University of California, San Diego}
\email{bprhoades@ucsd.edu}

\begin{abstract}
    Let $\xxx_{k \times p}$ be a $k \times p$ matrix of variables and let $\FF[\xxx_{k \times p}]$ be the polynomial ring in these variables. Given two weak compositions $\alpha,\beta \models_0 n$ of lengths $\ell(\alpha) = k$ and $\ell(\beta) = p$, we study the ideal $I_{\alpha,\beta} \subseteq \FF[\xxx_{k \times \ell}]$ generated by row sums, column sums, monomials in row $i$ of degree $> \alpha_i$, and monomials in column $j$ of degree $> \beta_j$. We prove results connecting algebraic properties of the quotient ring $R_{\alpha,\beta} := \FF[\xxx_{k \times \ell}]/I_{\alpha,\beta}$ with the set $\CCC_{\alpha,\beta}$ of $\alpha,\beta$-contingency tables. The standard monomial basis of $R_{\alpha,\beta}$ with respect to a diagonal term order is encoded by the matrix-ball avatar of the RSK correspondence. We describe the Hilbert series of $R_{\alpha,\beta}$ in terms of a zigzag statistic on contingency tables. The ring $R_{\alpha,\beta}$ carries a graded action of the product $\Stab(\alpha) \times \Stab(\beta)$ of symmetry groups of the sequences $\alpha = (\alpha_1,\dots,\alpha_k)$ and $\beta = (\beta_1,\dots,\beta_p)$; we describe how to calculate the isomorphism type of this graded action. Our analysis regards the set $\CCC_{\alpha,\beta}$ as a locus in the affine space $\Mat_{k \times p}(\FF)$ and applies  orbit harmonics  to this locus.
\end{abstract}

\maketitle

\section{Introduction}
\label{sec:Introduction}

Write $[n] := \{1,\dots, n\}$ and let $A = (a_{i,j})$ be a $k \times p$ matrix over the nonnegative integers. A set $Z \subseteq [k] \times [p]$ of matrix positions is a {\em zigzag} if it has the form 
$Z = \{ (i_1,j_1), \dots, (i_q,j_q) \}$ where $1 \leq i_1 \leq \cdots \leq i_q \leq k$ and $1 \leq j_1 \leq \cdots \leq j_q \leq p$.  Three zigzags are shown in stars as follows:
\begin{footnotesize}
\[
\begin{pmatrix}
    \cdot & \star & \star &  \star & \cdot & \cdot  & \cdot  \\
    \cdot & \cdot & \cdot & \cdot & \cdot & \star & \cdot  \\
    \cdot & \cdot & \cdot & \cdot & \cdot & \cdot & \cdot  \\
    \cdot & \cdot & \cdot &  \cdot & \cdot & \star & \star \\
    \cdot & \cdot & \cdot &  \cdot & \cdot & \cdot & \star \\
\end{pmatrix}, \quad 
\begin{pmatrix}
    \cdot & \star & \cdot &  \cdot & \cdot & \cdot  & \cdot  \\
    \cdot & \cdot & \cdot & \cdot & \cdot & \cdot & \cdot  \\
    \cdot & \star & \cdot & \cdot & \cdot & \cdot & \cdot  \\
    \cdot & \cdot & \star &  \star & \cdot & \star & \cdot \\
    \cdot & \cdot & \cdot &  \cdot & \cdot & \star & \star \\
\end{pmatrix}, \quad 
\begin{pmatrix}
    \cdot & \cdot & \cdot &  \cdot & \cdot & \cdot  & \cdot  \\
    \cdot & \cdot & \cdot & \cdot & \cdot & \cdot & \cdot  \\
    \cdot & \star & \star & \cdot & \star & \star & \star  \\
    \cdot & \cdot & \cdot &  \cdot & \cdot & \cdot & \cdot \\
    \cdot & \cdot & \cdot &  \cdot & \cdot & \cdot & \cdot \\
\end{pmatrix}.
\]\end{footnotesize}
Define the {\em zigzag number} of $A$ by
\begin{equation}
    \label{eq:zigzag-number-defn}
    \zigzag(A) := \max \left\{ \sum_{(i,j) \in Z} a_{i,j} \,:\, Z \subseteq [k] \times [\ell]  \text{ is a zigzag}\right\}.
\end{equation}
For example, we have
\[
\zigzag
\begin{footnotesize}
\begin{pmatrix}
    1 & 2 & 0 & 1  \\
    0 & 0 & 2  & 1 \\
    3 & 0 & 1 & 1 
\end{pmatrix}
\end{footnotesize}= 7
\]
as witnessed by either underlined zigzag of matrix positions 
\[ 
\begin{footnotesize}
\begin{pmatrix}
    \underline{\bf 1} & \underline{\bf 2} & 0 & 1  \\
    0 & 0 & \underline{\bf 2}  & 1 \\
    3 & 0 & \underline{\bf 1} & \underline{\bf 1} 
\end{pmatrix} 
\end{footnotesize} \quad \text{or} \quad 
\begin{footnotesize}
\begin{pmatrix}
    \underline{\bf 1} & \underline{\bf 2} & 0 & 1  \\
    0 & 0 & \underline{\bf 2}  & \underline{\bf 1} \\
    3 & 0 & 1 & \underline{\bf 1} 
\end{pmatrix}
\end{footnotesize}.
\]
When $A$ is the permutation matrix of a permutation $w$, the zigzag number $\zigzag(A)$ is the length of the longest increasing subsequence of $w$. More generally, if $A \mapsto (P,Q)$ under the Robinson-Schensted-Knuth correspondence, one has 
\begin{equation}
    \label{eq:zigzag-rsk}
    \zigzag(A) = \lambda_1 \text{ where $\lambda = (\lambda_1,\lambda_2,\dots)$ is the common shape of $P$ and $Q$.}
\end{equation}

Let $n \geq 0$ and consider weak compositions $\alpha = (\alpha_1,\dots,\alpha_k)$ and $\beta = (\beta_1,\dots,\beta_p)$ of $n$. An {\em $\alpha,\beta$-contingency table} is a $k \times p$ matrix $A = (a_{i,j})$ over $\ZZ_{\geq 0}$ with row sums $\alpha$ and column sums $\beta$. For example, the $3 \times 4$ matrix in the last paragraph is a $(4,3,5),(4,2,3,3)$-contingency table. If $\alpha = \beta = (1^n)$, contingency tables are permutation matrices. We write $\CCC_{\alpha,\beta}$ for the set of $\alpha,\beta$-contingency tables. 

Let $\FF$ be a field of characteristic zero, let $\xxx_{k \times p} = (x_{i,j})_{1 \leq i \leq k, \, 1 \leq j \leq p}$ be a $k \times p$ matrix of variables, and let $\FF[\xxx_{k \times p}]$ be the polynomial ring over these variables. The following quotient of $\FF[\xxx_{k \times p}]$ is our object of study.

\begin{defn}
    \label{def:contingency-quotient}
    Let $n \geq 0$ and let $\alpha = (\alpha_1,\dots,\alpha_k)$ and $\beta = (\beta_1,\dots,\beta_p)$ be weak compositions of $n$. Let $I_{\alpha,\beta} \subseteq \FF[\xxx_{k \times p}]$ be the ideal generated by $\dots$
    \begin{itemize}
        \item all row sums $x_{i,1} + \cdots + x_{i,p}$ for $1 \leq i \leq k$, 
        \item all column sums $x_{1,j} + \cdots + x_{k,j}$ for $1 \leq j \leq p$, 
        \item all monomials $x_{i,1}^{a_1} \cdots x_{i,p}^{a_p}$ in variables in the same row for which $a_1 + \cdots + a_p > \alpha_i$, and 
        \item all monomials $x_{1,j}^{b_1} \cdots x_{k,j}^{b_k}$ in variables in the same column for which $b_1 +\cdots + b_k > \beta_j$.
    \end{itemize}
    We denote the corresponding quotient of $\FF[\xxx_{k \times p}]$ by
    \[ R_{\alpha,\beta} := \FF[\xxx_{k \times p}] / I_{\alpha,\beta}.\]
\end{defn}

The ideal $I_{\alpha,\beta} \subseteq \FF[\xxx_{k \times p}]$ is homogeneous, so $R_{\alpha,\beta}$ is a graded ring. When $\alpha = \beta = (1^n)$, the quotient $R_{\alpha,\beta}$ was defined and studied by the second author \cite{Rhoades}.
The product of symmetric groups $\symm_k \times \symm_p$ acts on the matrix of variables $\xxx_{k \times p}$ by row and column permutation, and the polynomial ring $\FF[\xxx_{k \times p}]$ is thus a graded $\symm_k \times \symm_p$-module. The ideal $I_{\alpha,\beta}$ is in general not stable under the full $\symm_k \times \symm_p$-action. However, if $\Stab(\alpha) \subseteq \symm_k$ and $\Stab(\beta) \subseteq \symm_p$ are the stabilizers of the sequences $\alpha = (\alpha_1,\dots,\alpha_k)$ and $\beta = (\beta_1,\dots,\beta_p)$, the ideal $I_{\alpha,\beta}$ is stable under the subgroup $\Stab(\alpha) \times \Stab(\beta) \subseteq \symm_k \times \symm_p$. The quotient $R_{\alpha,\beta}$ is a graded $\Stab(\alpha) \times \Stab(\beta)$-module.

Algebraic properties of $R_{\alpha,\beta}$ are governed by combinatorial properties of contingency tables in $\CCC_{\alpha,\beta}$. Recall that the {\em Hilbert series} of a graded vector space $V = \bigoplus_{d \geq 0} V_d$ is $\Hilb(V;q) := \sum_{d \geq 0} \dim V_d \cdot q^d$. We prove (Corollary~\ref{cor:hilbert-series}) that the Hilbert series of $R_{\alpha,\beta}$ is encoded by the zigzag statistic via
\begin{equation}
    \label{eq:hilbert-intro}
    \Hilb(R_{\alpha,\beta};q) = \sum_{A \in \CCC_{\alpha,\beta}} q^{n - \zigzag(A)}.
\end{equation}
Equation~\eqref{eq:hilbert-intro} reduces to \cite[Cor. 3.13]{Rhoades} in the permutation matrix case. The equality \eqref{eq:hilbert-intro} of Hilbert series reflects a deeper relationship between $R_{\alpha,\beta}$ and the Robinson-Schensted-Knuth correspondence.

The RSK correspondence is a  bijection $A \mapsto (P,Q)$ from $\ZZ_{\geq 0}$-matrices with finite support to ordered pairs $(P,Q)$ of semistandard Young tableaux of the same shape.   RSK  is usually described in terms of an insertion algorithm.
In the permutation matrix case, Viennot gave \cite{Viennot} a beautiful `geometric' reformulation of this bijection using a `shadow line' construction. The {\em matrix-ball construction} (see e.g. \cite{Fulton}) generalizes Viennot's work to aribtrary $\ZZ_{\geq 0}$-matrices.
The Hilbert series formula \eqref{eq:hilbert-intro} will arise from a connection between the standard monomial theory of $R_{\alpha,\beta}$ and the matrix-ball construction. 

We use the matrix-ball construction to associate (Definition~\ref{def:m-monomial}) a monomial $\mb(A)$ of degree $n - \zigzag(A)$ to any contingency table $A \in \CCC_{\alpha,\beta}$. If $\prec$ is any diagonal term order (see Section~\ref{sec:Background}) on the monomials of $\FF[\xxx_{k \times p}]$, we prove (Theorem~\ref{thm:standard-monomial-basis}) that $\{ \mb(A)\,:\, A \in \CCC_{\alpha,\beta}\}$ is the standard monomial basis of $R_{\alpha,\beta}$ with respect to $\prec$, generalizing \cite[Thm. 3.12]{Rhoades}. The Hilbert series  \eqref{eq:hilbert-intro} is an immediate corollary. Since the construction $A \leadsto \mb(A)$ will encode an iteration of the matrix-ball construction, this result (loosely) says that the quotient ring $R_{\alpha,\beta}$ ``knows" the matrix-ball avatar of RSK.

Our results extend to the equivariant setting. The group $\Stab(\alpha) \times \Stab(\beta)$  acts on the set $\CCC_{\alpha,\beta}$ of contingency tables by row and column rearrangement; we write $\FF[\CCC_{\alpha,\beta}]$ for the associated permutation module. We prove (Corollary~\ref{cor:ungraded-module-structure}) that 
\begin{equation}
    R_{\alpha,\beta} \cong_{\Stab(\alpha) \times \Stab(\beta)} \FF[\CCC_{\alpha,\beta}]
\end{equation}
as ungraded $\Stab(\alpha) \times \Stab(\beta)$-modules. The ring $R_{\alpha,\beta}$ therefore gives a graded refinement of the action of $\Stab(\alpha) \times \Stab(\beta)$ on contingency tables. We give a description of this graded refinement in Theorem~\ref{thm:graded-module-structure}. We describe a technique involving symmetric functions for calculating the graded character of $R_{\alpha,\beta}$; see Problem~\ref{prob:psi-on-schur} and Proposition~\ref{prop:psi-on-h}.

The algebra of $R_{\alpha,\beta}$ and the combinatorics of $\CCC_{\alpha,\beta}$ are connected via the {\em orbit harmonics} technique of  deformation theory. Let $\FF^N$ be affine $N$-space over $\FF$ with coordinate ring $\FF[\xxx_N]$ and let $\Zpoints \subseteq \FF^N$ be a finite point set. We have the vanishing ideal
\begin{equation}
\II(\Zpoints) := \{ f \in \FF[\xxx_N] \,:\, f(\zzz) = 0 \text{ for all } \zzz \in \Zpoints \}.\end{equation}
Let $\FF[\Zpoints]$ be the $\FF$-vector space with basis $\Zpoints$. Since $\Zpoints$ is finite, by Lagrange interpolaion we have $\FF[\Zpoints] \cong \FF[\xxx_N]/\II(\Zpoints)$ as ungraded $\FF$-vector spaces. If we denote by $\gr \, \II(\Zpoints) \subseteq \FF[\xxx_N]$ the associated graded ideal of $\II(\Zpoints)$ and set 
\begin{equation}
\RRR(\Zpoints) := \FF[\xxx_N]/\gr \, \II(\Zpoints),
\end{equation}
there holds the further vector space isomorphism
\begin{equation}
\label{eq:orbit-harmonics-intro}
    \FF[\Zpoints] \cong \FF[\xxx_N]/\II(\Zpoints) \cong \RRR(\Zpoints)
\end{equation}
where $\RRR(\Zpoints)$ is a graded $\FF$-vector space. If $G \subseteq GL_N(\FF)$ is a finite matrix group and $\Zpoints$ is stable under the action of $G$, we may regard \eqref{eq:orbit-harmonics-background} as an isomorphism of ungraded $G$-modules where $\RRR(\Zpoints)$ is also a graded $G$-module. In geometric terms, the transition $\II(\Zpoints) \leadsto \gr \, \II(\Zpoints)$ is the flat limit which linearly deforms the locus $\Zpoints$ to a subscheme of degree $\# \Zpoints$ supported at the origin, as shown below in a case where $\# \Zpoints = 6$ is stabilized by a linear group $G \cong \symm_3$.

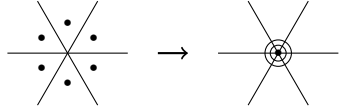
\begin{center}
 \begin{tikzpicture}[scale = 0.2]
\draw (-4,0) -- (4,0);
\draw (-2,-3.46) -- (2,3.46);
\draw (-2,3.46) -- (2,-3.46);

 \fontsize{5pt}{5pt} \selectfont
\node at (0,2) {$\bullet$};
\node at (0,-2) {$\bullet$};

\node at (-1.73,1) {$\bullet$};
\node at (-1.73,-1) {$\bullet$};
\node at (1.73,-1) {$\bullet$};
\node at (1.73,1) {$\bullet$};

\draw[thick, ->] (6,0) -- (8,0);

\draw (10,0) -- (18,0);
\draw (12,-3.46) -- (16,3.46);
\draw (12,3.46) -- (16,-3.46);

\draw (14,0) circle (15pt);
\draw(14,0) circle (25pt);
\node at (14,0) {$\bullet$};

 \end{tikzpicture}
\end{center}

Orbit harmonics provides a plethora of interested graded rings in algebraic combinatorics with ties to coinvariant theory \cite{Kostant}, Springer fibers \cite{HRS, Griffin}, delta operators \cite{HRS}, cyclic sieving \cite{OR}, Coxeter-Catalan combinatorics \cite{ARR}, Ehrhart theory \cite{ReinerRhoades}, torus-equivariant cohomology \cite{CMR,CDHP, GM}, and Donaldson-Thomas theory \cite{RRT}. When $\Zpoints$ is a combinatorially interesting locus, one often has an algebraically interesting  ring $\RRR(\Zpoints)$. Giving an explicit presentation of $\RRR(\Zpoints)$ as a quotient of $\FF[\xxx_N]$ by computing a finite generating set of $\gr \, \II(\Zpoints)$ is difficult in general.

We view the set $\CCC_{\alpha,\beta}$ of contingency tables as a  locus of points in the affine space $\Mat_{k \times p}(\FF)$ of $k \times p$ matrices over $\FF$.\footnote{Since $\FF$ has characteristic zero, this yields $\# \CCC_{\alpha,\beta}$ distinct matrices.} Identifying the coordinate ring of $\Mat_{k \times p}(\FF)$ with $\FF[\xxx_{k \times p}]$, the orbit harmonics quotient $\RRR(\CCC_{\alpha,\beta})$ is the quotient $\FF[\xxx_{k \times p}]/\gr \, \II(\CCC_{\alpha,\beta})$. We prove (Theorem~\ref{thm:standard-monomial-basis}) that
\begin{equation}
\label{eq:intro-orbit-harmonics-identification}
    \RRR(\CCC_{\alpha,\beta}) = R_{\alpha,\beta}
\end{equation}
so that $R_{\alpha,\beta}$ is the orbit harmonics quotient for the contingency table locus $\CCC_{\alpha,\beta}$.

This paper is not the first to apply orbit harmonics to matrix loci. In \cite{Rhoades}, the second author applied orbit harmonics to the locus $\symm_n = \CCC_{(1^n),(1^n)}$ of $n \times n$ permutation matrices, obtaining \eqref{eq:intro-orbit-harmonics-identification} in the permutation matrix case. It was proven \cite[Thm. 3.12]{Rhoades} that the standard monomial theory of $\RRR(\symm_n)$ is governed by the Viennot shadow construction on permutation matrices. Liu extended \cite{Liu} this work to the complex reflection group $G(r,1,n)$ of $r$-colored permutations. The rings $\RRR(\Zpoints)$ coming from various loci $\Zpoints \subseteq \Mat_{n \times n}(\FF)$ corresponding to involutions in $\symm_n$ were studied by Liu-Ma-Rhoades-Zhu \cite{LMRZ}.

While the main results in this paper are mostly extensions of results in \cite{Rhoades} from permutation matrices to general contingency tables, the {\em proofs} of our results are substantially more algebraically involved. Unlike in permutation setting $\alpha = \beta = (1^n)$\footnote{where the convenient \cite[Lem. 3.1]{Rhoades} applies}, for general compositions $\alpha$ and $\beta$ the task of finding explicit elements $f \in I_{\alpha,\beta}$ to study the Gr\"obner theory of $I_{\alpha,\beta}$ and bound the quotient $R_{\alpha,\beta} = \FF[\xxx_{k \times p}]/I_{\alpha,\beta}$ from above has proven daunting. To get around this issue, we introduce auxiliary quotient rings $R_\ddd$ involving a 1-dimensional array of variables (Definition~\ref{def:one-row-quotient}). By applying Lefschetz techniques and characterizing the Macaulay-inverse system associated to $R_\ddd$ (Proposition~\ref{prop:inverse-system-spanning}), we indirectly deduce the existence of $f \in I_{\alpha,\beta}$ with monomial support in a single row or column which have strategic leading terms. By applying certain polarization operators (Definition~\ref{def:polarization-operators}), we can `spread these elements  out' across multiple rows and columns, thus attaining sufficiently many diagonal leading terms to bound the quotient $R_{\alpha,\beta}$ from above. 

The rest of the paper is organized as follows. In {\bf Section~\ref{sec:Background}} we give background material on combinatorics,  representation theory, and commutative algebra. The technical {\bf Section~\ref{sec:One-row}} studies the one-row quotient rings $R_\ddd$ for later application to $R_{\alpha,\beta}$. {\bf Section~\ref{sec:Matrix}} recalls the matrix-ball construction of RSK and proves associated combinatorial results. In {\bf Section~\ref{sec:Hilbert}} we prove the identification \eqref{eq:intro-orbit-harmonics-identification} of $R_{\alpha,\beta}$ and $\RRR(\CCC_{\alpha,\beta})$ and calculate its standard monomial basis in terms of the matrix-ball construction. {\bf Section~\ref{sec:Module}} studies the graded module structure of $R_{\alpha,\beta}$. We close in {\bf Section~\ref{sec:Future}} with some open problems on log-concavity, top-heaviness, and graded Ehrhart theory.

\section{Background}
\label{sec:Background}

\subsection{Combinatorics} Let $n \geq 0$. A {\em weak composition} of $n$ is a sequence $\alpha = (\alpha_1,\dots,\alpha_k)$ in $\ZZ_{\geq 0}$ such that $\alpha_1 + \cdots + \alpha_k = n$. We write $\alpha \models_0 n $ to indicate that $\alpha$ is a weak composition of $n$ and $\ell(\alpha) = k$ for the number of parts of $\alpha$. For example, $\alpha = (1,2,0,2) \models_0 5$ satisfies $\ell(\alpha) = 4$.

Let $A = (a_{i,j})$ be a $k \times p$ matrix. The {\em row vector} $\row(A) = (\row(A)_1,\dots,\row(A)_k)$ and {\em column vector} $\col(A) = (\col(A)_1,\dots,\col(A)_p)$ of $A$ are defined by
\[ r_i := \sum_{j=1}^p a_{i,j}  \text{ for $1 \leq i \leq k$ and }  c_j := \sum_{i=1}^k a_{i,j} \text{ for $1 \leq j \leq p$}.\]
As in the introduction, if $\alpha,\beta \models_0 n$, a $\ZZ_{\geq 0}$-matrix $A$ is an {\em $\alpha,\beta$-contingency table} if $\row(A) = \alpha$ and $\col(A) = \beta$. We write $\CCC_{\alpha,\beta}$ for the family of $\alpha,\beta$-contingency tables.

A {\em partition} of $n$ is a weakly decreasing sequence $\lambda = (\lambda_1 \geq \cdots \geq \lambda_k)$ of positive integers with $\lambda_1 + \cdots + \lambda_k = n$. We write $\lambda \vdash n$ to indicate that $\lambda$ is a partition of $n$ and $\ell(\lambda)=k$ for the number of parts of $\lambda$. We identify partitions $\lambda\vdash n$ with their (English) Young diagrams consisting of $\lambda_i$ left-justified boxes in row $i$. For example, the Young diagram of $(3,2,2) \vdash 7$ is 
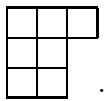
\[\begin{footnotesize}
    \begin{young} & &  \\ & \\ & \\  \end{young}.
\end{footnotesize}\]

Let $\lambda$ be a partition. A {\em semistandard Young tableau} (SSYT) of shape $\lambda$ is a filling $T: \lambda \to \ZZ_{>0}$ of its Young diagram which is weakly increasing across rows and strictly increasing down columns. We write $\shape(T) = \lambda$ for the shape of $\lambda$. An example SSYT of shape $(3,2,2)$ is as follows
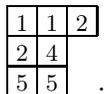
\[ \begin{footnotesize}
    \begin{young} 1 & 1 & 2 \\  2 & 4 \\ 5 & 5 \end{young}.
\end{footnotesize}\]
The {\em content} of a SSYT $T$ is the weak composition $\alpha = (\alpha_1,\alpha_2,\dots)$ where $\alpha_i$ is the number of boxes labeled $i$. The SSYT above has content $(2,2,0,1,2)$. We write $\SSYT(\lambda,\alpha)$ for the family of SSYT $T$ with shape $\lambda$ and content $\alpha$. The cardinality of this set is given by the {\em Kostka number} $K_{\lambda,\alpha} := \# \SSYT(\lambda,\alpha)$.

Let $\alpha,\beta \models_0 n$ be weak compositions. The {\em Robinson-Schensted-Knuth correspondence} (RSK) is an explicit bijection $A \mapsto (P,Q)$ from contingency tables $A \in \CCC_{\alpha,\beta}$ to ordered pairs $(P,Q)$ of $n$-box SSYT of the same shape $\lambda \vdash n$ where $P$ has content $\alpha$ and $Q$ has content $\beta$; see e.g. \cite{Fulton,Sagan,Stanley}. Thanks to this bijection, one has 
\begin{equation}
    \# \CCC_{\alpha,\beta} = \sum_{\lambda \vdash n} K_{\lambda,\alpha}\cdot K_{\lambda,\beta}.
\end{equation}
%The RSK correspondence is a fundamental bijection in algebraic combinatorics; see  \cite{Fulton,Sagan,Stanley} for textbook treatments. RSK is usually defined using an insertion algorithm, but we will use the equivalent {\em matrix-ball construction} recalled in Section~\ref{sec:Matrix} below.

\subsection{Representation theory}
Let $\symm_n$ be the symmetric group on $n$ letters and write $\FF[\symm_n]$ for its group algebra. If $\alpha = (\alpha_1,\dots,\alpha_k) \models_0 n$, we have the {\em parabolic subgroup} $\symm_\alpha := \symm_{\alpha_1} \times \cdots \times \symm_{\alpha_k} \subseteq \symm_n$.

Since $\FF$ has characteristic zero, irreducible representations of  $\symm_n$ over $\FF$ are indexed by partitions of $n$. If $\lambda \vdash n$ is a partition, we write $V^\lambda$ for the associated irreducible $\symm_n$-module. We recall the standard dictionary between $\symm_n$-representation theory and symmetric functions; see e.g. \cite{Sagan,Stanley} for more details.

Let $\Lambda = \bigoplus_{n \geq 0} \Lambda_n$ be the graded $\FF$-algebra of symmetric functions over an infinite alphabet $\xxx = (x_1,x_2,\dots)$. Vector space bases of $\Lambda_n$ are indexed by partitions $\lambda \vdash n$. We will use the {\em complete homogeneous basis} $\{h_\lambda\}$ and the {\em Schur basis} $\{s_\lambda\}$.

Let $V$ be a finite-dimensional $\symm_n$-module. There are unique multiplicities $c_\lambda \geq 0$ such that $V \cong \bigoplus_{\lambda \vdash n} c_\lambda V^\lambda$. The {\em Frobenius image} of $V$ is the symmetric function 
\begin{equation}
    \Frob(V) := \sum_{\lambda \vdash n} c_\lambda \cdot s_\lambda.
\end{equation}
For example, we have
\[ \Frob(V^\lambda) = s_\lambda \quad \text{and} \quad \Frob(M^\lambda) = h_\lambda \]
where $M^\lambda = \FF[\symm_n/\symm_\lambda] = \Ind_{\symm_\lambda}^{\symm_n}(\one_{\symm_\lambda})$ is the permutation action of $\symm_n$ on the  left cosets $\symm_n/\symm_\lambda$.

If $\mu = (\mu_1,\dots, \mu_r) \vdash n$, irreducible representations of the Young subgroup $\symm_\mu \subseteq \symm_n$ are indexed by $r$-tuples $(\lambda^{(1)},\dots, \lambda^{(r)})$ of partitions where $\lambda^{(i)} \vdash \mu_i$. We let $V^{(\lambda^{(1)},\dots, \lambda^{(r)})} := V^{\lambda^{(1)}} \otimes \cdots \otimes V^{\lambda^{(r)}}$ denote the $\symm_\mu$-irreducible associated to an $r$-tuple of partitions and set
\begin{equation}
    \Frob(V^{(\lambda^{(1)},\dots, \lambda^{(r)})}) := s_{\lambda^{(1)}} \otimes \cdots \otimes s_{\lambda^{(r)}} \in \Lambda_{\mu_1} \otimes \cdots \otimes \Lambda_{\mu_r}.
\end{equation}
Extending additively, we obtain an element $\Frob(V) \in \Lambda_{\mu_1} \otimes \cdots \otimes \Lambda_{\mu_r}$ for any $\symm_\mu$-module $V$.

\subsection{Commutative algebra} Let $\FF$ be a field of characteristic zero, let $\xxx_N = (x_1,\dots,x_N)$ be a list of $N$ variables, and let $\FF[\xxx_N]$ be the polynomial ring in these variables over $\FF$. A total order $\prec$ on the monomials in $\FF[\xxx_N]$ is a {\em term order} if 
\begin{itemize}
    \item we have $1 \preceq m$ for all monomials $m$, and
    \item if $m_1 \prec m_2$ then $m_1 \cdot m_3 \prec m_2 \cdot m_3$ for all monomials $m_1,m_2,m_3$.
\end{itemize}
For example, the {\em lexicographical} term order $<_{lex}$ is given by $x_1^{a_1} \cdots x_N^{a_N} <_{lex} x_1^{b_1} \cdots x_N^{b_N}$ if there exists $1 \leq i \leq N$ with $a_1 = b_1, \dots, a_{i-1}=b_{i-1},$ and $a_i < b_i$. We define lexicographical order $(a_1,\dots,a_N) <_{lex} (b_1,\dots,b_N)$ on integer sequences using the same rule.

Let $\prec$ be a term order on $\FF[\xxx_N]$. If $f \in \FF[\xxx_N]- \{0\}$ is a nonzero polynomial, we write $\initial_\prec(f)$ for the $\prec$-largest monomial which appears in $f$ with nonzero coefficient. If $I \subseteq \FF[\xxx_N]$ is an ideal, the {\em initial ideal} of $I$ is 
\begin{equation}
    \initial_\prec(I) := (\initial_\prec(f) \,:\, f \in I - \{0\}) \subseteq \FF[\xxx_N].
\end{equation}
A monomial $m \in \FF[\xxx_N]$ is a {\em standard monomial of $I$ with respect to $\prec$} if $m \notin \initial_\prec(I)$. The set of all standard monomials descends to a basis of $\FF[\xxx_N]/I$ known as the {\em standard monomial basis}. In particular, the standard monomial basis is uniquely determined from the ideal $I$ and the term order $\prec$.

For $1 \leq i \leq N$, let $\partial/\partial x_i: \FF[\xxx_N] \to \FF[\xxx_N]$ be the standard partial derivative operator. If $f = f(x_1,\dots,x_N) \in \FF[\xxx_N]$, we have an operator $\partial f: \FF[\xxx_N] \to \FF[\xxx_N]$ given by
\[ \partial f := f(\partial/\partial x_1, \dots, \partial/\partial x_N).\]
This gives rise to a map $\odot: \FF[\xxx_N] \times \FF[\xxx_N] \to \FF[\xxx_N]$ given by $f \odot g := (\partial f)(g)$. The $\odot$-action gives $\FF[\xxx_N]$ the structure of an $\FF[\xxx_N]$-module. If $I \subseteq \FF[\xxx_N]$ is a homogeneous ideal, the {\em Macaulay-inverse system} (or {\em harmonic space}) of $I$ is 
\begin{equation}
    I^\perp := \{ g \in \FF[\xxx_N] \,:\, f \odot g = 0 \text{ for all $f \in I$} \}.
\end{equation}
The inverse system $I^\perp$ is a graded $\FF$-linear subspace of $\FF[\xxx_N]$. Furthermore, we have 
\begin{equation}
\label{eq:qoutient-harmonic-hilbert}
    \Hilb(\FF[\xxx_N]/I;q) = \Hilb(I^\perp ;q).
\end{equation}
Equation~\ref{eq:qoutient-harmonic-hilbert} uses the fact that $\FF$ has characteristic zero. For general fields, the definition of $I^\perp$ must be adjusted using divided power rings; see e.g. \cite{ReinerRhoades}.

If $f \in \FF[\xxx_N] - \{0\}$ is a nonzero polynomial, write $\tau(f)$ for the top-degree homogeneous component of $f$. That is, if $f = f_d + \cdots + f_1 + f_0$ where $f_i$ is homogeneous of degree $i$ and $f_d \neq 0$, we have $\tau(f) = f_d$. If $I \subseteq \FF[\xxx_N]$ is an ideal, the {\em associated graded ideal} $\gr \, I \subseteq \FF[\xxx_N]$ is 
\begin{equation}
    \gr \, I := (\tau(f) \,:\, f \in I, \, f \neq 0 ).
\end{equation}
The ideal $\gr \, I$ is homogeneous by construction so that $\FF[\xxx_N]/\gr \, I$ is a graded ring. If $\BBB$ is a set of homogenoeus polynomials which descends to an $\FF$-basis of $\FF[\xxx_N]/\gr \, I$, then $\BBB$ also descends to an $\FF$-basis of $\FF[\xxx_N]/I$; see e.g. \cite[Lem. 3.15]{Rhoades}.

We regard $\FF[\xxx_N]$ as the coordinate ring of affine $N$-space $\FF^N$. If $\Zpoints \subseteq \FF^N$ is a finite locus, we have the vanishing ideal
\begin{equation}
    \II(\Zpoints) := \{ f \in \FF[\xxx_N] \,:\, f(\zzz) = 0 \text{ for all $\zzz \in \Zpoints$} \}.
\end{equation}
Let $\FF[\Zpoints]$ be the ring of functions $\Zpoints \to \FF$ with pointwise operations. Since $\Zpoints$ is finite, Lagrange interpolation gives the identification of ungraded $\FF$-algebras $\FF[\Zpoints] = \FF[\xxx_N]/\II(\Zpoints)$. By the last paragraph, we have the further isomorphism
\begin{equation}
    \label{eq:orbit-harmonics-background}
    \FF[\Zpoints] = \FF[\xxx_N]/\II(\Zpoints) \cong_\FF \FF[\xxx_N]/\gr \, \II(\Zpoints)
\end{equation}
of $\FF$-vector spaces\footnote{This will never be an isomorphism of $\FF$-algebras unless $\# \Zpoints \leq 1$.} where 
\begin{equation}
    \RRR(\Zpoints) := \FF[\xxx_N]/\gr \, \II(\Zpoints)
\end{equation}
has the additional structure of a graded $\FF$-vector space. If $G \subseteq GL_N(\FF)$ is a finite matrix group acting on $\FF^N$ and the locus $\Zpoints$ is $G$-stable, the identifications \eqref{eq:orbit-harmonics-background} are isomorphisms of $G$-modules where $\RRR(\Zpoints) = \FF[\xxx_N]/\gr \, \II(\Zpoints)$ has the additional structure of a graded $G$-module.

We consider polynomial rings over matrices of variables. If $\xxx_{k \times p} = (x_{i,j})$ is a $k \times p$ matrix of variables and $A = (a_{i,j}) \in \Mat_{k \times p}(\ZZ_{\geq 0})$, we write $\xxx^A := \prod_{i=1}^k \prod_{j=1}^p x_{i,j}^{a_{i,j}}$ for the corresponding monomial in $\FF[\xxx_{k \times p}]$. The set $\{ \xxx^A \,:\, A \in \Mat_{k \times p}(\ZZ_{\geq 0}) \}$ is an $\FF$-basis of $\FF[\xxx_{k \times p}]$. The {\em row degree} $\rdeg(\xxx^A) = (r_1,\dots,r_k)$ and {\em column degree} $\cdeg(\xxx^A) = (c_1,\dots,c_p)$ of the monomial $\xxx^A$ are inherited from the matrix $A$ via
\[ \rdeg(\xxx^A) := \row(A) \quad \text{and} \quad \cdeg(\xxx^A) := \col(A).\]
The {\em diagonal degree} of $\xxx^A$ is \[\ddeg(\xxx^A) = (d_1,d_2,\dots,d_{k+p-1}) \text{ where } d_q = \sum_{i+j-1=q} a_{i,j}.\] 
For example, if $A$ is the $3 \times 4$ matrix
\begin{footnotesize}
\[ A = \begin{pmatrix} 1 & 2 & 0 & 1 \\ 0 & 2 & 0 & 1 \\ 3 & 0 & 1 & 1 \end{pmatrix}\]
\end{footnotesize}
we have \[\rdeg(\xxx^A) = (4,3,5), \quad  \cdeg(\xxx^A) = (4,4,1,3), \quad  \text{and}  \quad \ddeg(\xxx^A) = (1,2,5,1,2,1).\]
A polynomial $f \in \FF[\xxx_{k \times p}]$ is {\em row-homogeneous} (resp. {\em column-homogeneous}) if every monomial $\xxx^A$ appearing in $f$ has the same row degree (resp. column degree).

A term order $\prec$ on the monomials of $\FF[\xxx_{k \times p}]$ is {\em diagonal} if 
\begin{center}
$m \prec m'$ whenever $\ddeg(m) <_{lex} \ddeg(m')$.
\end{center}
Diagonal term orders are easy to generate; simply take lexicographical order with respect to any ordering of the variables $x_{i,j}$ for which $x_{i,j} \succ x_{i',j'}$ whenever $i + j < i' + j'$. Diagonal term orders may be regarded as two-dimensional extensions of lexicographical order.

\begin{observation}
    \label{obs:diagonal-restriction}
    Let $\prec$ be a diagonal term order on the monomials of $\FF[\xxx_{k \times p}]$. Upon restriction to monomials in any row or column of the variable matrix $\xxx_{k \times p}$, the term order $\prec$ is lexicographical order on monomials in the variables in that row or column.
\end{observation}

\section{One-row quotients}
\label{sec:One-row}

In this technical section, we study the following graded quotient of a polynomial ring in $n$ variables. The only result we will use in future sections is Corollary~\ref{cor:deduced-polynomials}. On a first reading, it may be best to skip this section and accept Corollary~\ref{cor:deduced-polynomials} on faith.

\begin{defn}
    \label{def:one-row-quotient}
    Let $\ddd = (d_1,\dots,d_n)$ be a list of $n$ nonnegative integers and set $d := d_1 + \cdots + d_n$. Let $I_\ddd \subseteq \FF[\xxx_n]$ be the ideal with generating set 
    \[ I_\ddd = (x_1^{d_1 + 1}, \dots, x_n^{d_n + 1}, x_1 + \cdots + x_n).\]
    Let $R_\ddd := \FF[\xxx_n]/I_\ddd$ be the corresponding quotient ring.
\end{defn}

Roughly speaking, the ideal $I_\ddd$ of Definition~\ref{def:one-row-quotient} is a kind of restriction of the contingency table ideals $I_{\alpha,\beta}$ to a single row or column. The goal of this section is to calculate the standard monomial basis of $R_\ddd$ with respect to the lexicographical term order.   As a first step, we compute the Hilbert series of $R_\ddd$ using Lefschetz theory.

\subsection{Lefschetz theory and Hilbert series} Let $A = \bigoplus_{i=0}^d A_i$ be a graded $\FF$-algebra with $A_d \neq 0$. The algebra $A$ satisfies {\em Poincar\'e duality} if $A_d \cong \FF$ and the canonical map $A_i \otimes A_{d-i} \to A_d \cong \FF$ is a perfect pairing for all $0 \leq i \leq d$. In this case, an element $L \in A_1$ is a {\em Lefschetz element} if the map $(-) \times L^{d-2i}: A_i \to A_{d-1}$ is bijective for all $i \leq \frac{d}{2}$. The following result is known to experts.

\begin{proposition}
    \label{prop:lefschetz} Let $\ddd = (d_1,\dots,d_n) \in (\ZZ_{\geq 0})^n$ satisfy $\sum_{i=1}^n d_i = d$ and let $J_\ddd \subseteq \FF[\xxx_n]$ be the ideal 
    \[ J_\ddd := (x_1^{d_1+1},\dots,x_n^{d_n+1}).\]
    The quotient $\FF[\xxx_n]/J_\ddd$ satisfies Poincar\'e duality and $L := x_1 + \cdots + x_n$ is a Lefschetz element. 
\end{proposition}

\begin{proof}
    Poincar\'e duality is easy from the definitions. The Lefschetz property follows from Stanley's theory of differential posets \cite{StanleyDifferential} as follows.
    
    Let $C_m$ be the chain with $m$ elements and consider the direct product $P_\ddd := C_{d_1} \times \cdots \times C_{d_n}$ with its componentwise partial order. The graded poset $P_\ddd$ is a differential poset \cite[Prop. 5.1]{StanleyDifferential}. The rank function on $P_\ddd$ gives a grading on the vector space $\FF[P_\ddd]$ with respect to which there holds an isomorphism $\FF[\xxx_n]/J_\ddd \cong \FF[P_\ddd]$ of graded $\FF$-vector spaces. Multiplication by $L$ translates under this isomorphism to the linear operator $U: \FF[P_\ddd]\to\FF[P_\ddd]$ given by $U: x \mapsto \sum_{x \lessdot y} y$. (Here $x \lessdot y$ means that $y$ covers $x$ in the poset $P_\ddd$.) If  $i < \frac{d}{2}$, \cite[Ex. 4.13]{StanleyDifferential} proves that the map $U^{d-2i}: \FF[P_\ddd]_i \to \FF[P_\ddd]_{d-i}$ is a bijection.
\end{proof}

Proposition~\ref{prop:lefschetz} has the following geometric interpretation. Let $\PP^m$ be the $m$-dimensional complex projective space of lines through the origin in $\CC^{m+1}$. The product $\PP^{d_1} \times \cdots \times \PP^{d_n}$ has singular cohomology presentation
\begin{equation}
    H^*(\PP^{d_1} \times \cdots \times \PP^{d_n};\FF) = \FF[\xxx_n]/J_\ddd.
\end{equation}
Here the variable $x_i$ corresponds to the first Chern class $c_1(\LLL_i) \in H^2(\PP^{d_1} \times \cdots \times \PP^{d_n};\FF)$ where $\LLL_i$ is the tautological line bundle over the $i^{th}$ factor of $\PP^{d_1} \times \cdots \times \PP^{d_n}$. Since $\PP^{d_1} \times \cdots \times \PP^{d_n}$ is a smooth complex projective variety, its cohomology ring satisfies Poincar\'e duality and admits a Lefschetz element. Proposition~\ref{prop:lefschetz} says that $L = x_1 + \cdots + x_n$ is one such Lefschetz element. We will use Proposition~\ref{prop:lefschetz} to obtain the Hilbert series of $\FF[\xxx_d]/I_\ddd$; see \cite{KR} for an application of Proposition~\ref{prop:lefschetz} to fermionic diagonal coinvariant rings.

 For any $i$, multiplication by $L = x_1 + \cdots + x_n$ gives a map $(-) \times L : (\FF[\xxx_n]/J_\ddd)_i \to (\FF[\xxx_n]/J_\ddd)_{i+1}$. By Proposition~\ref{prop:lefschetz} this map is injective for $i < \frac{d}{2}$ and surjective for $i \geq \frac{d}{2}$.
 Since $I_\ddd = J_\ddd + (L)$, the Hilbert series of $R_\ddd$ is given by
\begin{multline}
    \label{eq:one-row-hilbert}
    \Hilb(R_\ddd;q) = \trunc_{\leq \frac{d}{2}} \left\{  (1-q) \times \Hilb(\FF[\xxx_n]/J_\ddd;q)\right\} \\ = \trunc_{\leq \frac{d}{2}} \left\{  (1-q) \times \prod_{i=1}^n (1+q+ \cdots + q^{d_i})\right\}.
\end{multline}
Here $\trunc_{\leq \frac{d}{2}}\{-\}$ is the operator which truncates power series to degrees $\leq \frac{d}{2}$. 
Evaluating \eqref{eq:one-row-hilbert} at $q \to 1$ gives the vector space dimension of $R_\ddd$.  We have the following combinatorial interpretation of $\dim_\FF R_\ddd$.

\begin{lemma}
    \label{lem:dimension-tableau}
    Let $\ddd = (d_1,\dots,d_n)$ with $d := \sum_{i=1}^n d_i$. Write $\TTT_\ddd$ for the set of two-row rectangular SSYT which contain $\leq d_i$ copies of $i$ for each $1 \leq i\leq n$. The following quantities are equal.
    \begin{enumerate}
        \item The vector space dimension $\dim_\FF R_\ddd$.
        \item The number of sequences $(a_1 \leq \cdots \leq a_m)$ which arise as the first row of a tableau $T \in \TTT_\ddd$.
        \item The number of sequences $(b_1 \leq\cdots\leq b_m)$ which arise as the second row of a tableau $T \in \TTT_\ddd$.
    \end{enumerate}
\end{lemma}

The second author is grateful to Tanny Libman for discussions related to Lemma~\ref{lem:dimension-tableau}.

\begin{proof}
    We only prove (1) = (2); the proof of (1) = (3) is similar and left to the reader. Let $\WWW_{\ddd,m}$ be the set of weak compositions 
    \[ \WWW_{\ddd,m} := \{ \beta \models_0 m \,:\, \ell(\beta)= n \text{ and } \beta_i \leq d_i \text{ for all $1 \leq i \leq n$} \}. \]Then $\dim_\FF (\FF[\xxx_n]/J_\ddd)_m = \# \WWW_{\ddd,m}$ where $(\FF[\xxx_n]/J_\ddd)_m$ is the degree $m$ piece of $\FF[\xxx_n]/J_\ddd$. The Hilbert series formula \eqref{eq:one-row-hilbert} implies
    \begin{equation}
    \label{eq:one-row-piece-dimension}
        \dim_\FF (R_\ddd)_m =
        \begin{cases}
        \# \WWW_{\ddd,m} - \# \WWW_{\ddd,m-1} & m \leq \frac{d}{2}, \\
        0 & m > \frac{d}{2}.
        \end{cases}
    \end{equation}
    We express the set $\TTT_\ddd$ of tableaux as a disjoint union
    \begin{equation}
        \TTT_d = \bigsqcup_{m=0}^{\lfloor \frac{d}{2} \rfloor} \TTT_{\ddd,m} \quad \text{where} \quad 
        \TTT_{\ddd,m} := \{ T \in \TTT_d \,:\, T \text{ has $m$ columns} \}.
    \end{equation}
    We have a function
    \begin{equation}
        \varphi: \TTT_{\ddd,m} \longrightarrow \WWW_{\ddd,m}
    \end{equation}
    given by $\varphi: T \mapsto (c_1,\dots,c_n)$ where $c_i$ is the number of $i$'s in the first row of $T$.
    For example, if $n = 5, \ddd = (1,3,3,2,2),$ and $m=4$ we have
    \[ T = 
    \begin{footnotesize}
        \begin{young} 1 & 2 & 2 & 4 \\ 2 & 3 & 4 & 5  \end{young}
    \end{footnotesize} \in \TTT_{\ddd,m}\]
    with $\varphi(T) = (1,2,0,1,0) \in \WWW_{\ddd,m}$. We will prove (1) = (2) by showing
    \begin{equation}
        \label{eq:phi-image-goal}
        \# \image(\varphi) = \# \WWW_{\ddd,m} - \# \WWW_{\ddd,m-1} \quad \text{for all $m \leq \frac{d}{2}$.}
    \end{equation}
    The idea in proving Equation~\eqref{eq:phi-image-goal} is to construct an injective map
    \begin{equation}
        \psi: \WWW_{\ddd,m-1} \hookrightarrow \WWW_{\ddd,m}
    \end{equation}
    for each $m \leq \frac{d}{2}$ such that we have the disjoint union
    \begin{equation}
        \WWW_{\ddd,m} = \image(\varphi) \sqcup \image(\psi).
    \end{equation}
    
    The function $\psi: \WWW_{\ddd,m-1} \to \WWW_{\ddd,m}$ is constructed as follows. Start with a weak composition $\beta = (\beta_1,\beta_2,\dots,\beta_n) \in \WWW_{\ddd,m-1}$. We use the running example of 
    \[ \ddd = (3,2,3,3,2,2) \quad \text{and} \quad \beta = (3,0,0,2,1,0)\]
    so that $n = 6$ and $m = 7$. We will place dots over various entries in the composition $\beta$. An entry $\beta_i$ is called {\em saturated} if the number of dots above that entry plus $\beta_i$ equals $d_i$. We perform the following procedure.
    \begin{enumerate}
        \item Working from right to left, place $\beta_i$ dots above  unsaturated entries strictly right of, and as close as possible to, position $i$. If there are insufficiently many unsaturated entries to do this, place as many dots as possible and call $\beta_i$ {\em unsatisfied}.
        \item Define $\psi(\beta)$ by incrementing the rightmost unsaturated entry by 1.
    \end{enumerate}
    The dotting procedure in our example proceeds as follows:
    \[ (3,0,0,2,1,0) \leadsto (3,0,0,2,1,\dot{0}) \leadsto (3,0,0,2,\dot{1},\ddot{0}) \leadsto (3,\ddot{0},\dot{0},2,\dot{1},\ddot{0})\]
    where we used $\ddd = (3,2,3,3,2,2)$ to determine the saturation condition at each stage. In this case, every entry is satisfied and the entries $\dot{0}$ and $2$ are unsaturated. Since $2$ is the rightmost unsaturated entry, we increment the 2 to obtain
    \[ \psi: (3,0,0,2,1,0) \mapsto (3,0,0,3,1,2).\]
    Observe that the new 3 is the leftmost (and in this case, the only) unsatisfied element of $\psi(\beta)$.

    We claim that $\psi$ is a well-defined function $\WWW_{\ddd,m-1} \to \WWW_{\ddd,m}$. Since $m \leq \frac{d}{2}$ by assumption and $\beta_i \leq d_i$ for all $i$, there will be at least one unsaturated entry after completing Step (1), so that in Step (2) there will be an unsaturated entry $\beta_i$ to increment. Since $\beta_i < d_i$ for any unsaturated entry, we have $\varphi(\beta) \in \WWW_{\ddd,m}$. This concludes the proof that $\psi: \WWW_{\ddd,m-1} \to \WWW_{\ddd,m}$ is a well-defined function.

    The next step is to prove that $\psi: \WWW_{\ddd,m-1} \to \WWW_{\ddd,m}$ is an injection. To do this, we show that any $\beta \in \WWW_{\ddd,m-1}$ can be recovered from $\psi(\beta)$.
    The key observation is that the element $\beta_i$ which was incremented to $\beta_i + 1$ in the calculation $\beta \mapsto \psi(\beta)$ is the leftmost unsatisfied element in $\psi(\beta)$. Indeed, when performing Step (1) it is impossible to place $\beta_i + 1$ dots when reaching position $i$ since the elements to the right of position $i$ are saturated after $\beta_i$ dots are placed (here we use that $\beta_i$ is unsaturated in $\beta$). The entry $\beta_i + 1$ is therefore unsatisfied in $\psi(\beta)$. Since $\beta_i$ is unsaturated in $\beta$, we also know that $\beta_i + 1$ is the leftmost unsatisfied entry in $\psi(\beta)$ since any unsatisfied entry to the left of $\beta_i + 1$ could have distributed dots over $\beta_i + 1$ to achieve satisfaction. Therefore, a given elment $\beta \in \WWW_{\ddd,m-1}$ may be recovered from $\psi(\beta)$ by applying Step (1) to $\psi(\beta)$ and decrementing the leftmost unsatisfied element by 1. We conclude that $\psi: \WWW_{\ddd,m-1} \hookrightarrow \WWW_{\ddd,m}$ is an injection.

    Finally, we prove the disjoint union decomposition $\WWW_{\ddd,m} = \image(\varphi) \sqcup \image(\psi)$.  Let $\beta \in \WWW_{\ddd,m}$. If every entry of $\beta$ is satisfied, we claim that $\beta = \varphi(T)$ for some  tableau $T \in \TTT_{\ddd,m}$. Indeed, the first row of $T$ is determined by $\beta$ and the second row of $T$ is determined by the positions of the dots after performing Step (1). For example, if $\ddd = (3,2,3,3,2,2)$ one has 
    \[ \beta =  (2,1,0,0,1,0) \leadsto (2,1,0,0,1,\dot{0}) \leadsto(2,1,\dot{0},0,1,\dot{0}) \leadsto (2,\dot{1},\ddot{0},0,1,\dot{0}) \leadsto 
    \begin{footnotesize}
    \begin{young} 
    1 & 1 & 2 & 5 \\ 2 & 3 & 3 & 6\end{young}
    \end{footnotesize} = T.\]
    Similarly, if $\beta = \varphi(T)$ for some $T \in \TTT_{\ddd,m}$, it is not hard to see that every entry in $\beta$ is satisfied. Therefore, we have
    \[ \image(\varphi) = \{ \beta \in \WWW_{\ddd,m} \,:\, \text{every entry in $\beta$ is satisfied} \}.\]
    On the other hand, suppose $\beta \in \WWW_{\ddd,m}$ has at least one unsatisfied entry. Define a composition $\bar{\beta} \in \WWW_{\ddd,m}$ by decrementing the leftmost unsatisfied entry $\beta_i$ in $\beta = (\beta_1,\dots,\beta_i,\dots,\beta_n)$, so that $\bar{\beta} = (\beta_1,\dots,\beta_i-1,\dots\beta_n)$. Observe that the entries $\beta_{i+1},\beta_{i+2},\dots,\beta_n$ are saturated in $\bar{\beta}$ since $\beta_i$ is unsatisfied in $\beta$. If $\beta_i - 1$ were saturated in $\bar{\beta}$, then some element to the right of $\beta_i$ would be unsatisfied in $\beta$. It follows that $\beta_i - 1$ is the rightmost unsaturated element of $\bar{\beta}$, so that $\psi(\bar{\beta}) = \beta$. Thus, we have
    \[ \image(\psi) = \{ \beta \in \WWW_{\ddd.m} \,:\, \text{at least one entry in $\beta$ is unsatisfied} \}.\]
    This gives the required disjoint union decomposition $\WWW_{\ddd,m} = \image(\varphi) \sqcup \image(\psi)$ and completes the proof that (1) = (2).
\end{proof}

\subsection{Inverse system and standard monomials} The inverse system $I_\ddd^\perp \subseteq \FF[\xxx_n]$ is a graded subspace with the same Hilbert series as $\FF[\xxx_n]/I_\ddd$. We give a combinatorial spanning set of this inverse system. 

If $T$ is a two-row rectangular SSYT with entries $\leq n$, define a polynomial $f_T \in \FF[\xxx_n]$ by 
\begin{equation}
    f_T := \prod_{a \, < \, b \text{ is a column of } T} (x_a - x_b).
\end{equation}
For example, if $T$ is the tableau
\begin{footnotesize}
\begin{equation*}
    \begin{young}
        1 & 1 & 2 & 3 \\
        2 & 3 & 3 & 4
    \end{young}
\end{equation*}
\end{footnotesize}
then $f_T = (x_1 - x_2)(x_1  - x_3)(x_2 - x_3)(x_3 - x_4).$ The polynomials arising from the set $\TTT_\ddd$ of tableaux in Lemma~\ref{lem:dimension-tableau} span the inverse system $I_\ddd^\perp$.

\begin{proposition}
    \label{prop:inverse-system-spanning}
    The inverse system $I_\ddd^\perp \subseteq \FF[\xxx_n]$ is spanned over $\FF$ by $\{f_T \,:\, T \in \TTT_d\}$.
\end{proposition}

\begin{proof}
    Let $T \in \TTT_d$. We establish the membership $f_T \in I_\ddd^\perp$ by showing that $f_T$ is annihilated under the $\odot$-action by each generator of $I_\ddd$. Since $T$ has $\leq i$ copies of $d_i$, the $x_i$-degree of $f_T$ is $\leq d_i$ and we have $x_i^{d_i+1} \odot f_T = 0$. If $T$ has columns $(a_1 < b_1), \dots, (a_m < b_m)$, the Leibniz Rule implies 
    \begin{align*}
        (x_1 + \cdots + x_n) \odot f_T 
        &= (x_1 + \cdots + x_n) \odot \prod_{i=1}^m  \left[ (x_{a_i} - x_{b_i}) \right] \\
        &= \sum_{j=1}^m (x_1 + \cdots + x_n) \odot (x_{a_j} - x_{b_j}) \times \prod_{i \neq j} (x_{a_i} - x_{b_i})   \\ 
        &= 0.
    \end{align*}
    This establishes the containment of subspaces 
    \begin{equation}
    \label{eq:subspace-containment-one-row}
        \spn_\FF \{ f_T \,:\, T \in \TTT_d \} \subseteq I_\ddd^\perp
    \end{equation}
    of the polynomial algebra $\FF[\xxx_n]$. Write $\prec$ for the lexicographical term order. For $T \in \TTT_d$, it is not hard to see that the  $\prec$-leading term of $f_T$ is 
    \begin{equation}
        \initial_\prec (f_T) = x_1^{c_1} \cdots x_n^{c_n}
    \end{equation}
    where $c_i$ is the number of $i$'s in the first row of $T$. Since $\dim_\FF I_\ddd^\perp = \dim_\FF R_\ddd$, the equality (1) = (2) of Lemma~\ref{lem:dimension-tableau} forces the containment \eqref{eq:subspace-containment-one-row} to be an equality.
\end{proof}

For example, if $n = 3$ and $\ddd = (1,2,1)$, the set $\TTT_d$ consists of the tableaux
\begin{footnotesize}
\[ \varnothing, \quad
\begin{young}
    1 \\ 2
\end{young} \, \, , \quad 
\begin{young}
    1 \\ 3 
\end{young} \, \, , \quad 
\begin{young}
    2 \\ 3 
\end{young} \, \, , \quad 
\begin{young}
    1 & 2 \\ 2 & 3
\end{young}
\]
\end{footnotesize}
where $\varnothing$ is the empty tableau. It follows that $I_\ddd^\perp$ is spanned by \[\{ 1, \quad x_1-x_2, \quad x_1-x_3, \quad x_2 - x_3, \quad (x_1-x_2)(x_2-x_3) \}.\]
Observe that this spanning set is not linearly independent.

We will need a general relationship between the inverse system of a graded ideal in a polynomial ring and the standard monomial basis of the corresponding quotient. If $\prec$ is a term order  and if $f \in \FF[\xxx_n]$ is nonzero, let 
\begin{equation}
    \fin_\prec(f) := \text{the $\prec$-{\bf smallest} monomial appearing in $f$ with nonzero coefficient.}
\end{equation}
The following result is known to experts. Since the authors were unable to find a reference, we include a proof.

\begin{lemma}
    \label{lem:inverse-harmonic}
    Fix an arbitrary term order $\prec$ on the monomials of $\FF[\xxx_n]$ and let $I \subseteq \FF[\xxx_n]$ be a homogeneous ideal with inverse system $I^\perp \subseteq \FF[\xxx_n]$. The standard monomial basis of $\FF[\xxx_n]/I$ is the set $\{ \fin_\prec(f) \,:\, f \in I^\perp, \, f \neq 0 \}$.
\end{lemma}

\begin{proof}
    Let $\BBB$ be the $\prec$-standard monomial basis of $\FF[\xxx_n]/I$. We certainly have
    \begin{equation}
        \label{eq:inverse-harmonic-one}
        \sum_{m \in \BBB} q^{\deg(m)} = \Hilb(\FF[\xxx_n]/I;q).
    \end{equation}
    Write $\CCC := \{ \fin_\prec(f) \,:\, f \in I^\perp, \, f \neq 0 \}$. A Gaussian elimination argument shows
    \begin{equation}
        \label{eq:inverse-harmonic-two}
        \sum_{m \in \CCC} q^{\deg(m)} = \Hilb(I^\perp;q).
    \end{equation}If $f \in I^\perp - \{0\}$ and $\fin_\prec(f) \notin \BBB$, there exists $g \in I - \{0\}$ such that $\initial_\prec(g) = \fin_\prec(f)$. We therefore have
    \begin{equation}
        g \odot f \doteq \initial_\prec(g) \odot \fin_\prec(f) \neq 0 
    \end{equation}
    where $\doteq$ means equality up to a nonzero scalar. This contradicts the assumption $f \in I^\perp$ and we deduce the containment
    \begin{equation}
    \label{eq:c-contained-in-b}
        \CCC \subseteq \BBB.
    \end{equation}
    Since $\Hilb(\FF[\xxx_n]/I;q) = \Hilb(I^\perp;q)$, Equations~\eqref{eq:inverse-harmonic-one} and \eqref{eq:inverse-harmonic-two} imply that the containment \eqref{eq:c-contained-in-b} is an equality.
\end{proof}

We are in a position to describe the lexicographical standard monomial basis of $R_\ddd$.

\begin{proposition}
    \label{prop:one-row-standard-basis}
    Let $\ddd = (d_1,\dots,d_n) \in (\ZZ_{\geq 0})^n$ and consider the lexicographical term order $\prec$ on the monomials of $\FF[\xxx_n]$. The standard monomial basis of $R_\ddd$ is 
    \[ \{ \text{monomials } x_1^{e_1} \cdots x_n^{e_n} \,:\, \text{there exists $T \in \TTT_\ddd$ with $e_i$ copies of $i$ in its second row \}.}\]
\end{proposition}

Continuing our example of $n = 3$ and $\ddd = (1,2,1)$, we conclude that $\{1,x_2,x_3,x_2x_3\}$ is the standard monomial basis of $R_\ddd$.  Observe that distinct tableaux $T \in \TTT_d$ can give rise to the same monomial.

\begin{proof}
    If $T \in \TTT_\ddd$ has $e_i$ copies of $i$ in its second row, it is not hard to see that $\fin_\prec(f_T) = x_1^{e_1} \cdots x_n^{e_n}$. The result follows from Proposition~\ref{prop:inverse-system-spanning} and Lemma~\ref{lem:inverse-harmonic}.
\end{proof}

A tableau  $T \in \TTT_\ddd$ has $\leq d_i$ copies of $i$. From this it is easily seen that if $x_1^{e_1} \cdots x_n^{e_n}$ arises in the standard monomial basis of Proposition~\ref{prop:one-row-standard-basis}, one has
\[ 2 e_1 + 2 e_2 + \cdots + 2 e_{i-1} + e_i \leq d_1 + d_2 + \cdots + d_{i-1} \text{ for all $1 \leq i \leq n$}.
\]The following corollary of Proposition~\ref{prop:one-row-standard-basis} is the only result from this section which will be used later.

\begin{corollary}
     \label{cor:deduced-polynomials}
    Let $\ddd = (d_1,\dots,d_n) \in (\ZZ_{\geq 0})^n$ and let $m = x_1^{e_1} \cdots x_n^{e_n}$ be a monomial in $\FF[\xxx_n]$. If 
    \[ 2e_1 + 2e_2 + \cdots + 2e_{i-1} + e_i > d_1 + d_2 + \cdots + d_{i-1} \]
    for some $1 \leq i \leq n$, there exists a nonzero polynomial $f \in I_\ddd$ such that $\initial_\prec(f) = m$ where $\prec$ is the lexicographical term order.
\end{corollary}

The authors do not know an explicit formula for the elements $f \in I_\ddd$ guaranteed by Corollary~\ref{cor:deduced-polynomials}. Finding such a formula could make the arguments in this paper significantly more direct. 

%To give an example, if $n = 3$ and $\ddd = (d_1,d_2,d_3) =  (5,4,7)$ then {\tt sage} gives the following element of the reduced lexicographical Gr\"obner basis of $I_\ddd$:
%\[ f  = x_2^4 \cdot x_3^2 + \frac{4}{3} \cdot x_2^3 \cdot x_3^3 + x_2^2 \cdot x_3^4 + \frac{2}{5} \cdot x_2 \cdot x_3^5 + \frac{1}{15} \cdot x_3^6 \]
%corresponding to the sequence $(e_1,e_2,e_3) = (0,4,2)$ which satisfies $2 e_1 + 2 e_2 + e_3 > d_1 + d_2$.
%It is unclear how to predict the coefficients of this polynomial from the sequence $\ddd = (5,4,7)$. There does not seem to be a significantly simpler element of $I_\ddd$ with lexicographical leading term $x_2^4 \cdot x_3^2$.

\section{Matrix-ball construction}
\label{sec:Matrix}

Let $\alpha,\beta \models_0 n$ be weak compositions. The RSK correspondence \cite{Knuth,Schensted} is a bijection
\[ \CCC_{\alpha,\beta} \xrightarrow{\,\, \sim \, \, } \bigsqcup_{\lambda \vdash n} \SSYT(\lambda,\alpha) \times \SSYT(\lambda,\beta).\]
This bijection was discovered, and is most commonly worked with, via an insertion algorithm. The matrix-ball construction is a `geometric' formulation of RSK which will be more useful for our purposes.

Let $A = (a_{i,j}) \in \Mat_{k \times p}(\ZZ_{\geq 0})$ and write $A^{(1)} := A$. We form a $k \times p$ grid where the cell with matrix coordinates $(i,j)$ contains $a_{i,j}$ balls running from northwest to southeast. For example, if $k = 3, p = 4$, and 
\begin{footnotesize}
\[ A^{(1)} = A = \begin{pmatrix} 1 & 2 & 0 & 1 \\ 0 & 0 & 2 & 1 \\ 3 & 0 & 1 & 1 \end{pmatrix}\]
\end{footnotesize}
is the matrix appearing in the introduction, the grid is shown below.

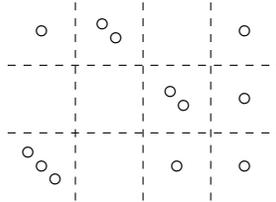
\begin{center}
    \begin{tikzpicture}[scale = 0.9]
            \draw[dashed] (0,1) -- (4,1);
            \draw[dashed] (0,0) -- (4,0);

            \draw[dashed] (1,-1) -- (1,2);
            \draw[dashed] (2,-1) -- (2,2);
            \draw[dashed] (3,-1) -- (3,2);

            \node at (0.5,1.5) {$\circ$};
            \node at (0.5,-0.5) {$\circ$};
            \node at (0.7,-0.7) {$\circ$};
            \node at (0.3,-0.3) {$\circ$};

            \node at (1.4,1.6) {$\circ$};
            \node at (1.6,1.4) {$\circ$};

            \node at (3.5,1.5) {$\circ$};

            \node at (2.4,0.6) {$\circ$};
            \node at (2.6,0.4) {$\circ$};

            \node at (2.5,-0.5) {$\circ$};
            \node at (3.5,0.5) {$\circ$};
            \node at (3.5,-0.5) {$\circ$};
    \end{tikzpicture}
\end{center}

We recursively label these balls with $1,2,\dots$ as follows.
\begin{itemize}
    \item If $b$ is a ball and there does not exist a distinct ball north or west of $b$, the label of $b$ is $1$.
    \item Otherwise, the label of $b$ is the quantity 
    \[ \max \{ \text{ label of $b'$} \,:\, b' \text{ is north or west of $b$} \, \} + 1.\]
\end{itemize}
In our example, we have the following labels.
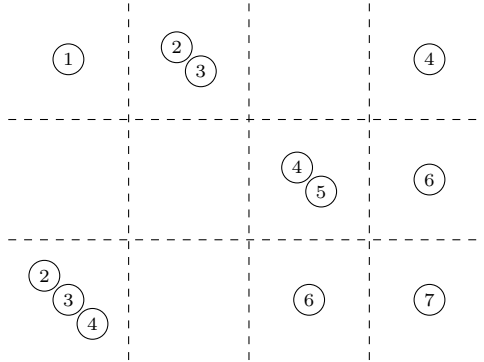
\begin{center}
    \begin{tikzpicture}[scale = 1.6]
            \draw[dashed] (0,1) -- (4,1);
            \draw[dashed] (0,0) -- (4,0);

            \draw[dashed] (1,-1) -- (1,2);
            \draw[dashed] (2,-1) -- (2,2);
            \draw[dashed] (3,-1) -- (3,2);

            \node at (0.5,1.5) {$\begin{tiny}\circled{1}\end{tiny}$};
            \node at (0.5,-0.5) {$\begin{tiny}\circled{3}\end{tiny}$};
            \node at (0.7,-0.7) {$\begin{tiny}\circled{4}\end{tiny}$};
            \node at (0.3,-0.3) {$\begin{tiny}\circled{2}\end{tiny}$};

            \node at (1.4,1.6) {$\begin{tiny}\circled{2}\end{tiny}$};
            \node at (1.6,1.4) {$\begin{tiny}\circled{3}\end{tiny}$};

            \node at (3.5,1.5) {$\begin{tiny}\circled{4}\end{tiny}$};

            \node at (2.4,0.6) {$\begin{tiny}\circled{4}\end{tiny}$};
            \node at (2.6,0.4) {$\begin{tiny}\circled{5}\end{tiny}$};

            \node at (2.5,-0.5) {$\begin{tiny}\circled{6}\end{tiny}$};
            \node at (3.5,0.5) {$\begin{tiny}\circled{6}\end{tiny}$};
            \node at (3.5,-0.5) {$\begin{tiny}\circled{7}\end{tiny}$};
    \end{tikzpicture}
\end{center}

A ball $b$ is {\em northern} if it has smallest row index among the set of balls sharing its label. A ball $b$ is {\em western} if it has smallest column index among the set of balls sharing its label.  The following diagram is our example with northern balls labeled $N$ and western balls labeled $W$.

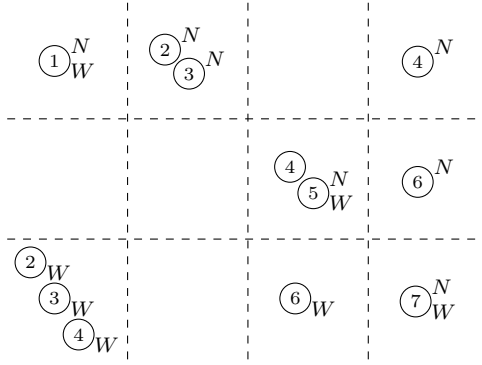
\begin{center}
    \begin{tikzpicture}[scale = 1.6]
            \draw[dashed] (0,1) -- (4,1);
            \draw[dashed] (0,0) -- (4,0);

            \draw[dashed] (1,-1) -- (1,2);
            \draw[dashed] (2,-1) -- (2,2);
            \draw[dashed] (3,-1) -- (3,2);

            \node at (0.5,1.5) {$\begin{tiny}\circled{1}^N_W\end{tiny}$};
            \node at (0.5,-0.5) {$\begin{tiny}\circled{3}_W\end{tiny}$};
            \node at (0.7,-0.8) {$\begin{tiny}\circled{4}_W\end{tiny}$};
            \node at (0.3,-0.2) {$\begin{tiny}\circled{2}_W\end{tiny}$};

            \node at (1.4,1.6) {$\begin{tiny}\circled{2}^N\end{tiny}$};
            \node at (1.6,1.4) {$\begin{tiny}\circled{3}^N\end{tiny}$};

            \node at (3.5,1.5) {$\begin{tiny}\circled{4}^N\end{tiny}$};

            \node at (2.35,0.6) {$\begin{tiny}\circled{4}\end{tiny}$};
            \node at (2.65,0.4) {$\begin{tiny}\circled{5}^N_W\end{tiny}$};

            \node at (2.5,-0.5) {$\begin{tiny}\circled{6}_W\end{tiny}$};
            \node at (3.5,0.5) {$\begin{tiny}\circled{6}^N\end{tiny}$};
            \node at (3.5,-0.5) {$\begin{tiny}\circled{7}^N_W\end{tiny}$};
    \end{tikzpicture}
\end{center}

Let $x_i$ be the number of northern balls in row $i$ and $y_j$ be the number of western balls in column $j$.  We have two one-row SSYT $(P_1,Q_1)$ where $P_1$ has $x_i$ copies of $i$ and $Q_1$ has $y_j$ copies of $j$. These will be the first rows of the image $A \mapsto (P,Q)$. In our example, we have 
$P_1 = \begin{footnotesize} \begin{young} 1 & 1 & 1 & 1 & 2 & 2 & 3 \end{young} \end{footnotesize}$ and $Q_1 = \begin{footnotesize} \begin{young} 1 & 1 & 1 & 1 & 3 & 3 & 4 \end{young} \end{footnotesize}$.

We use our labeled ball diagram to construct a matrix $A^{(2)}$ from $A^{(1)}$ as follows. For any fixed label $\ell$, the set of balls labeled $\ell$ lie in distinct matrix positions, and these positions have the form $\{ (i_1,j_1),\dots,(i_m,j_m)\}$ where $i_1 < \cdots < i_m$ and $j_1 > \cdots > j_m$. Let $B_\ell = (b_{i,j})) \in \Mat_{k \times p}(\ZZ_{\geq 0})$ be the $0,1$-matrix where 
\begin{equation}
    b_{i,j} = \begin{cases}
        1 & \text{if $(i,j) \in \{ (i_1,j_2), (i_2,j_3), \dots, (i_{m-1},j_m) \}$,} \\
        0 & \text{otherwise.}
    \end{cases}
\end{equation}
The matrix $A^{(2)}$ is defined by $A^{(2)} := \sum_{\ell \geq 0} B_\ell$ where the sum is matrix addition. We have componentwise inequalities
\[    \row(A^{(2)}) \leq \row(A^{(1)}) \quad \text{and} \quad \col(A^{(2)}) \leq \col(A^{(1)}). \]
In our running example, we have 
\begin{footnotesize}
\[ B_2 = B_3 =  \begin{pmatrix} 0 & 0 & 0 & 0 \\ 0 & 0 & 0 & 0 \\ 0 & 1 & 0 & 0 \end{pmatrix}, \quad B_4 = \begin{pmatrix} 0 & 0 & 0 & 0 \\ 0& 0 & 0 & 1 \\ 0 & 0  & 1 & 0 \end{pmatrix}, \quad B_6 = \begin{pmatrix} 0  & 0 & 0 & 0 \\ 0 & 0 & 0 & 0 \\ 0 & 0 & 0 & 1 \end{pmatrix},\]
\end{footnotesize}
and $B_1, B_5, B_7$ are the zero matrix. Therefore
\begin{footnotesize}
\[ A^{(2)} = B_1 + \cdots + B_7 = \begin{pmatrix} 0 & 0 & 0 & 0 \\ 0 & 0 & 0 & 1 \\ 0 & 2 & 1 & 1 \end{pmatrix}.\]
\end{footnotesize}

We repeat our ball labeling process on the matrix $A^{(2)}$. In our example, this yields the following figure.
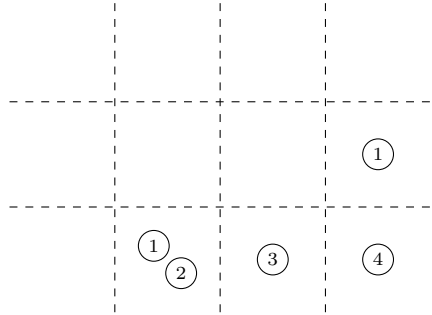
\begin{center}
    \begin{tikzpicture}[scale = 1.4]
            \draw[dashed] (0,1) -- (4,1);
            \draw[dashed] (0,0) -- (4,0);

            \draw[dashed] (1,-1) -- (1,2);
            \draw[dashed] (2,-1) -- (2,2);
            \draw[dashed] (3,-1) -- (3,2);

            \node at (1.37,-0.37) {$\begin{tiny}\circled{1}\end{tiny}$};
            \node at (1.63,-0.63) {$\begin{tiny}\circled{2}\end{tiny}$};

            \node at (3.5,0.5) {$\begin{tiny}\circled{1}\end{tiny}$};
            \node at (2.5,-0.5) {$\begin{tiny}\circled{3}\end{tiny}$};
            \node at (3.5,-0.5) {$\begin{tiny}\circled{4}\end{tiny}$};
    \end{tikzpicture}
\end{center}
Recording northern and western balls, we obtain the second rows $P_2$ and $Q_2$ of the image $(P,Q)$ of $A$ under RSK. In our example, we have $P_2 = \begin{footnotesize}\begin{young} 2 & 3 & 3 & 3 \end{young}\end{footnotesize}$ and $Q_2= \begin{footnotesize}\begin{young} 2 & 2 & 3 &4\end{young}\end{footnotesize}.$ 
We use the ball labelings to obtain $A^{(3)}$ from $A^{(2)}$ as before. In our example, this gives
\begin{footnotesize}
\[
A^{(3)} = \begin{pmatrix} 0 & 0 & 0 & 0 \\ 0 & 0 & 0 & 0 \\ 0 & 0 & 0 & 1 \end{pmatrix}.
\]
\end{footnotesize}
We have $P_3 = \begin{footnotesize}\begin{young} 3 \end{young}\end{footnotesize}$, $Q_3 = \begin{footnotesize}\begin{young} 4 \end{young}\end{footnotesize}$, and $A^{(4)}$ is the zero matrix. We conclude that our example matrix $A$ maps to $(P,Q)$ under RSK where
\[ P = 
\begin{footnotesize}\begin{young} 1 & 1 & 1 & 1 & 2 & 2 & 3 \\ 2 & 3 & 3 & 3 \\ 3 \end{young}\end{footnotesize} \quad \text{and} \quad 
Q = \begin{footnotesize}\begin{young} 1 & 1 & 1 & 1 & 3 & 3& 4 \\ 2 & 2 & 3 & 4 \\ 4 \end{young}\end{footnotesize}.\]

In general, given a nonzero matrix $A^{(i)} \in \Mat_{k \times p}(\ZZ_{\geq 0})$ we use the ball labeling procedure to define a new matrix $A^{(i+1)} \in \Mat_{k \times p}(\ZZ_{\geq 0})$. This new matrix satisfies
\begin{equation}
    \label{eq:matrix-inequalities}
    \row(A^{(i+1)}) < \row(A^{(i)}) \quad \text{and} \quad \col(A^{(i+1)}) < \col(A^{(i)})
\end{equation}
where $<$ is the componentwise partial order on sequences of integers of the same length. We also produce the rows $P_i$ and $Q_i$ of the image $(P,Q)$ of $A$ under RSK. This procedure continues until $A^{(i+1)}$ is the zero matrix. 

\begin{theorem}
    \label{thm:matrix-ball}
    For any $A \in \Mat_{k \times p}(\ZZ_{\geq 0})$, the figures $P$ and $Q$ produced by the above construction are semistandard Young tableaux. Furthermore, the map $A \mapsto (P,Q)$ is the RSK bijection 
    \[ \CCC_{\alpha,\beta} \xrightarrow{\,\, \sim \, \, } \bigsqcup_{\lambda \vdash n} \SSYT(\lambda,\alpha) \times \SSYT(\lambda,\beta).\]
\end{theorem}

Theorem~\ref{thm:matrix-ball} is a foundational result whose history is somewhat difficult to trace. When $A$ is a permutation matrix, Theorem~\ref{thm:matrix-ball} reduces to Viennot's shadow line construction \cite{Viennot}. The general case of $\ZZ_{\geq 0}$-matrices may be deduced from that of permutation matrices using `standardization' techniques; see e.g. \cite[p. 321]{Stanley}. Fulton gave \cite[Sec. 4.2]{Fulton} a textbook treament of Theorem~\ref{thm:matrix-ball}.

If $A \mapsto (P,Q)$ under RSK, the tableaux $P,Q$ encode a wealth of combinatorial information about the matrix $A$. For the matrix $A$ of our running example, we computed in the introduction that $\zigzag(A) = 7$. The first rows of both $P$ and $Q$ have length 7. The following result is essentially due to Schensted.

\begin{theorem}
    \label{thm:zigzag-characterization}
    Let $A$ be a $\ZZ_{\geq 0}$-matrix, suppose $A \mapsto (P,Q)$ under RSK, and let $\lambda$ be the common shape of $P$ and $Q$. We have
    \[ \zigzag(A) = \lambda_1.\]
    In particular, the partition $\lambda$ is a single row if and only if $A$ is a zigzag matrix.
\end{theorem}

When $\beta = (1^p)$, we may view the matrix $A$ as a length $p$ word $w_1 \dots w_p$ over $\ZZ_{\geq 0}$ and Theorem~\ref{thm:zigzag-characterization} is a famous result of Schensted \cite{Schensted}. While Theorem~\ref{thm:zigzag-characterization} is widely known and may be deduced from Schensted's result using standardization, we give a proof whose underlying idea will be useful in proving a key spanning result (Lemma~\ref{lem:spanning-lemma}) later on.

\begin{proof}
    The length $\lambda_1$ of the first row of $\lambda$ equals the number of distinct ball labels when performing the first iteration of the matrix-ball construction on $A$. If $Z$ is a zigzag of matrix positions, the balls in $Z$ must have distinct labels so that $\zigzag(A) \leq \lambda_1$. 

    We construct a zigzag $Z$ such that $\sum_{(i,j) \in Z} a_{i,j} = \lambda_1$ as follows. Let $(i_1,j_1)$ be the position of any ball labeled $\lambda_1$ and let $\ell_1 \leq \lambda_1$ be the smallest label of a ball in position $(i_1,j_1)$. If $\ell_1 > 1$, there exists a ball labeled $\ell_1 - 1$ in a position $(i_{2},j_{2})$ with $i_{2} \leq i_1$ and $j_{2} \leq j_1$. Observe that $\{ (i_{1}, j_{1}), (i_2,j_2) \}$ is a zigzag. Let $\ell_{2}$ be the smallest label of a ball in position $(i_{2},j_{2})$. If $\ell_{2} > 1$, there exists a ball labeled $\ell_{2} - 1$ in a position $(i_{3},j_{3})$ with $i_{3} \leq i_{2}$ and $j_{3} \leq j_{2}$. Then $\{ (i_{1},j_{1}), (i_{2},j_{2}), (i_3,j_3)\}$ is a zigzag. Continuing in this fashion, we obtain a zigzag $Z = \{ (i_1,j_1), \dots, (i_m,j_m) \}$ such that $a_{(i_1,j_1)} + \cdots + a_{(i_m,j_m)} = \lambda_1$, as required.
\end{proof}

Our running example may help clarify the proof of Theorem~\ref{thm:zigzag-characterization} given above. In this example, a possible zigzag $Z$ as in the above proof is shown in red as follows. 
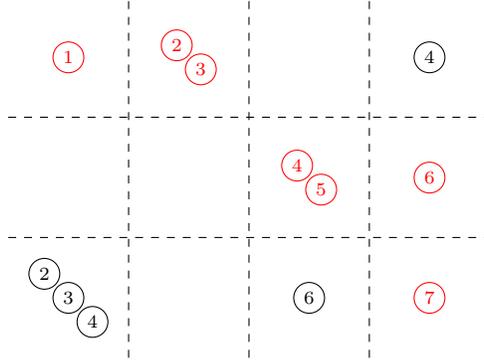
\begin{center}
    \begin{tikzpicture}[scale = 1.6]
            \draw[dashed] (0,1) -- (4,1);
            \draw[dashed] (0,0) -- (4,0);

            \draw[dashed] (1,-1) -- (1,2);
            \draw[dashed] (2,-1) -- (2,2);
            \draw[dashed] (3,-1) -- (3,2);

            \node at (0.5,1.5) {${\color{red} \begin{tiny}\circled{1}\end{tiny}}$};
            \node at (0.5,-0.5) {$\begin{tiny}\circled{3}\end{tiny}$};
            \node at (0.7,-0.7) {$\begin{tiny}\circled{4}\end{tiny}$};
            \node at (0.3,-0.3) {$\begin{tiny}\circled{2}\end{tiny}$};

            \node at (1.4,1.6) {${\color{red} \begin{tiny}\circled{2}\end{tiny}}$};
            \node at (1.6,1.4) {${\color{red} \begin{tiny}\circled{3}\end{tiny}}$};

            \node at (3.5,1.5) {$\begin{tiny}\circled{4}\end{tiny}$};

            \node at (2.4,0.6) {${\color{red} \begin{tiny}\circled{4}\end{tiny}}$};
            \node at (2.6,0.4) {${\color{red} \begin{tiny}\circled{5}\end{tiny}}$};

            \node at (2.5,-0.5) {$\begin{tiny}\circled{6}\end{tiny}$};
            \node at (3.5,0.5) {${\color{red} \begin{tiny}\circled{6}\end{tiny}}$};
            \node at (3.5,-0.5) {${\color{red} \begin{tiny}\circled{7}\end{tiny}}$};
    \end{tikzpicture}
\end{center}
One other zigzag is possible in this case, corresponding to choosing the $\begin{tiny}\circled{6}\end{tiny}$ in row 3 instead of the $\begin{tiny}\circled{6}\end{tiny}$ in row 2. These are exactly the zigzags witnessing $\zigzag(A) = 7$ shown in the introduction. The construction of $Z$ in Theorem~\ref{thm:zigzag-characterization} will play a crucial role in our analysis of the quotient ring $R_{\alpha,\beta}$.

If $A \in \CCC_{\alpha,\beta}$, we use the matrix-ball construction to attach a monomial to $A$ as follows. In the permutation matrix case, the following definition specializes to the shadow monomials $\mathfrak{s}(w)$ for $w \in \symm_n$ introduced in \cite{Rhoades}.

\begin{defn}
    \label{def:m-monomial}
    Let $A \in \CCC_{\alpha,\beta}$ be a contingency table and consider the matrix $B := A^{(2)}$ obtained from $A = A^{(1)}$ after one iteration of the matrix-ball construction. We define the {\em matrix-ball monomial} $\mb(A) \in \FF[\xxx_{k \times p}]$  by
    \[ \mb(A) := \xxx^B.\]
\end{defn}

By Theorem~\ref{thm:zigzag-characterization}, the degree of $\mb(A)$ is given by
\begin{equation}
\label{eq:m-monomial-degree}
    \deg(\mb(A)) = n - \zigzag(A) \quad \text{where } \sum_{i,j} a_{i,j} = n. 
\end{equation}
We will prove in Theorem~\ref{thm:standard-monomial-basis} that the monomials $\{\mb(A) \,:\, A \in \CCC_{\alpha,\beta}\}$ descend to an $\FF$-basis of $R_{\alpha,\beta}$. To this end, we develop a better understanding of the matrices $B$ arising in Definition~\ref{def:m-monomial}. A matrix $B = (b_{i,j}) \in \Mat_{k \times p}(\ZZ_{\geq 0})$ is an {\em $\alpha,\beta$-subtingency table} if
\[ \row(B) \leq \alpha \quad \text{and} \quad \col(B) \leq \beta\]
where the inequalities are componentwise.
We write $\SSS_{\alpha,\beta}$ for the family of $\alpha,\beta$-subtingency tables. One has the disjoint union decomposition
\[ \SSS_{\alpha,\beta} = \bigsqcup_{\substack{\alpha'\leq\alpha \\ \beta'\leq\beta}} \CCC_{\alpha',\beta'}\]
where $\leq$ is componentwise inequality. Subtingency tables  relate to the matrix-ball construction as follows.

\begin{proposition}
    \label{prop:monomial-is-subtingency}
    Let $A \in \CCC_{\alpha,\beta}$ be a contingency table and write $\mb(A) = \xxx^B$. Then $B \in \SSS_{\alpha,\beta}$ is a subtingency table.
\end{proposition}

\begin{proof}
    This is immediate from the inequalities \eqref{eq:matrix-inequalities}.
\end{proof}

The converse to Proposition~\ref{prop:shadow-characterization} is false: not every subtingency table in $\SSS_{\alpha,\beta}$ comes from applying the matrix-ball construction to a contingency table in $\CCC_{\alpha,\beta}$. We may use the matrix-ball construction to characterize the subtingency tables $B \in \SSS_{\alpha,\beta}$ such that $\xxx^B = \mb(A)$ for some contingency table $A \in \CCC_{\alpha,\beta}$. The following result should be compared with \cite[Lem. 3.6]{Rhoades}.

\begin{proposition}
    \label{prop:shadow-characterization}
    Let $\alpha,\beta \models_0 n$ and suppose $\ell(\alpha) = k$ and $\ell(\beta) = p$. Let $B \in \SSS_{\alpha,\beta}$ be a subtingency table. Perform the matrix-ball construction on $B$ and define sequences $x_1 \dots x_k$ and $y_1 \dots y_p$ over $\ZZ_{\geq 0}$ by 
    \[ x_i := \# \text{ of northern balls in row $i$}\]
    and
    \[ y_j := \# \text{ of western balls in column $j$}. \]
    Then $\xxx^B = \mb(A)$ for some contingency table $A \in \CCC_{\alpha,\beta}$ if and only if 
    \begin{equation}
    \label{eq:x-inequality}
    \sum_{q=1}^{i} x_q  \leq \sum_{q=1}^{i-1} (\alpha_q - \row(B)_q) \text{ for all $1 \leq i \leq k$}
    \end{equation}
    and 
    \begin{equation} 
    \label{eq:y-inequality}
    \sum_{q=1}^j y_q \leq \sum_{q=1}^{j-1} (\beta_q - \col(B)_q) \text{ for all $1 \leq j \leq p$}.\end{equation}
\end{proposition}

Note that the range of summation on the left-hand sides of the inequalities \eqref{eq:x-inequality} and \eqref{eq:y-inequality} includes one more index than on the right-hand sides of these inequalities.

\begin{proof}
    Apply RSK to the matrix $B$ to obtain a pair $B \mapsto (\overline{P},\overline{Q})$ of SSYT. By the matrix-ball construction, the first row of $\overline{P}$ contains $x_i$ copies of $i$ for all $1\leq i \leq k$. Similarly, the first row of $\overline{Q}$ contains $y_j$ copies of $j$ for all $1 \leq j \leq q$.

    Let $P_0$ be the one-row SSYT with $\alpha_i  -\row(B)_i$ copies of $i$ for $1 \leq i \leq k$. Similarly, let $Q_0$ be the one-row SSYT with $\beta_j - \col(B)_j$ copies of $j$ for $1 \leq j \leq p$. Let $P$ be the figure obtained by gluing the single row $P_0$ on top of $\overline{P}$ and define $Q$ analogously. 
    
    If $\sum_{i=1}^k x_i > n - \sum_{i=1}^k \row(B)_i$ (or equivalently $\sum_{j=1}^p y_j > n - \sum_{j=1}^p \col(B)_j$), the first row of $P$ and $Q$ will be strictly shorter than the second row and these figures will not have partition shape. The conditions \eqref{eq:x-inequality} and \eqref{eq:y-inequality} are satisfied if and only if both $P$ and $Q$ have partition shape and are semistandard.

    The proposition follows easily from the above observations. Suppose first that $\xxx^B = \mb(A)$ for some $A \in \CCC_{\alpha,\beta}$ and let $A \mapsto (P,Q)$ be the image of $A$ under RSK. Let $(\overline{P},\overline{Q})$ be the tableaux obtained by removing the first rows of $P$ and $Q$; we necessarily have $B \mapsto (\overline{P},\overline{Q})$. Since $P,Q$ are semistandard, the inequalities \eqref{eq:x-inequality} and \eqref{eq:y-inequality} must be satisfied. For the other direction, suppose the inequalities \eqref{eq:x-inequality} and \eqref{eq:y-inequality} are satisfied. Let $B \mapsto (\overline{P},\overline{Q})$ under RSK and define the pair $(P,Q)$ as in the above paragraphs. The inequalities \eqref{eq:x-inequality} and \eqref{eq:y-inequality} imply that $P$ and $Q$ are SSYT. If $A \mapsto (P,Q)$ under RSK we have $A \in \CCC_{\alpha,\beta}$ with $\xxx^B = \mb(A)$.
\end{proof}

An example may help in understanding Proposition~\ref{prop:shadow-characterization} and its proof. Suppose $\alpha = (2,3,2), \beta = (2,2,0,3)$, and
\begin{footnotesize}
\[ B = \begin{pmatrix} 0 & 0 & 0 & 0 \\ 0 & 0 & 0 & 1 \\ 0 & 2 & 0 & 1 \end{pmatrix}\]
\end{footnotesize}
so that $\row(B) = (0,1,3)$ and $\col(B) = (0,2,0,2)$. We compute the sequences $(x_1 ,x_2,x_3) = (0,1,2)$ and $(y_1,y_2,y_3,y_4) = (0,2,0,1)$. We have 
\[ y_1 + y_2 + y_3 + y_4 = 3 > (2-0) + (2-2) + (0-0) =  (\beta_1 - \col(B)_1) + (\beta_2 - \col(B)_2) + (\beta_3 - \col(B)_3)\]
which implies that $\xxx^B \neq \mb(A)$ for any contingency table $A \in \CCC_{\alpha,\beta}$. Indeed, applying RSK to $B$ yields the tableaux
\[ \overline{P} = 
\begin{footnotesize}
    \begin{young} 2 & 2 & 3 \\ 3 \end{young} 
\end{footnotesize}
\quad \text{and} \quad \overline{Q} = 
\begin{footnotesize}
  \begin{young} 2 & 2 & 4 \\ 4 \end{young}  
\end{footnotesize}.\]
Gluing the row $\begin{footnotesize}\begin{young} 1 & 1 & 4 \end{young}\end{footnotesize}$ of content $\beta - \col(B) = (2,2,0,3) - (0,2,0,2) = (2,0,0,1)$ to the top of $\overline{Q}$ yields
\[ \begin{footnotesize}
    \begin{young}
        1 & 1 & 4 \\ 2 & 2 & 4 \\ 4
    \end{young},
\end{footnotesize}\]
which is not a SSYT. Observe that the violation of the semistandard condition corresponds to the index $j=4$ for which $\sum_{q=1}^j y_q > \sum_{q=1}^{j-1} (\beta_q - \col(B)_q)$.

\section{Hilbert series and standard basis}
\label{sec:Hilbert}

In this section we prove (Theorem~\ref{thm:standard-monomial-basis}) that $\{ \mb(A) \,:\, A \in \CCC_{\alpha,\beta}\}$ is the standard monomial basis of $R_{\alpha,\beta}$ with respect to a diagonal term order $\prec$. The most difficult step is proving a spanning result (Lemma~\ref{lem:spanning-lemma}) which states the set $\{ \mb(A) \,:\, A \in \CCC_{\alpha,\beta}\}$ contains the $\prec$-standard monomial basis of $R_{\alpha,\beta}$. The relations $f \in I_{\alpha,\beta}$ which allow us to do this will be derived from elements of the `one-row' ideals $I_\ddd \subseteq \FF[x_1,\dots,x_n]$ in Section~\ref{sec:One-row} furnished by Corollary~\ref{cor:deduced-polynomials}. In order to transfer from a single row $\{x_1,\dots,x_n\}$ of variables to the variable matrix $\xxx_{k \times p}$, we introduce polarization operators. For technical reasons related to these polarization operators, we will need to express the defining ideal $I_{\alpha,\beta} \subseteq \FF[\xxx_{k \times p}]$ as a sum $I_{\alpha,\beta} = I^\colsum_\alpha + I^\rowsum_\beta$ of two smaller ideals.

\subsection{Polarization operators and zigzag lifting} For $1 \leq i \leq k$ and $1 \leq j \leq p$ we have the partial derivative operator $\partial/\partial x_{i,j} : \FF[\xxx_{k \times p}] \to \FF[\xxx_{k \times p}]$. We use partial derivatives to build the following polarization operators.

\begin{defn}
    \label{def:polarization-operators}
    Let $1 \leq i_0,i_1 \leq k$ be two row indices. The {\em row polarization operator} \[\rho_{i_1 \to i_0}: \FF[\xxx_{k \times p}] \longrightarrow \FF[\xxx_{k \times p}]\] is defined by the formula
    \[ \rho_{i \to i'}(f) = \sum_{j=1}^p x_{i_0,j} \cdot \frac{\partial f}{\partial x_{i_1,j}} \quad \text{for all $f \in \FF[\xxx_{k \times p}].$}\]
    Similarly, if $1 \leq j_0,j_1 \leq p$ are two column indices the {\em column polarization operator} \[ \kappa_{j_1 \to j_0}: \FF[\xxx_{k \times p}] \longrightarrow \FF[\xxx_{k \times p}]\] is defined by
    \[
    \kappa_{j_1 \to j_0}(f) := \sum_{i=1}^k x_{i,j_0} \cdot \frac{\partial f}{\partial x_{i,j_1}} \quad \text{for all $f \in \FF[\xxx_{k \times p}$].}
    \]
\end{defn}

The notation reflects the fact that $\rho_{i_1 \to i_0}$ `moves' variables from row $i_1$ to row $i_0$ and $\kappa_{j_1 \to j_0}$ `moves' variables from column $j_1$ to column $j_0$. More precisely, we have the following observation.  We use $\eee_k$ to denote the row vector with a 1 in position $k$ and zeros elsewhere. (The length of $\eee_k$ will always be clear from context.)

\begin{observation}
    \label{obs:polarization-effect-on-degree} 
    Suppose $f \in \FF[\xxx_{k \times p}]$ is row-homogeneous. If $1 \leq i_0,i_1 \leq k$ then $\rho_{i_1 \to i_0}(f)$ is row-homogeneous with
    \[ \rdeg(\rho_{i_1 \to i_0}(f)) = \rdeg(f) + \eee_{i_0} - \eee_{i_1}.\]
    Similarly, if $f$ is column-homogeneous and $1 \leq j_0,j_1\leq p$ then $\kappa_{j_1 \to j_0}$ is column-homogeneous with 
    \[ \cdeg(\kappa_{j_1 \to j_0}(f)) = \cdeg(f) + \eee_{j_0} - \eee_{j_1}.\]
\end{observation}

Polarization operators similar to those in Definition~\ref{def:polarization-operators} (and `higher' and `fermionic' variants thereof) play a crucial role in studying inverse systems associated to coinvariant rings \cite{Bergeron,Haiman,KRR,KR,KRSkein,Lentfer, RW}. As in \cite[Rmk. 3.9]{KRR}, our polarization operators generate actions of Lie algebras. It can be shown that the operators 
\[ \{ \rho_{i \to i+1}, \, \rho_{i+1 \to i} \,:\, 1 \leq i \leq k-1\} \quad \text{and} \quad \{ \kappa_{j \to j+1}, \, \kappa_{j+1 \to j} \,:\, 1 \leq j \leq p-1 \} \]
generate an action of the direct sum $\littlesl_k \oplus \littlesl_p$ of Lie algebras on the ring $\FF[\xxx_{k \times p}]$ where the standard Chevalley generators map to
\[(f_i,0) \mapsto \rho_{i \to i+1}, \quad (e_i,0) \mapsto \rho_{i+1 \to i}, \quad (0,f_j) \mapsto \kappa_{j \to j+1}, \quad (0,e_j) \mapsto \kappa_{j+1 \to j}.\]  Polarization operators  preserve ideals closely related to $I_{\alpha,\beta}$.

\begin{defn}
    \label{def:row-and-column-ideals}
    Let $\alpha = (\alpha_1,\dots,\alpha_k)$ be a weak composition with $k$ parts. The {\em column sum ideal} $I^\colsum_\alpha \subseteq \FF[\xxx_{k \times p}]$ is generated by $\dots$
    \begin{itemize}
        \item all column sums $x_{1,j} + \cdots + x_{k,j}$ for $1 \leq j \leq p$, and 
        \item all monomials $x_{i,1}^{a_1} \cdots x_{i,p}^{a_p}$ in variables in the same row $i$ for which $a_1 + \cdots + a_p > \alpha_i$.
    \end{itemize}
    Similarly, if $\beta = (\beta_1,\dots,\beta_p)$ is a weak composition with $p$ parts, the {\em row sum ideal} $I^\rowsum_\beta \subseteq \FF[\xxx_{k \times p}]$ is generated by $\dots$
    \begin{itemize}
        \item all row sums $x_{i,1} + \cdots + x_{i,p}$ for $1 \leq i \leq k$, and 
        \item all monomials $x_{1,j}^{b_1} \cdots x_{k,j}^{b_k}$ in variables in the same column $j$ for which $b_1 + \cdots + b_k > \beta_j.$
    \end{itemize}
\end{defn}

The ideal $I^\colsum_\beta$ is generated by sums of columns and products of variables within rows, while $I^\rowsum_\alpha$ is the other way around. We hope the notation $(-)^\rowsum$ and $(-)^\colsum$ will minimize confusion. The relationship between the ideals in Definition~\ref{def:row-and-column-ideals} and our ideal of interest $I_{\alpha,\beta}$ is as follows.

\begin{observation}
    \label{obs:ideal-sum}
    Let $\alpha,\beta \models_0 n$ be weak compositions. We have the equality of ideals 
    \[ I_{\alpha,\beta} = I^\colsum_\alpha + I^\rowsum_\beta.\]
\end{observation}

Observation~\ref{obs:ideal-sum} comes from comparing the generating sets of the ideals on either side. Row polarization operators preserve row sum ideals while column polarization operators preserve column sum ideals.

\begin{lemma}
    \label{lem:polarization-preserves-ideals}
    The row polarization operator $\rho_{i_1 \to i_0}$ preserves the row sum ideal $I^\rowsum_\beta$. The column polarization operator $\kappa_{j_1 \to j_0}$ preserves the column sum ideal $I^\colsum_{\alpha}$.
\end{lemma}

Recall that a {\em derivation} of an $\FF$-algebra $A$ is a linear map $d: A \to A$ which satisfies the Leibniz rule $d(f \cdot g) = f \cdot d(g) + g \cdot d(f)$ for all $f,g \in A$.

\begin{proof}
    We prove the case of $\rho_{i_1 \to i_0}$; the case of $\kappa_{j_1 \to j_0}$ is similar.  The map $\rho_{i_1 \to i_0}$ is a derivation of $\FF[\xxx_{k \times p}]$. By the Leibniz property, it suffices to show $\rho_{i_1 \to i_0}(g) \in I^\rowsum_\beta$ when $g$ is a {\em generator} of $I^\rowsum_\beta$. Indeed, for $1 \leq i \leq k$ we have
    \[ \rho_{i_1 \to i_0}: x_{i,1} + \cdots + x_{i,p} \mapsto \begin{cases}
        x_{i_0,1} + \cdots + x_{i_0,p} & \text{if $i = i_1$}, \\
        0 & \text{otherwise.}
    \end{cases}\]
    If $1 \leq j \leq p$ and $x_{1,j}^{b_1} \cdots x_{k,j}^{b_k}$ is a monomial with $b_1 + \cdots + b_k > \beta_k$, one has 
    \[
    \rho_{i_1 \to i_0}: x_{1,j}^{b_1} \cdots x_{k,j}^{b_k} \mapsto b_{i_1} \cdot  \frac{x_{i_0}}{x_{i_1}} \cdot x_{1,j}^{b_1} \cdots x_{k,j}^{b_k} \in I^\rowsum_\beta
    \]
    where the image is interpreted as zero when $b_{i_1} = 0$.
\end{proof}

The operator $\rho_{i_1 \to i_0}$ does not preserve the ideal $I^{\colsum}_\alpha$. Similarly, the operator $\kappa_{j_1 \to j_0}$ does not preserve the ideal $I^\rowsum_\beta$. Neither of these polarization operators preserves $I_{\alpha,\beta}$. This lack of preservation and Lemma~\ref{lem:polarization-preserves-ideals} were the motivation for defining  $I^\colsum_\alpha$ and $I^\rowsum_\beta$.

We want to understand the interaction between polarization and  diagonal term orders. If $m > 0$, the iterated polarlization operator \[(\rho_{i_1 \to i_0})^m = \overbrace{\rho_{i_1 \to i_0} \circ \cdots \circ \rho_{i_1 \to i_0}}^{m}: \FF[\xxx_{k \times p}] \longrightarrow \FF[\xxx_{k \times p}]\]
moves $m$ variables from row $i_1$ to row $i_0$, and similarly for $(\kappa_{j_1 \to j_0})^m$. We emulate this combinatorially with the following shift operations on matrices. If $1 \leq i \leq k$ and $1 \leq j \leq p$, we let $E_{i,j} \in \Mat_{k \times p}(\ZZ_{\geq 0})$ be the matrix with a one in position $(i,j)$ and zeros elsewhere.

\begin{defn}
    \label{def:shift-operators}
    Let $A = (a_{i,j}) \in \Mat_{k \times p}(\ZZ_{\geq 0})$ and let $1 \leq i_0,i_1 \leq k$. Let $r := \row(A)_{i_1}$ be the sum of entries in row $i_1$. For $0 \leq m \leq r$, define the {\em row shift} $\shift^\row_{i_1 \to i_0, m}(A) \in \Mat_{k \times p}(\ZZ_{\geq 0})$ by
    \[
    \shift^\row_{i_1 \to i_0, m}(A) = A + \sum_{j=1}^p c_j \cdot (E_{i_0,j} - E_{i_1,j})
    \]
    where the sequence $(c_1,c_2,\dots,c_p)$ is lexicographically maximal so that $\dots$
    \begin{itemize}
        \item we have $0 \leq c_j \leq a_{i_1,j}$ for all $j$, and
        \item we have $c_1 + c_2 + \cdots + c_p = m$.
    \end{itemize}
\end{defn}

Loosely speaking, the operator $\shift^\row_{i_1 \to i_0,m}(A)$ transfers a total value of $m$ from row $i_1$ to row $i_0$ within columns in a lexicographically maximal way. If $k = 4$ and
\begin{footnotesize}
\[    A = \begin{pmatrix}
        0 & 2 & 1 & 3 & 1 \\ 2 & 0 & 1 & 1 & 0  \\ 0 &3 & 1 & 2 & 1\\ 1 & 0 & 2 & 1 & 0 
    \end{pmatrix}\]
\end{footnotesize}
the matrices $\shift^\row_{4 \to 2,m}(A)$ for $m=1,2,3$ are as follows:
\begin{footnotesize}
    \[
    \shift^\row_{4 \to 2,1}(A) = \begin{pmatrix}
        0 & 2 & 1 & 3 & 1 \\ 3 & 0 & 1 & 1 & 0  \\ 0 &3 & 1 & 2 & 1\\ 0 & 0 & 2 & 1 & 0 
    \end{pmatrix}, \quad 
    \shift^\row_{4 \to 2,2}(A) = \begin{pmatrix}
        0 & 2 & 1 & 3 & 1 \\ 3 & 0 & 2 & 1 & 0  \\ 0 &3 & 1 & 2 & 1\\ 0 & 0 & 1 & 1 & 0 
    \end{pmatrix}, \quad 
    \shift^\row_{4 \to 2,3}(A) = \begin{pmatrix}
        0 & 2 & 1 & 3 & 1 \\ 3 & 0 & 3 & 1 & 0  \\ 0 &3 & 1 & 2 & 1\\ 0 & 0 & 0 & 1 & 0 
    \end{pmatrix}.
    \]
\end{footnotesize}
Row shift gives the diagonal leading term of the row polarization of a monomial.

\begin{lemma}
    \label{lem:polarization-on-monomials}
    Let $\prec$ be a diagonal term order and suppose $1 \leq i_0 < i_1 \leq k$. If $0 \leq m \leq \row(A)_{i_1}$, the $\prec$-leading term of $(\rho_{i_0 \to i_1})^m(\xxx^A)$ is
    \[ \initial_\prec (\rho_{i_0 \to i_1})^m(\xxx^A) = \xxx^B \quad \quad \text{where } B = \shift^\row_{i_0 \to i_1, m}(A).\]
\end{lemma}

\begin{proof}
    The monomials $\xxx^C$ appearing in $(\rho_{i_0 \to i_1})^m(\xxx^A)$ are those for which
    \begin{equation}
     C = A + \sum_{j=1}^p c_j \cdot (E_{i_0,j} - E_{i_1,j})
    \end{equation}
    for some tuple $(c_1,\dots,c_p)$ for which $0 \leq c_j \leq a_{i_1,j}$ and $c_1 + \cdots + c_p = m$. The diagonal degrees of $\xxx^C$ and $\xxx^A$ are related by
    \begin{equation}
        \ddeg(\xxx^C) = \ddeg(\xxx^A) + \sum_{j=1}^p c_j \cdot (\eee_{i_0+j-1} - \eee_{i_1 + j-1}).
    \end{equation}
    Since $\prec$ is a diagonal term order, we have
    \begin{equation}
        \xxx^{C} \prec \xxx^{C'} \text{ whenever } \ddeg(\xxx^C) \leq_{lex} \ddeg(\xxx^{C'}).
    \end{equation}
    These observations together with the assumption $i_0 < i_1$ complete the proof.
\end{proof}

With an eye toward the proof of Lemma~\ref{lem:polarization-on-monomials} and diagonal degree comparison, we introduce a version of the shift operator on one-dimensional arrays.

\begin{defn}
    \label{def:split-operator}
    Let $(d_1,\dots,d_n) \in (\ZZ_{\geq 0})^n$ be a sequence of nonnegative integers with $d_1 + \cdots + d_n = d$. Let $1 \leq s \leq t \leq n$ and assume that $d_1 = d_2 = \cdots = d_t = 0$. For any $0 \leq m \leq d$, an {\em $s$-step leftward shift of magnitude $m$} is a sequence of the form
    \begin{equation}
        (d'_1,\dots,d'_n) = (d_1,\dots,d_n) + \sum_{j=t+1}^{n} c_j \cdot (\eee_{j-s} - \eee_{j})
    \end{equation}
    for some integers $0 \leq c_j \leq d_j$ such that $c_{t+1} + c_{t+2} + \cdots + c_n = m$. The {\em leftward split} 
    \begin{equation}
        \spl^\lft_{s,m}(d_1,\dots,d_n)
    \end{equation}
    is the leftward shift so obtained by choosing $(c_{t+1},c_{t+2}, \dots, c_n)$ lexicographically maximal.
\end{defn}

For example, if $(d_1,\dots,d_n) = (0,0,0,2,1,0,3)$ we have the leftward splits
\[ \spl^\lft_{2,1}(d_1,\dots,d_n) = (0,1,0,1,1,0,3), \quad
\spl^\lft_{2,2}(d_1,\dots,d_n) = (0,2,0,0,1,0,3),\]
\[\spl^\lft_{2,3}(d_1,\dots,d_n) = (0,2,1,0,0,0,3), \quad
\spl^\lft_{2,4}(d_1,\dots,d_n) = (0,2,1,0,1,0,2), \]
\[ 
\spl^\lft_{2,4}(d_1,\dots,d_n) = (0,2,1,0,2,0,1), \quad
\spl^\lft_{2,4}(d_1,\dots,d_n) = (0,2,1,0,3,0,0).\]
A $2$-step leftward shift of $(d_1,\dots,d_n)$ of magnitude 3 different from $\spl^\lft_{2,3}(d_1,\dots,d_n) = (0,2,1,0,0,0,3)$ is \[(0,1,0,1,3,0,1) = (d_1,\dots,d_n) + (\eee_2 - \eee_4) + 2 \cdot (\eee_5 - \eee_7). \]
We need a simple result relating leftward shifts and lexicographical order.

\begin{lemma}
    \label{lem:split-lex}
    Suppose $(d_1,\dots,d_n)$ and $(e_1,\dots,e_n)$ are two sequences of positive integers with the same sum $d$ which satisfy the lexicographical comparison 
    \[ (e_1,\dots,e_n) \leq_{lex} (d_1,\dots,d_n). \]
    Assume that $1 \leq s \leq t$ is such that $d_1 = \cdots = d_t = 0$ and let $1 \leq m \leq d$. Assume further that 
        there exists an $s$-step leftward shift $(e'_1,\dots,e'_n)$ of $(e_1,\dots,e_n)$ of magnitude $m$ so that \[\spl^\lft_{s,m}(d_1,\dots,d_n) \leq_{lex} (e'_1,\dots,e'_n).\]
     Then $(e_1,\dots,e_n) = (d_1,\dots,d_n)$ and $(e'_1,\dots,e'_n) = \spl^\lft_{s,m}(d_1,\dots,d_n)$.   
\end{lemma}

\begin{proof}
    It is clear that $(e'_1,\dots,e'_n) \leq_{lex}\spl^\lft_{s,m}(e_1,\dots,e_n)$, so one has 
    \begin{equation}
    \label{eq:split-lex}
    \spl^\lft_{s,m}(d_1,\dots,d_n) \leq_{lex} (e'_1,\dots,e'_n) \leq_{lex} \spl^\lft_{s,m}(e_1,\dots,e_n).
    \end{equation}
    The inequality $\spl^\lft_{s,m}(d_1,\dots,d_n) \leq _{lex} \spl^\lft_{s,m}(e_1,\dots,e_n)$ the hypothesis $(e_1,\dots,e_n) \leq_{lex} (d_1,\dots,d_n)$ and \eqref{eq:split-lex} combine to show $(e_1,\dots,e_n) = (d_1,\dots,d_n)$, so that $(e'_1,\dots,e'_n) =  \spl^\lft_{s,m}(e_1,\dots,e_n)$ by \eqref{eq:split-lex}.
\end{proof}

Polarization operators have especially nice properties when they act on polynomials with `zigzag' leading terms. Recall that the {\em support} of a matrix $A = (a_{i,j})$ is the set of positions $(i,j)$ for which $a_{i,j} \neq 0$.

\begin{defn}
    \label{def:zigzag}
    A matrix $A$ is {\em zigzag} if its support is a zigzag. A monomial $\xxx^A$ is {\em zigzag} if its exponent matrix $A$ is zigzag.
\end{defn}

An example $4 \times 4$ zigzag matrix is 
\begin{footnotesize}
\[ A = \begin{pmatrix} 0 & 0 & 0 & 0 \\  1 & 2 & 1 & 0 \\ 0 & 0 & 2 & 1 \\ 0 & 0 & 0 & 1 \end{pmatrix}.\]
\end{footnotesize}
The corresponding zigzag monomial is $\xxx^A = x_{2,1} \cdot x_{2,2}^2 \cdot x_{2,3} \cdot x_{3,3}^2 \cdot x_{3,4} \cdot x_{4,4}$. If $A$ is a zigzag matrix with at least two nonzero rows, the {\em row merge} $\merge^\row(A)$ is obtained by adding the first nonzero row to the second nonzero row of $A$ and replacing the first nonzero row with a row of zeros. In our example we have
\begin{footnotesize}
\[ \merge^\row(A) = \begin{pmatrix} 0 & 0 & 0 & 0 \\ 0 & 0 & 0 & 0 \\ 1 & 2 & 3 & 1 \\ 0 & 0 & 0 & 1 \end{pmatrix}.\]
\end{footnotesize}
The following observation records properties of the $\shift$ and $\merge$ operations on zigzag matrices.

\begin{observation}
    \label{obs:shift-merge-inverse} 
    Let $A$ be a zigzag matrix whose first nonzero row has index $i_1$. Let $r = \row(A)_{i_1}$ be the sum of this row and choose $i_0 < i_1$. 
    \begin{enumerate}
        \item For any $0 \leq m \leq r$, the matrix $\shift^\row_{i_1 \to i_0,m}(A)$ is zigzag. Furthermore, if $B = \shift^\row_{i_1 \to i_0,m}(A)$ we have
        \[ \ddeg(\xxx^B) = \spl^\lft_{i_1-i_0,m}(\ddeg(\xxx^A)).\]
        \item Suppose $A$ has at least two nonzero rows. Then $\merge^\row(A)$ is zigzag.
        \item Suppose $A$ has at least two nonzero rows and let $i_2$ be the minimal index $> i_1$ of a nonzero row of $A$. We have 
        \[\shift^\row_{i_2 \to i_1,r} \circ \merge^\row(A) = A.\] 
    \end{enumerate}
\end{observation}

Let $\prec$ be a diagonal term order on $\FF[\xxx_{k \times p}]$. In order to prove that $\{ \mb(A) \,:\, A \in \CCC_{\alpha,\beta}\}$ contains the $\prec$-standard monomial basis of $R_{\alpha,\beta} = \FF[\xxx_{k \times p}]/I_{\alpha,\beta}$, we produce polynomials $f \in I_{\alpha,\beta}$ with interesting zigzag $\prec$-leading terms $\initial_\prec(f)$ which reflect the combinatorics of Theorem~\ref{thm:zigzag-characterization} and Proposition~\ref{prop:monomial-is-subtingency}. Lemma~\ref{lem:polarization-preserves-ideals} motivates the strategy of inductively constructing these polynomials using iterated polarization operators. The following result relates $\initial_\prec(f)$ and $\initial_\prec(\rho_{i_1 \to i_0})^m(f)$ when $1 \leq i_0 < i_1 \leq m$. We assume that $\initial_\prec(f)$ is zigzag and that $f$ does not involve the variable $x_{i,j}$ when $i <i_1$.

\begin{lemma}
    \label{lem:polarization-leading}
    Let $\prec$ be a diagonal term order and let $f \in \FF[\xxx_{k \times p}]$ be a nonzero row-homogeneous polynomial. Assume that $\initial_\prec(f) = \xxx^A$ where $A$ is a zigzag matrix. Let $1 \leq i_1 \leq k$ be the minimal index of a nonzero row of $A$ and let $r = \row(A)_{i_1}$ be the sum of this row. 
    
    Let $i_0 < i_1$ and $0 < m < r$. We have 
    \[ \initial_\prec (\rho_{i_1 \to i_0})^m(f) = \initial_\prec (\overbrace{\rho_{i_1 \to i_0} \circ \cdots \circ \rho_{i_1 \to i_0}}^m)(f) = \xxx^B \quad \text{where $B := \shift^\row_{i_1 \to i_0,m}(A)$.}\]
\end{lemma}

The assumption that $1 \leq i_1 \leq k$  is the {\em minimal} index of a nonzero row of $A$ together with the row-homogeneity of $f$ implies that $f$ does not involve variables $x_{i,j}$ for which $i < i_1$.
Lemma~\ref{lem:polarization-leading} would be false without this minimality assumption. For example, if 
\[ f = x_{1,2} \cdot x_{3,3} + x_{1,3} \cdot x_{3,1}  \]
then $f$ is row-homogeneous and  $\initial_\prec(f) = x_{1,2} \cdot x_{3,3}$, which is a zigzag monomial. However, we have 
\[ \rho_{3 \to 1} (f) = x_{1,2} \cdot x_{1,3} + x_{1,3} \cdot x_{1,1}\]
so that $\initial_\prec (\rho_{3 \to 1}(f)) = x_{1,3} \cdot x_{1,1}$ instead of $x_{1,2} \cdot x_{1,3}$.

\begin{proof}
    Lemma~\ref{lem:polarization-on-monomials} implies that $\xxx^B$ is the $\prec$-leading monomial of $(\rho_{i_1 \to i_0})^m(\xxx^A)$.  Let $\xxx^{A'} \neq \xxx^A$ be a distinct monomial appearing in $f$. Since $\prec$ is a diagonal term order we have
    \begin{equation}
        \label{eq:polarization-one}
        \xxx^{A'} \prec \xxx^A \text{ and therefore } \ddeg(\xxx^{A'}) \leq_{lex} \ddeg(\xxx^A).
    \end{equation}
    Let $\xxx^{B'} = \initial_\prec (\rho_{i_1 \to i_0})^m(\xxx^{A'})$ be the $\prec$-largest monomial which appears in $(\rho_{i_1 \to i_0})^m(\xxx^{A'})$. By the definition of $B$ and Lemma~\ref{lem:polarization-on-monomials} we have
    \begin{equation}
        \label{eq:polarization-two}
        B = \shift^\row_{i_1 \to i_0,m}(A) \text{ and }
        B' = \shift^\row_{i_1 \to i_0,m}(A').
    \end{equation}
    Assume that 
    \begin{equation}
        \label{eq:polarization-three}
        \xxx^{B} \preceq \xxx^{B'} \text{ and therefore } \ddeg(\xxx^B) \leq_{lex} \ddeg(\xxx^{B'}).
    \end{equation}
    We derive the contradiction $\xxx^A \preceq \xxx^{A'}$, thus proving the lemma.

    If $\xxx^C$ is any monomial in $\FF[\xxx_{k \times p}]$ and $C = (c_{i,j})$, we use the row $i_1$ to factor 
    \begin{equation}
        \label{eq:polarization-five}
        \xxx^C = \xxx^C_{< \, i_1} \cdot \xxx^C_{ = \, i_1} \cdot \xxx^C_{> \, i_1}
    \end{equation}
    where
    \begin{equation}
        \label{eq:polarization-six}
        \xxx^C_{< \, i_1} := \prod_{i=1}^{i_1-1}\prod_{j=1}^p x_{i,j}^{c_{i_,j}}, \quad 
        \xxx^C_{= \, i_1} := \prod_{j=1}^p x_{i_1,j}^{c_{i_1,j}}, \quad 
        \xxx^C_{> \, i_1} := \prod_{i = i_1 + 1}^k \prod_{j=1}^p x_{i,j}^{c_{i,j}}.
    \end{equation}
    The proof is a careful analysis of these factorizations for $C = A,B,A',B'$.

    Our minimality assumption on $i_1$ and \eqref{eq:polarization-two} imply
    \begin{equation}
        \label{eq:polarization-seven}
        \xxx^A_{< \, i_1} = \xxx^{A'}_{< \, i_1} = 1, \quad  \xxx^{A}_{> \, i_1} = \xxx^B_{> \, i_1}, \quad \xxx^{A'}_{> \, i_1} = \xxx^{B'}_{> \, i_1}.
    \end{equation}
    By \eqref{eq:polarization-two}, we see that
    \begin{multline}
        \label{eq:polarization-eight}
        \ddeg(\xxx^{B}), \, \ddeg(\xxx^{B'}) \text{ are (respectively) $(i_1-i_0)$-step leftward shifts} \\  \text{of $\ddeg(\xxx^A), \, \ddeg(\xxx^{A'})$ of magnitude $m$.}
    \end{multline}
    Since $A$ is zigzag, \eqref{eq:polarization-two} and Observation~\ref{obs:shift-merge-inverse} (1) yield
    \begin{equation}
        \label{eq:polarization-nine}
        \ddeg(\xxx^B) = \spl^\lft_{i_1-i_0,m} (\ddeg(\xxx^A)).
    \end{equation}
    The conditions \eqref{eq:polarization-one}, \eqref{eq:polarization-three}, \eqref{eq:polarization-eight}, and  \eqref{eq:polarization-nine} are exactly the hypotheses of Lemma~\ref{lem:split-lex}. Applying Lemma~\ref{lem:split-lex} gives 
    \begin{equation}
        \label{eq:polarization-ten}
        \ddeg(\xxx^A) = \ddeg(\xxx^{A'}) \quad \text{and} \quad \ddeg(\xxx^B) = \ddeg(\xxx^{B'}) = \spl^\lft_{i_1-i_0,m}(\ddeg(\xxx^A)).
    \end{equation}
    Since the operation $\shift^{\row}_{i_1 \to i_0,m}(-)$ moves entries from row $i_1$ to row $i_0$ within columns while leaving other rows unchanged, the conditions \eqref{eq:polarization-two}, \eqref{eq:polarization-seven}, and \eqref{eq:polarization-ten} force
    \begin{equation}
        \label{eq:polarization-eleven}
        \xxx^B_{< \, i_1} = \xxx^{B'}_{< \, i_1} \quad \text{and} \quad \frac{\xxx^A_{= \, i_1}}{\xxx^{B}_{= \, i_1}} = \frac{\xxx^{A'}_{= \, i_1}}{\xxx^{B'}_{= \, i_1}}.
    \end{equation}
    The ratio $\frac{\xxx^A_{= \, i_1}}{\xxx^{B}_{= \, i_1}} = \frac{\xxx^{A'}_{= \, i_1}}{\xxx^{B'}_{= \, i_1}}$ is a monomial (i.e. has no negative exponents) by \eqref{eq:polarization-two}.  Since $\xxx^B \preceq \xxx^{B'}$ by \eqref{eq:polarization-three} and $\xxx^B_{< \, i_1} = \xxx^{B'}_{< \, i_1}$ by \eqref{eq:polarization-eleven}, we obtain
    \begin{equation}
        \label{eq:polarization-twelve}
        \xxx^B_{= \, i_1} \cdot \xxx^B_{> \, i_1} \preceq \xxx^{B'}_{= \, i_1} \cdot \xxx^{B'}_{> \, i_1}
    \end{equation}
    because $\prec$ is a term order. The equality of ratios in \eqref{eq:polarization-eleven} the status of $\prec$ as a term order give
    \begin{equation}
        \label{eq:polarization-thirteen}
        \xxx^A_{= \, i_1} \cdot \xxx^B_{> \, i_1} \preceq \xxx^{A'}_{= \, i_1} \cdot \xxx^{B'}_{> \, i_1}.
    \end{equation}
    Combining \eqref{eq:polarization-seven} and \eqref{eq:polarization-thirteen}, we arrive at 
    \begin{multline}
        \label{eq:polarization-fifteen}
        \xxx^A = \xxx^A_{< \, i_1} \cdot \xxx^A_{= \,i_1} \cdot \xxx^{A}_{> \, i_1} = 
        \xxx^A_{= \,i_1} \cdot \xxx^{A}_{> \, i_1} = 
        \xxx^A_{= \,i_1} \cdot \xxx^{B}_{> \, i_1} \preceq  \\
        \xxx^{A'}_{= \,i_1} \cdot \xxx^{B'}_{> \, i_1} = 
        \xxx^{A'}_{= \,i_1} \cdot \xxx^{A'}_{> \, i_1} = 
        \xxx^{A'}_{< \, i_1} \cdot \xxx^{A'}_{= \,i_1} \cdot \xxx^{A'}_{> \, i_1} = \xxx^{A'}
    \end{multline}
    which contradicts the condition $\xxx^{A'} \prec \xxx^A$ of \eqref{eq:polarization-one}.
\end{proof}

Lemma~\ref{lem:polarization-leading} applies to a very special class of polynomials. Fortunately, this is exactly the class of polynomials we will need. Lemma~\ref{lem:polarization-leading} gives the following two-dimensional enhancement of Corollary~\ref{cor:deduced-polynomials}.

\begin{lemma}
    \label{lem:zigzag-leading}
    Let $\prec$ be a diagonal term order on $\FF[\xxx_{k \times p}]$.
    Let $\beta \models_0 n$ be a weak composition with $\ell(\beta) = p$ and consider the row sum ideal $I^\rowsum_\beta \subseteq \FF[\xxx_{k \times p}]$. Let $A = (a_{i,j})$ be a $k \times p$ zigzag matrix. Suppose there exists $1 \leq j \leq p$ so that 
    \begin{equation}
    \label{eq:zigzag-violation}
        \sum_{q=1}^j \col(A)_q > \sum_{q=1}^{j-1} (\beta_q - \col(A)_q).
    \end{equation}
    There exists a nonzero row-homogeneous polynomial $f \in I^\rowsum_\beta$ such that $\initial_\prec(f) = \xxx^A$. Furthermore, if the smallest row in $A$ has index $i$, we may assume that $f$ involves no variables with row index $< i$.
\end{lemma}

The similarity between inequality \eqref{eq:zigzag-violation}, the inequality of Corollary~\ref{cor:deduced-polynomials}, and the negation of the second inequality of Proposition~\ref{prop:shadow-characterization} is no coincidence. The last sentence of Lemma~\ref{lem:zigzag-leading} will not see further use, but is necessary for the following inductive argument.

\begin{proof}
    We apply upward induction on the number of nonzero rows of the matrix $A$. First suppose that $A$ has a single nonzero row with index $i$. We have a natural embedding
    \begin{equation}
        \iota: \FF[x_{i,1}, \dots, x_{i,p}] \hookrightarrow \FF[\xxx_{k \times p}].
    \end{equation}
    If $\ddd = (\beta_1 +1,\dots,\beta_p + 1)$ and $I_\ddd \subseteq \FF[x_{i,1},\dots,x_{i,p}]$ is the ideal of 
    Definition~\ref{def:one-row-quotient}, one easily sees that $\iota(I_\ddd) \subseteq I^\rowsum_\beta$ by a check on generators. Since $\prec$ resctricts to lexicographical term order on the monomials of $\FF[x_{i,1}, \dots, x_{i,p}]$ (Observation~\ref{obs:diagonal-restriction}), we may appeal to  Corollary~\ref{cor:deduced-polynomials} to complete the proof in this case.

    Now assume that at least two rows of $A$ are nonzero. Let $i_0 < i_1$ be the smallest indices of nonzero rows in $A$. The merged matrix $\bar{A} := \merge^\row(A)$ has one fewer nonzero row than $A$, and the smallest nonzero row index in $\bar{A}$ is $i_1$. The inequality \eqref{eq:zigzag-violation} is preserved under row merges (and row splits), so induction furnishes a row-homogeneous polynomial $\bar{f} \in I^\rowsum_\beta$ which does not involve variables with row index $< i_1$ such that $\initial_\prec(\bar{f}) = \xxx^{\bar{A}}$. If $m = \row(A)_{i_0}$, Observation~\ref{obs:shift-merge-inverse} implies
    \[ \shift^\row_{i_1 \to i_0,m}(\bar{A}) = (\shift^\row_{i_1 \to i_0,m} \circ \merge^\row)(A) = A.\]
    Lemma~\ref{lem:polarization-leading} implies that $f := (\rho_{i_1 \to i_0})^m(\bar{f})$ has $\prec$-leading monomial $\xxx^A$. Since $\bar{f}$ does not involve variables of row index $< i_1$, we see that $f$ does not involve variables of row index $< i_0$ and the proof is complete. 
\end{proof}

The `transposed' version of Lemma~\ref{lem:zigzag-leading} will be be cited in the key spanning result of Lemma~\ref{lem:spanning-lemma}, so we state it explicitly. 

\begin{lemma}
    \label{lem:zigzag-leading-transpose}
    Let $\prec$ be a diagonal term order on $\FF[\xxx_{k \times p}]$.
    Let $\alpha \models_0 n$ be a weak composition with $\ell(\beta) = k$ and consider the row sum ideal $I^\colsum_\alpha \subseteq \FF[\xxx_{k \times p}]$. Let $A = (a_{i,j})$ be a $k \times p$ zigzag matrix. Suppose there exists $1 \leq i \leq k$ so that 
    \begin{equation}
    \label{eq:zigzag-violation}
        \sum_{q=1}^i \row(A)_q > \sum_{q=1}^{i-1} (\alpha_q - \row(A)_q).
    \end{equation}
    There exists a nonzero column-homogeneous polynomial $f \in I^\colsum_\alpha$ such that $\initial_\prec(f) = \xxx^A$. Furthermore, if the smallest column in $A$ has index $j$, we may assume that $f$ involves no variables with column index $< j$.
\end{lemma}

\begin{proof}
    Apply Lemma~\ref{lem:zigzag-leading} and the $\FF$-algebra isomorphism $\FF[\xxx_{k \times p}]/I^\colsum_\alpha \cong \FF[\xxx_{p \times k}]/I^\rowsum_\alpha$ induced by $x_{i,j} \leftrightarrow x_{j,i}$.
\end{proof}

We do not know explicit formulas for polynomials $f$ guaranteed by Lemmas~\ref{lem:zigzag-leading} and \ref{lem:zigzag-leading-transpose}. This lack of explicit knowledge is inherited from our lack of explicit knowledge regarding Corollary~\ref{cor:deduced-polynomials}.

\subsection{The spanning lemma} The goal of this subsection is to show that $\{ \mb(A) \,:\, A \in \CCC_{\alpha,\beta}\}$ contains any diagonal standard monomial basis $R_{\alpha,\beta}$. The following lemma will simplify the rings under consideration.

\begin{lemma}
    \label{lem:ideal-maps}
    Let $B \in \SSS_{\alpha,\beta}$ be an $\alpha,\beta$-subtingency table. Define compositions $\bar{\alpha},\bar{\beta}$ by 
    \[ \bar{\alpha} := \alpha - \row(B) \quad \text{and} \quad \bar{\beta} := \beta - \col(B)\]
    where the subtraction is componentwise. We have the containment
    \[ \xxx^B \cdot I_{\bar{\alpha},\bar{\beta}} \subseteq I_{\alpha,\beta}\]
    of ideals in $\FF[\xxx_{k \times p}]$.
\end{lemma}

\begin{proof}
    This straightforward check on the generators of $I_{\bar{\alpha},\bar{\beta}}$ is left to the reader.
\end{proof}

The containment of Lemma~\ref{lem:ideal-maps} is in fact an equality, but we will not require this stronger statement. We are in a position to prove our spanning result.

\begin{lemma}
    \label{lem:spanning-lemma}
    Let $\prec$ be a diagonal term order on $\FF[\xxx_{k \times p}]$ and let $\alpha,\beta \models_0 n$ satisfy $\ell(\alpha) = k$ and $\ell(\beta) = p$. Let $\BBB$ be the standard monomial basis of $R_{\alpha,\beta} = \FF[\xxx_{k \times p}]/I_{\alpha,\beta}$ with respect to $\prec$. Then
    \[ \BBB \subseteq \{ \mb(A) \,:\, A \in \CCC_{\alpha,\beta} \}.\]
\end{lemma}

Recall that the set $\SSS_{\alpha,\beta}$ of $\alpha,\beta$-subtingency tables consists of $\ZZ_{\geq 0}$-matrices $C = (c_{i,j})$ for which we have the componentwise inequalities $\row(C) \leq \alpha$ and $\col(C) \leq \beta$.

\begin{proof}
    Let $C = (c_{i,j}) \in \Mat_{k \times p}(\ZZ_{\geq 0})$ be a matrix such that $\xxx^C \notin \{ \mb(A) \,:\, A \in \CCC_{\alpha,\beta} \}$. It suffices to establish a congruence of the form
    \begin{equation}
    \label{eq:goal-congruence}
        \xxx^C \equiv \text{(an $\FF$-linear combination of monomials $\xxx^{C'} \prec \xxx^C)$} \mod I_{\alpha,\beta}.
    \end{equation}
    If $C$ is not an $\alpha,\beta$-subtingency table, we have $\xxx^C \in I_{\alpha,\beta}$ and \eqref{eq:goal-congruence} is obvious. We may therefore assume that $C \in \SSS_{\alpha,\beta}$ is an $\alpha,\beta$-subtingency table.

    Use $C$ to construct a labeled grid of balls as in the matrix-ball construction. As in Proposition~\ref{prop:shadow-characterization}, let $x_i$ be the number of northern balls in row $i$ and let $y_j$ be the number of western balls in column $j$. Since $C \in \SSS_{\alpha,\beta}$ and yet $\xxx^C \notin \{ \mb(A) \,:\, A \in \CCC_{\alpha,\beta} \}$, Proposition~\ref{prop:shadow-characterization} implies that at least one of the following is true.
    \begin{enumerate}
        \item There exists $1 \leq i \leq k$ such that 
        \[ \sum_{q=1}^i x_q > \sum_{q=1}^{i-1} (\alpha_q - \row(C)_q).\]
        \item There exists $1 \leq j \leq p$ such that 
        \[ \sum_{q=1}^j y_q > \sum_{q=1}^{j-1} (\beta_q - \col(C)_q).\]
    \end{enumerate}
    Let us assume that (2) holds; the case of (1) is similar. We perform an argument similar to that used to prove Theorem~\ref{thm:zigzag-characterization}.

    Choose $1 \leq j \leq k$ minimal such that (1) holds and define $\ell := \sum_{q=1}^{j} y_q$.  By the minimality of $j$, there exists a northern ball labeled $\ell_1 := \ell$ in column $j_1 := j$. Let $(i_1,j_1)$ be the position of such a ball and let $\ell_2$ be the minimal label of a ball in position $(i_1,j_1)$. If $\ell_2 > 1$, there exists a position $(i_2,j_2)$ with $i_2 \leq i_1$ and $j_2 \leq j_1$ which contains a ball labeled $\ell_2 - 1$. Let $\ell_3$ be the minimal label of a ball in position $(i_2,j_2)$. If $\ell_3 > 1$, there exists a ball labeled $\ell_3 -1$ in some position $(i_3,j_3)$ with $i_3 \leq i_2$ and $j_3 \leq j_2$. Continuing in this fashion, we obtain a zigzag $Z = \{ (i_1,j_1), \dots,(i_m,j_m) \}$ of positions which contain balls of labels $1,2,\dots,\ell$.

    An example may clarify these constructions. Suppose $k = 4, p = 6, \beta = (5,4,3,3,4,3)$, and
    \begin{footnotesize}
        \[ C = \begin{pmatrix} 0 & 0 & 0 & 0 & 0 & 0 \\ 0 & 1 & 2 & 0 & 1 & 0 \\ 0 & 0 & 0 & 2 & 1 & 2 \\ 0 & 3 & 0 & 1 & 1 & 1 \end{pmatrix}.\]
    \end{footnotesize}
    Observe the componentwise inequality $\col(C) = (0,4,2,3,3,3) \leq (5,4,3,3,4,3) = \beta$. Applying the matrix-ball construction to $C$ yields the following figure with western balls decorated with $W$.
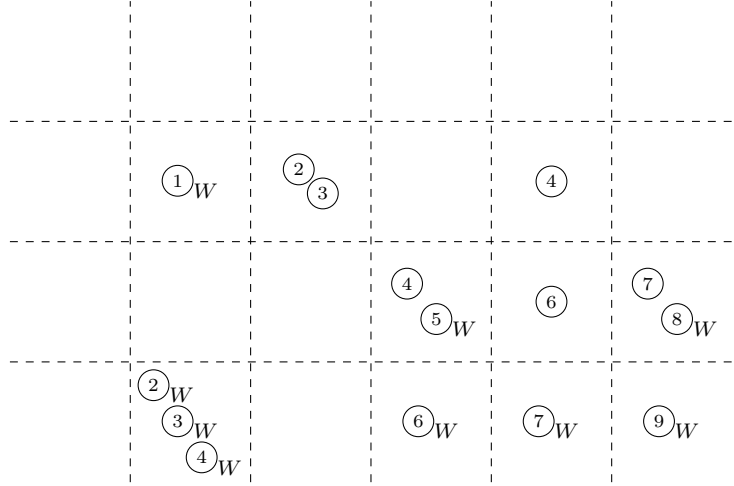
    \begin{center}
    \begin{tikzpicture}[scale = 1.6]
            \draw[dashed] (-1,2) -- (5,2);
            \draw[dashed] (-1,1) -- (5,1);
            \draw[dashed] (-1,0) -- (5,0);

            \draw[dashed] (0,-1) -- (0,3);
            \draw[dashed] (1,-1) -- (1,3);
            \draw[dashed] (2,-1) -- (2,3);
            \draw[dashed] (3,-1) -- (3,3);
            \draw[dashed] (4,-1) -- (4,3);

            \node at (0.5,1.5) {${\begin{tiny}\circled{1}_W\end{tiny}}$};
            \node at (0.5,-0.5) {$\begin{tiny}\circled{3}_W\end{tiny}$};
            \node at (0.7,-0.8) {$\begin{tiny}\circled{4}_W\end{tiny}$};
            \node at (0.3,-0.2) {$\begin{tiny}\circled{2}_W\end{tiny}$};

            \node at (1.4,1.6) {${ \begin{tiny}\circled{2}\end{tiny}}$};
            \node at (1.6,1.4) {${ \begin{tiny}\circled{3}\end{tiny}}$};

            \node at (3.5,1.5) {$\begin{tiny}\circled{4}\end{tiny}$};

            \node at (2.3,0.65) {${\begin{tiny}\circled{4}\end{tiny}}$};
            \node at (2.65,0.35) {${\begin{tiny}\circled{5}_W\end{tiny}}$};

            \node at (2.5,-0.5) {$\begin{tiny}\circled{6}_W\end{tiny}$};
            \node at (3.5,0.5) {${ \begin{tiny}\circled{6}\end{tiny}}$};
            \node at (3.5,-0.5) {${\begin{tiny}\circled{7}_W\end{tiny}}$};

            \node at (4.3,0.65) {$\begin{tiny}\circled{7}\end{tiny}$};
            \node at (4.65,0.35) {$\begin{tiny}\circled{8}_W\end{tiny}$};

            \node at (4.5,-0.5) {$\begin{tiny}\circled{9}_W\end{tiny}$};
    \end{tikzpicture}
\end{center}
    By examining the columns occupied by western balls, we see that $(y_1,\dots,y_6) = (0,4,0,2,1,2)$.
    We have the componentwise difference $\beta - \col(C) = (5,0,1,0,1,0)$. The minimal value of $j$ for which the condition (2) is satisfied is $j = 5$. We have $\ell = y_1 + \cdots + y_j = y_1 + \cdots + y_5 = 0 + 4 + 0 + 2 + 1 = 7$. There is a unique ball labeled $\ell = 7$ in column $j = 5$; it is in matrix position $(i_1,j_1) = (4,5)$. One possible zigzag $Z = \{(2,2),(2,3),(3,4),(3,5),(4,5)\}$ constructed as in the above paragraph is shown in red below. (There is  another possible zigzag involving the $\begin{footnotesize}\circled{6}\end{footnotesize}$ in row 3.)
    
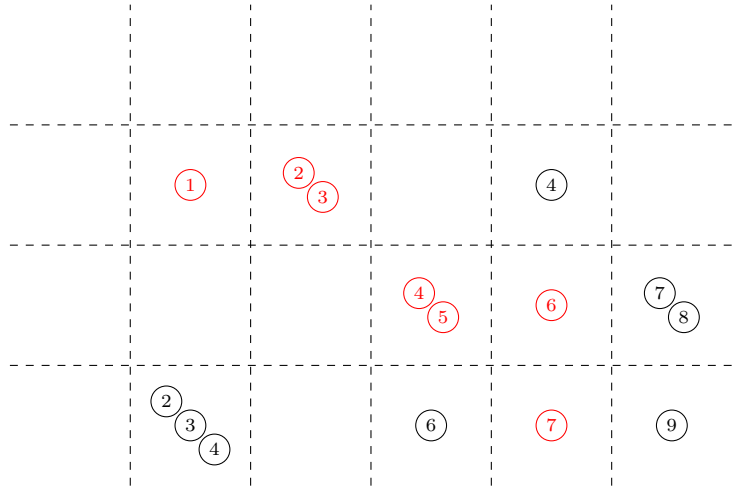
\begin{center}
    \begin{tikzpicture}[scale = 1.6]
            \draw[dashed] (-1,2) -- (5,2);
            \draw[dashed] (-1,1) -- (5,1);
            \draw[dashed] (-1,0) -- (5,0);

            \draw[dashed] (0,-1) -- (0,3);
            \draw[dashed] (1,-1) -- (1,3);
            \draw[dashed] (2,-1) -- (2,3);
            \draw[dashed] (3,-1) -- (3,3);
            \draw[dashed] (4,-1) -- (4,3);

            \node at (0.5,1.5) {${\color{red} \begin{tiny}\circled{1}\end{tiny}}$};
            \node at (0.5,-0.5) {$\begin{tiny}\circled{3}\end{tiny}$};
            \node at (0.7,-0.7) {$\begin{tiny}\circled{4}\end{tiny}$};
            \node at (0.3,-0.3) {$\begin{tiny}\circled{2}\end{tiny}$};

            \node at (1.4,1.6) {${\color{red} \begin{tiny}\circled{2}\end{tiny}}$};
            \node at (1.6,1.4) {${\color{red} \begin{tiny}\circled{3}\end{tiny}}$};

            \node at (3.5,1.5) {$\begin{tiny}\circled{4}\end{tiny}$};

            \node at (2.4,0.6) {${\color{red} \begin{tiny}\circled{4}\end{tiny}}$};
            \node at (2.6,0.4) {${\color{red} \begin{tiny}\circled{5}\end{tiny}}$};

            \node at (2.5,-0.5) {$\begin{tiny}\circled{6}\end{tiny}$};
            \node at (3.5,0.5) {${\color{red} \begin{tiny}\circled{6}\end{tiny}}$};
            \node at (3.5,-0.5) {${\color{red} \begin{tiny}\circled{7}\end{tiny}}$};

            \node at (4.4,0.6) {$\begin{tiny}\circled{7}\end{tiny}$};
            \node at (4.6,0.4) {$\begin{tiny}\circled{8}\end{tiny}$};

            \node at (4.5,-0.5) {$\begin{tiny}\circled{9}\end{tiny}$};
    \end{tikzpicture}
\end{center}

    Define a new matrix $\overline{C} = (\bar{c}_{i,j}) \in \Mat_{k \times p}(\ZZ_{\geq 0})$ by the rule 
    \begin{equation}
        \bar{c}_{i,j} := \begin{cases}
            c_{i,j} & \text{if $(i,j) \in Z$,} \\
            0 & \text{otherwise.}
        \end{cases}
    \end{equation}
    Define $B := C - \overline{C}$ by matrix subtraction and define 
    \begin{equation} \bar{\alpha} := \alpha - \row(B) \quad \text{and} \quad \bar{\beta} := \beta -\col(B)\end{equation}
    where the subtractions are componentwise. In our example, we have
        \[ \begin{footnotesize} \overline{C} = \begin{pmatrix} 0 &  0 & 0 & 0 & 0 & 0 \\ 0 & 1 & 2 & 0 & 0 & 0 \\ 0 & 0 & 0 & 2 & 1 & 0 \\ 0 & 0 & 0  & 0 & 1 & 0 \end{pmatrix} \end{footnotesize} \quad \text{and} \quad
        \begin{footnotesize}
            B = C - \overline{C} = \begin{pmatrix} 0 & 0 & 0 & 0 & 0 & 0 \\ 0 & 0 & 0 & 0 & 1 & 0 \\ 0 & 0 & 0 & 0 & 0 & 2 \\ 0 & 3 & 0 & 1 & 0 & 1 \end{pmatrix}
        \end{footnotesize}\]
        so that $\overline{C}$ corresponds to the red balls and $B$ corresponds to the black balls. We have 
        \[ \bar{\beta} = \beta - \col(B) = (5,4,3,3,4,3) - (0,3,0,1,1,3) = (5,1,3,2,3,0).\]

    The matrix $\overline{C}$ is zigzag by construction. We have 
    \begin{equation}
        \sum_{q=1}^j \col(\overline{C})_q = \ell = \sum_{q=1}^{j} y_q > \sum_{q=1}^{j-1} (\beta_q - \col(C)_i) = \sum_{q=1}^{j-1} (\bar{\beta}_q - \col(\overline{C}))_q.
    \end{equation}
    Lemma~\ref{lem:zigzag-leading} guarantees the existence of a polynomial $\bar{f} \in I^\rowsum_{\bar{\beta}}$ such that  $\initial_\prec(\bar{f}) = \xxx^{\overline{C}}$. A glance at Observation~\ref{obs:ideal-sum} shows that $\bar{f} \in I_{\bar{\alpha},\bar{\beta}}$. Finally, Lemma~\ref{lem:ideal-maps} shows $f := \xxx^B \cdot \bar{f} \in I_{\alpha,\beta}$. Since 
    \begin{equation}
        \initial_\prec(f) = \initial_\prec(\xxx^B \cdot \bar{f}) = \xxx^B \cdot \initial_\prec(\bar{f}) = \xxx^B \cdot \xxx^{\overline{C}} = \xxx^C,
    \end{equation}
    we may use $f \in I_{\alpha,\beta}$ to obtain a congruence of the form \eqref{eq:goal-congruence}, which completes the proof. (If (2) held instead of (1), we would apply Lemma~\ref{lem:zigzag-leading-transpose} instead of Lemma~\ref{lem:zigzag-leading}.) 
\end{proof}

\subsection{Standard monomial basis and Hilbert series} We are ready to reap the harvest of the technical lemmata. We interpret the quotient $R_{\alpha,\beta}$ in terms of orbit harmonics and describe its diagonal standard monomial basis.

\begin{theorem}
    \label{thm:standard-monomial-basis}
    Let $n \geq 0$ and let $\alpha,\beta \models_0 n$ be weak compositions of $n$ of lengths $\ell(\alpha) = k$ and $\ell(\beta) = p$. We have the equality of ideals in $\FF[\xxx_{k \times p}]$:
    \begin{equation}
        I_{\alpha,\beta} = \gr \, \II(\CCC_{\alpha,\beta})
    \end{equation}
    and therefore the identification
    \begin{equation}
        R_{\alpha,\beta} = \RRR(\CCC_{\alpha,\beta})
    \end{equation}
    of orbit harmonics quotient rings. Furthermore, if $\prec$ is a diagonal term order on the monomials of $\FF[\xxx_{k \times p}]$, the standard monomial basis of $R_{\alpha,\beta}$ with respect to $\prec$ is 
    $\{ \mb(A) \,:\, A \in \CCC_{\alpha,\beta} \}$.
\end{theorem}

\begin{proof}
    We establish the ideal containment
    \begin{equation}
    \label{eq:orbit-harmonics-ideal-containment}
        I_{\alpha,\beta} \subseteq \gr \, \II(\CCC_{\alpha,\beta})
    \end{equation}
    by checking that each generator of $I_{\alpha,\beta}$ is the highest degree component of some polynomial in $\II(\CCC_{\alpha,\beta})$. Indeed, if $1 \leq i \leq k$ we have $x_{i,1} + \cdots + x_{i,p} - \alpha_i \in \II(\CCC_{\alpha,\beta})$ so that the row sum generator $x_{i,1} + \cdots + x_{i,p}$ of $I_{\alpha,\beta}$ lies in $\gr \, \II(\CCC_{\alpha,\beta})$. Similarly, since $x_{1,j} + \cdots + x_{k,j} - \beta_j \in \II(\CCC_{\alpha,\beta})$ for $1 \leq j \leq p$, one has $x_{1,j} + \cdots + x_{k,j} \in \gr \, \II(\CCC_{\alpha,\beta})$. Suppose $1 \leq i \leq k$ and $x_{i,1}^{a_1} \cdots x_{i,p}^{a_p}$ is a monomial satisfying $a_1 + \cdots + a_p > \alpha_i$. It follows that 
    \begin{equation}
        \prod_{j=1}^p x_{i,j} (x_{i,j} - 1) \cdots (x_{i,j} - a_j + 1) \in \II(\CCC_{\alpha,\beta})
    \end{equation}
    which implies $x_{i,1}^{a_1} \cdots x_{i,p}^{a_p} \in \gr \, \II(\CCC_{\alpha,\beta}).$ If $1 \leq j \leq p$ and $x_{1,j}^{b_1} \cdots x_{k,j}^{b_k}$ satisfies $b_1 + \cdots + b_k > \beta_j$, one proves that $x_{1,j}^{b_1} \cdots x_{k,j}^{b_k} \in \gr \, \II(\CCC_{\alpha,\beta})$ in a similar way. This completes the proof of \eqref{eq:orbit-harmonics-ideal-containment}.

    Applying the orbit harmonics isomorphism \eqref{eq:orbit-harmonics-background}, the containment \eqref{eq:orbit-harmonics-ideal-containment}, and Lemma~\ref{lem:spanning-lemma} one has the chain of (in)equalities
    \begin{equation}
        \# \CCC_{\alpha,\beta} = \dim_\FF \RRR(\CCC_{\alpha,\beta}) = \dim_\FF \FF[\xxx_{k \times p}]/\gr \, \II(\CCC_{\alpha,\beta}) \leq \dim_\FF \FF[\xxx_{k \times p}]/I_{\alpha,\beta} \leq \# \CCC_{\alpha,\beta}.
    \end{equation}
    This forces the ideal containment \eqref{eq:orbit-harmonics-ideal-containment} to be an equality $I_{\alpha,\beta} = \gr \, \II(\CCC_{\alpha,\beta})$ so that $R_{\alpha,\beta} = \RRR(\CCC_{\alpha,\beta})$. The claim about standard monomial bases follows from Lemma~\ref{lem:spanning-lemma}.
\end{proof}

To give an example of Theorem~\ref{thm:standard-monomial-basis}, suppose $\alpha = (3,2)$ and $\beta = (2,2,1)$. We have 
\[ \CCC_{\alpha,\beta} = \left\{
\begin{footnotesize}\begin{pmatrix} 2 & 1 & 0 \\ 0 & 1 & 1 \end{pmatrix}, \, \, \begin{pmatrix} 2 & 0 & 1 \\ 0 & 2 & 0 \end{pmatrix}, \, \, \begin{pmatrix} 1 & 2 & 0 \\ 1 & 0 & 1 \end{pmatrix}, \, \, \begin{pmatrix} 0 & 2 & 1 \\ 2 & 0 & 0 \end{pmatrix}, \,\, \begin{pmatrix} 1 & 1 & 1 \\ 1 & 1 & 0 \end{pmatrix}\end{footnotesize}   \right\}.\]
The matrix-ball monomials $\mb(A)$ of these contingency tables are 
\[ \mb \begin{footnotesize} \begin{pmatrix} 2 & 1 & 0 \\ 0 & 1 & 1 \end{pmatrix}\end{footnotesize}  = 1, \quad \quad \mb \begin{footnotesize} \begin{pmatrix} 2 & 0 & 1 \\ 0 & 2 & 0 \end{pmatrix}\end{footnotesize} = x_{2,3}, \quad \quad \mb \begin{footnotesize} \begin{pmatrix} 1 & 2 & 0 \\ 1 & 0 & 1 \end{pmatrix}\end{footnotesize} = x_{2,2},\]
\[ \mb \begin{footnotesize} \begin{pmatrix} 0 & 2 & 1 \\ 2 & 0 & 0 \end{pmatrix}\end{footnotesize} = x_{2,2}^2, \quad \quad \mb \begin{footnotesize} \begin{pmatrix} 1 & 1 & 1 \\ 1 & 1 & 0 \end{pmatrix}\end{footnotesize} = x_{2,2} \cdot x_{2,3}.\]

Let $A \in \CCC_{\alpha,\beta}$ and suppose $A \mapsto (P,Q)$ under RSK. Let $\lambda \vdash n$ be the common shape of $P$ and $Q$.
By Theorem~\ref{thm:zigzag-characterization} and Equation~\eqref{eq:m-monomial-degree} we have \[\deg(\mb(A)) = n - \zigzag(A) = n - \lambda_1.\] The following formula for the Hilbert series of $R_{\alpha,\beta}$ is an immediate consequence of Theorem~\ref{thm:standard-monomial-basis}.

\begin{corollary}
    \label{cor:hilbert-series}
    Let $n \geq 0$ and let $\alpha, \beta \models_0 n$ be weak compositions of $n$. The Hilbert series of $R_{\alpha,\beta}$ is given in terms of Kostka numbers by 
    \begin{equation}
    \Hilb(R_{\alpha,\beta};q) = \sum_{A \in \CCC_{\alpha,\beta}} q^{n- \zigzag(A)} = \sum_{\lambda \vdash n} K_{\lambda,\alpha} \cdot K_{\lambda,\beta} \cdot q^{n - \lambda_1}.
    \end{equation}
\end{corollary}

In our example of $\alpha = (3,2)$ and $\beta = (2,2,1)$, the semistandard tableaux of content $\alpha$ are 
\begin{footnotesize}
\[ \begin{young} 1 & 1 & 1 \\ 2 & 2 \end{young}, \quad  \begin{young} 1 & 1 & 1 & 2 \\ 2 \end{young}, \quad \begin{young} 1 & 1 & 1 & 2 & 2 \end{young} \,,\]
\end{footnotesize}
while the semistandard tableaux of content $\beta$ are
\begin{footnotesize}
\[ \begin{young} 1 & 1 \\ 2 & 2 \\ 3 \end{young}, \quad \begin{young} 1 & 1 & 2 \\ 2 \\ 3 \end{young}, \quad \begin{young} 1 & 1 & 2 \\ 2 & 3 \end{young}, \quad \begin{young} 1 & 1 & 3 \\ 2 & 2 \end{young},\quad
 \begin{young} 1 & 1 & 2 & 2 \\ 3 \end{young}, \quad \begin{young} 1 & 1 & 2 & 3 \\ 2 \end{young}, \quad \begin{young} 1 & 1 & 2 & 2 & 3 \end{young}\, .\]
\end{footnotesize}
Corollary~\ref{cor:hilbert-series} says that 
\[
\Hilb(R_{\alpha,\beta};q) = \overbrace{0 \cdot 1 \cdot q^3}^{\lambda = (2,2,1)} + \overbrace{0 \cdot 1 \cdot q^2}^{\lambda = (3,1,1)} + \overbrace{1 \cdot 2 \cdot q^2}^{\lambda = (3,2)} + \overbrace{1 \cdot 2 \cdot q^1}^{\lambda = (4,1)} + \overbrace{1 \cdot 1 \cdot q^0}^{\lambda = (5)} = 2 q^2 + 2q + 1
\]
which is consistent with the standard monomials of $R_{\alpha,\beta}$ computed above.

If $\gamma = (\gamma_1,\dots,\gamma_m) \in (\ZZ_{\geq 0})^m$, we write $\Stab(\gamma) \subseteq \symm_m$ for the subgroup
\begin{equation}
    \Stab(\gamma) := \{ w \in \symm_m \,:\, \gamma_{w(i)} = \gamma_i \text{ for all } 1 \leq i \leq m \}.
\end{equation}
In particular, if $\alpha,\beta \models_0 n$ satisfy $\ell(\alpha)= k$ and $\ell(\beta) = p$, we have $\Stab(\alpha) \times \Stab(\beta) \subseteq \symm_k \times \symm_p$. As in the introduction, the product group $\symm_k \times \symm_p$ acts on the polynomial ring $\FF[\xxx_{k \times p}]$ by row and column permutation of its variable matrix. The ideal $I_{\alpha,\beta} \subseteq \FF[\xxx_{k \times p}]$ and therefore the quotient ring $R_{\alpha,\beta} = \FF[\xxx_{k \times p}]/I_{\alpha,\beta}$ are stable under the subgroup $\Stab(\alpha) \times \Stab(\beta)$. The product group $\symm_k \times \symm_p$ also acts on the space $\Mat_{k \times p}(\FF)$ of $k \times p$ matrices over $\FF$; the set $\CCC_{\alpha,\beta}$ of $\alpha,\beta$ contingency tables is stable under $\Stab(\alpha) \times \Stab(\beta)$. The following result follows from \eqref{eq:orbit-harmonics-background} and Theorem~\ref{thm:standard-monomial-basis}.

\begin{corollary}
    \label{cor:ungraded-module-structure}
    Let $n \geq 0$ and let $\alpha, \beta \models_0 n$ be weak compositions of $n$. We have an isomorphism of ungraded $\Stab(\alpha) \times \Stab(\beta)$-modules
    \[ R_{\alpha,\beta} \cong \FF[\CCC_{\alpha,\beta}].\]
\end{corollary}

\section{Graded Module structure}
\label{sec:Module}

The group of combinatorial symmetries $\Stab(\alpha) \times \Stab(\beta)$ of the weak compositions $\alpha$ and $\beta$ acts on the set $\CCC_{\alpha,\beta}$ of contingency tables by row and column permutation. The orbit harmonics ring $\RRR(\CCC_{\alpha,\beta})$ gives a graded refinement of this action. In this section we give a technique for computing this graded structure, thus giving a graded refinement of Corollary~\ref{cor:ungraded-module-structure}. Without loss of generality, we restrict to the case where $\alpha$ and $\beta$ are partitions.

\subsection{Permutation group embeddings associated to partitions} Let $\mu = (\mu_1 \geq \cdots \geq \mu_k) \vdash n$ be a partition with $k$ parts. We associate a group embedding 
\begin{equation}
    \iota_\mu: \Stab(\mu) \hookrightarrow \symm_n
\end{equation}
to $\mu$ as follows. The subgroup $\Stab(\mu) \subseteq \symm_k$ is generated by transpositions $(i,i+1)$ for which $\mu_i = \mu_{i+1}$. It $(i,i+1)$ is such a transposition, we define its image $\iota_\mu(i,i+1) \in \symm_n$ by the product of 2-cycles
\begin{equation}
    \iota_\mu: (i,i+1) \mapsto (\mu_1 + \cdots + \mu_{i-1}+1, \mu_1 + \cdots + \mu_i + 1) \cdots (\mu_1 + \cdots + \mu_{i-1} + \mu_i, \mu_1 + \cdots + \mu_i + \mu_{i+1}).
\end{equation}
For example, if $\mu = (3,2,2,2,1,1) \vdash 11$, we have $\Stab(\mu) = \langle (2,3), (3,4), (5,6) \rangle \subseteq \symm_6$ and
\[ \iota_\mu: (2,3) \mapsto (4,6)(5,7), \, \quad \iota_\mu: (3,4) \mapsto (6,8)(7,9), \quad \quad \iota_\mu: (5,6) \mapsto (10,11).\]
It is not hard to see that $\iota_\mu$ is an injective group homomorphism.

Let $\eta_\mu \in \FF[\symm_n]$ be the group algebra element
\begin{equation}
    \eta_\mu := \sum_{w \in \symm_\mu} w \in \FF[\symm_n].
\end{equation}
which symmetrizes over the parabolic subgroup $\symm_\mu$. If $V$ is any $\FF[\symm_n]$-module, the $\FF$-linear subspace $V^{\symm_\mu} := \{ v \in V \,:\, w \cdot v = v \text{ for all } w \in \symm_\mu \}$ is given by
\begin{equation}
    V^{\symm_{\mu}} = \eta_\mu V := \{ \eta_\mu \cdot v \,:\, v \in V \}.
\end{equation}
For any $u \in \Stab(\mu) \subseteq \symm_k$, one has $\iota_\mu(u) \eta_\mu = \eta_\mu \iota_\mu(u)$ within the group algebra $\FF[\symm_n]$. This endows $V^{\symm_\mu} = \eta_\mu V$ with the structure of a  $\Stab(\mu)$-module via
\begin{equation}
    u \cdot v := \iota_\mu(u) \cdot v \quad \quad \text{for all $u \in \Stab(\mu)$ and $v \in V^{\symm_\mu}$.}
\end{equation}

\subsection{Graded action on contingency tables}
If $\mu,\nu \vdash n$ are two partitions, the product group $\Stab(\mu) \times \Stab(\nu)$ acts on the locus of contingency tables $\CCC_{\mu,\nu}$ by row and column permutation. This induces a graded $\Stab(\mu) \times \Stab(\nu)$-module structure on $\RRR(\CCC_{\mu,\nu}) = \FF[\xxx_{k \times p}]/\gr \, \II(\CCC_{\mu,\nu})$ where $\ell(\mu) = k$ and $\ell(\nu) = p$. This graded module structure is given as follows. If $\lambda \vdash n$ is a partition, we write $V^\lambda$ for the associated irreducible $\symm_n$-module.

\begin{theorem}
    \label{thm:graded-module-structure}
    Let $\mu,\nu \vdash n$ be two partitions and consider a degree $d \geq 0$. We have the isomorphism of $\Stab(\mu) \times \Stab(\nu)$-modules
    \[
    (R_{\mu,\nu})_d = \RRR(\CCC_{\mu,\nu})_d \cong \bigoplus_{\substack{\lambda \vdash n \\ \lambda_1 = n-d }} \eta_\mu V^\lambda \otimes \eta_\nu V^\lambda = 
    \bigoplus_{\substack{\lambda \vdash n \\ \lambda_1 = n-d }} (V^\lambda)^{\symm_\mu} \otimes (V^\lambda)^{\symm_\nu}.
    \]
\end{theorem}

\begin{proof}
    The equality $(R_{\mu,\nu})_d = \RRR(\CCC_{\mu,\nu})_d$ holds by Theorem~\ref{thm:standard-monomial-basis}. The $\Stab(\mu) \times \Stab(\nu)$-module structure on the direct summand $\eta_\mu V^\lambda \otimes \eta_\nu V^\lambda$ is 
    \[ (u_1,u_2) \cdot (v_1 \otimes v_2) := (u_1 \cdot v_1) \otimes (u_2 \cdot v_2) \quad \text{for $u_1 \in \Stab(\mu), \, u_2 \in \Stab(\nu), \, v_1 \in \eta_\mu V^\lambda, \, v_2 \in \eta_\nu V^\lambda.$}\]
    This is a well-defined action by the discussion in the previous subsection.

    We write $\symm_n := \CCC_{(1^n),(1^n)} \subseteq \Mat_{n \times n}(\FF)$ for the locus of $n \times n$ permutation matrices and $R_n := R_{(1^n),(1^n)} = \RRR(\symm_n)$ for its graded orbit harmonics quotient ring. We write $I_n = I_{(1^n),(1^n)} := \gr \, \II(\symm_n)$ for the defining ideal of $R_n$. To hopefully reduce confusion, we use $x$-variables for the quotient $R_{\mu,\nu}$ and $y$-variables for the quotient $R_n$, so that 
    \begin{equation}
        R_{\mu,\nu} = \RRR(\CCC_{\mu,\nu}) = \FF[\xxx_{k \times p}]/I_{\mu,\nu} = \FF[\xxx_{k \times p}]/\gr \, \II(\CCC_{\mu,\nu})
    \end{equation}
    and
    \begin{equation}
        R_n = \RRR(\symm_n) = \FF[\yyy_{n \times n}]/I_n = \FF[\yyy_{n \times n}]/\gr \, \II(\symm_n).
    \end{equation}
    We use Theorem~\ref{thm:standard-monomial-basis} together with the graded $\symm_n \times \symm_n$-structure of $R_n$ calculated by the second author \cite{Rhoades} to determine the graded $\Stab(\mu) \times \Stab(\nu)$-structure of $R_{\mu,\nu}$.

    Theorem~\ref{thm:standard-monomial-basis} yields an ungraded isomorphism of $\symm_n \times \symm_n$-modules
    \begin{equation}
    \label{eq:module-identification-one}
        R_n \cong_{\symm_n \times \symm_n} \FF[\symm_n]
    \end{equation}
    where $\symm_n \times \symm_n$ acts on the group algebra $\FF[\symm_n]$ by left and right multiplication. The set 
    \begin{equation}
    \symm_\mu \backslash \symm_n / \symm_\nu := \{ \symm_{\mu} \cdot w \cdot \symm_\nu \,:\, w \in \symm_n \}
    \end{equation}of double cosets is in bijection with the family $\CCC_{\mu,\nu}$ of contingency tables. If we take $\symm_\mu \times \symm_\nu$-invariants on both sides of Equation~\eqref{eq:module-identification-one}, we obtain 
    \begin{equation}
        \label{eq:module-identification-two}
        (\eta_\mu,\eta_\nu) \cdot R_n = (R_n)^{\symm_\mu \times \symm_\nu} \cong_\FF \FF[\CCC_{\mu,\nu}] 
    \end{equation}
    as $\FF$-vector spaces.

    We define a homomorphism $\tilde{\sigma}_{\mu,\nu}: \FF[\xxx_{k \times p}] \longrightarrow \FF[\yyy_{n \times n}]$ of $\FF$-algebras by the rule
    \begin{equation}
    \label{eq:tilde-sigma-definition}
        \tilde{\sigma}_{\mu,\nu}(x_{i,j}) := \sum_{a \, = \, \mu_1 + \cdots + \mu_{i-1}+1}^{\mu_1 + \cdots + \mu_{i-1} + \mu_1}  \left(\sum_{b  \, = \, \nu_1 + \cdots + \nu_{i-1} + 1}^{\nu_1 + \cdots + \nu_{i-1} + \nu_i} y_{a,b} \right).
    \end{equation}
    In other words, if we partition the $n \times n$ matrix of variables $\yyy_{n \times n}$ into block submatrices with row blocks of sizes $\mu_1,\dots,\mu_k$ and column blocks of sizes $\nu_1,\dots,\nu_p$, then $\tilde{\sigma}_{\mu,\nu}(x_{i,j})$ is the sum of the $y$-variables appearing in the $(i,j)$-block. For example, if $\mu = (2,2,1)$ and $\nu = (3,2)$ we have the block decomposition 
    \[ \yyy_{5 \times 5} = 
    \left(\begin{array}{c c c | c c }
     y_{1,1} & y_{1,2} & y_{1,3} & y_{1,4} & y_{1,5} \\
     y_{2,1} & y_{2,2} & y_{2,3} & y_{2,4} & y_{2,5} \\
     \hline
     y_{3,1} & y_{3,2} & y_{3,3} & y_{3,4} & y_{3,5} \\
     y_{4,1} & y_{4,2} & y_{4,3} & y_{4,4} & y_{4,5} \\ 
     \hline
     y_{5,1} & y_{5,2} & y_{5,3} & y_{5,4} & y_{5,5}
    \end{array}\right)
    \]
    and the map $\tilde{\sigma}_{\mu,\nu}: \FF[\xxx_{3\times2}] \to \FF[\yyy_{5 \times 5}]$ is defined by
    \[ \tilde{\sigma}_{\mu,\nu}: x_{1,1} \mapsto y_{1,1} + y_{1,2} + y_{1,3} + y_{2,1} + y_{2,2} + y_{2,3}, \quad \quad \tilde{\sigma}_{\mu,\nu}: x_{1,2} \mapsto y_{1,4} + y_{1,5} + y_{2,4} + y_{2,5}, \]
    \[ \tilde{\sigma}_{\mu,\nu}: x_{2,1} \mapsto y_{3,1} + y_{3,2} + y_{3,3} + y_{4,1} + y_{4,2} + y_{4,3}, \quad \quad \tilde{\sigma}_{\mu,\nu}: x_{2,2} \mapsto y_{3,4} + y_{3,5} + y_{4,4} + y_{4,5},\]
    \[ \tilde{\sigma}_{\mu,\nu}: x_{3,1} \mapsto y_{5,1} + y_{5,2} + y_{5,3}, \quad \quad \tilde{\sigma}_{\mu,\nu}: x_{3,2} \mapsto y_{5,4} + y_{5,5}.\]
    It is easily seen that 
    \begin{equation}
    \label{eq:module-identification-three}
        \image ( \tilde{\sigma}_{\mu,\nu}) \subseteq \FF[\yyy_{n \times n}]^{\symm_{\mu} \times \symm_\nu}.
    \end{equation}
    This containment is strict in most cases. Furthermore, we claim
    \begin{equation}
        \label{eq:module-identification-four}
        \tilde{\sigma}_{\mu,\nu}(I_{\mu,\nu}) \subseteq I_n.
    \end{equation}
    The containment \eqref{eq:module-identification-four} may be checked on the generators of $I_{\mu,\nu}$. For the row sum and column sum generators this is immediate from \eqref{eq:tilde-sigma-definition}. If a row product $x_{i,1}^{a_1} \cdots x_{i,p}^{a_p}$ satisfies $a_1 + \cdots + a_p > \alpha_i$, we have $\tilde{\sigma}_{\mu,\nu}(x_{i,1}^{a_1} \cdots x_{i,p}^{a_p}) \in I_n$ because every term in the expansion of 
    \[ \tilde{\sigma}_{\mu,\nu}(x_{i,1}^{a_1} \cdots x_{i,p}^{a_p}) = \tilde{\sigma}_{\mu,\nu}(x_{i,1})^{a_1} \cdots \tilde{\sigma}_{\mu,\nu}(x_{i,p})^{a_p}\] 
    is a product of $a_1 + \cdots + a_p$ $y$-variables living in $\alpha_i$ rows of $\yyy_{n \times n}$, and any such product will contain two variables in the same row by the Pigeonhole Principle.  
    A similar argument shows $\tilde{\sigma}_{\mu,\nu}(x_{i,j}^{b_1} \cdots x_{k,j}^{b_k}) \in I_n$ when $b_1 + \cdots + b_k > \beta_j$.

    The containments \eqref{eq:module-identification-three} and \eqref{eq:module-identification-four} imply that  $\tilde{\sigma}_{\mu,\nu}: \FF[\xxx_{k \times p}] \to \FF[\yyy_{n \times n}]$ induces an $\FF$-algebra homomorphism
    \begin{equation}
        \label{eq:module-identification-five}
        \sigma_{\mu,\nu}: R_{\mu,\nu} \longrightarrow (R_n)^{\symm_{\mu} \times \symm_{\nu}}.
    \end{equation}
    Although the containment \eqref{eq:module-identification-three} is usually strict so that $\tilde{\sigma}_{\mu,\nu}$ is usually not a surjection, we will see that $\sigma_{\mu,\nu}$ is an isomorphism. We prove the following weaker claim first.

    \begin{quote}
        $(\diamondsuit)$ The $\FF$-algebra homomorphism $\sigma_{\mu,\nu}: R_{\mu,\nu} \to (R_n)^{\symm_{\mu} \times \symm_{\nu}}$ is an epimorphism.
    \end{quote}
    
    The proof of $(\diamondsuit)$ proceeds by careful analysis of the quotient ring $(R_n)^{\symm_{\mu} \times \symm_{\nu}}$. We first rephrase $(\diamondsuit)$ as a statement involving generation of the $\FF$-algebra $(R_n)^{\symm_{\mu} \times \symm_{\nu}}$.

    \begin{quote}
        $(\spadesuit)$ The collection of cosets $\{ \tilde{\sigma}_{\mu,\nu}(x_{i,j}) + I_n \,:\, 1 \leq i \leq k, \, 1 \leq j \leq p \}$ generates the $\FF$-algebra $(\FF[\yyy_{n \times n}]/I_n)^{\symm_\mu \times \symm_{\nu}}$.
    \end{quote}

    It follows from the definitions that $(\diamondsuit) \Leftrightarrow (\spadesuit)$. We prove $(\spadesuit)$ by establishing a stronger and more convenient statement.
    
    Let $J_n \subseteq I_n$ be the smaller ideal generated by all products of two $y$-variables in the same row or column, i.e.
    \begin{equation}
        J_n := ( y_{i,j} \cdot y_{i,j'}, \, y_{i,j} \cdot y_{i',j} \,: \, 1 \leq i,i' \leq n, \, 1 \leq j,j' \leq n ).
    \end{equation}
    (The row sums and column sums which appear as generators of $I_n$ do not appear as generators of $J_n$.) 
    We will prove the following statement.

    \begin{quote}
        $(\clubsuit)$ The collection of cosets $\{ \tilde{\sigma}_{\mu,\nu}(x_{i,j}) + J_n \,:\, 1 \leq i \leq k, \, 1 \leq j \leq p \}$ generates the $\FF$-algebra $(\FF[\yyy_{n \times n}]/J_n)^{\symm_\mu \times \symm_{\nu}}.$
    \end{quote}

    Since $J_n \subseteq I_n$, we have $(\clubsuit) \Rightarrow (\spadesuit)$. Our proof of $(\clubsuit)$ begins by factoring the invariant ring $\FF[\yyy_{n \times n}]^{\symm_\mu \times \symm_\nu}$ along the block matrix decomposition of $\yyy_{n \times n}$ determined by $\mu$ and $\nu$. More precisely, we have the tensor product decomposition
    \begin{equation}
    \label{eq:tensor-product-identification}
        \FF[\yyy_{n \times n}]^{\symm_\mu \times \symm_\nu} = \bigotimes_{i=1}^k \bigotimes_{j=1}^p \FF\left[ y_{a,b} \,:\,  \begin{array}{c} \mu_1 + \cdots + \mu_{i-1} < a \leq \mu_1 + \cdots  + \mu_i \\ \nu_1 + \cdots + \nu_{i-1} < b \leq \nu_1 + \cdots + \nu_i \end{array} \right]^{\symm_{\mu_i} \times \symm_{\nu_j}}.
    \end{equation}
    Observe that $\tilde{\sigma}_{\mu,\nu}(x_{i,j})$ is the sum of variables involved in the $(i,j)$ factor of the decomposition~\eqref{eq:tensor-product-identification}. We prove generation modulo $J_n$ in each individual factor as follows.

    Let $r,s > 0$, let $\zzz_{r \times s} = (z_{i,j})$ be an $r \times s$ matrix of variables, and let $\FF[\zzz_{r \times s}]$ be the polynomial ring over these variables. The product group $\symm_r \times \symm_s$ acts on the rows and columns of the matrix $\zzz_{r \times s}$, and this induces an action on $\FF[\zzz_{r \times s}]$.  Write $J_{r,s} \subseteq \FF[\zzz_{r \times s}]$ for the ideal
    \begin{equation}
        J_{r,s} := ( z_{i,j} \cdot z_{i,j'}, \, z_{i,j} \cdot z_{i',j} \,:\, 1 \leq i,i' \leq r, \, \, 1 \leq j,j' \leq s)
    \end{equation}
    and observe that $J_{r,s}$ is $\symm_r \times \symm_s$-stable. 
    Let $\Sigma \in \FF[\zzz_{r \times s}]$ be $\Sigma := \sum_{i=1}^r \sum_{j=1}^s z_{i,j}.$  We establish the following.

    \begin{quote}
          ($\heartsuit$)
        The coset $\Sigma + J_{r,s}$ generates the $\FF$-algebra $(\FF[\zzz_{r \times s}] /J_{r,s})^{\symm_r \times \symm_s}$.
    \end{quote}

    The heart of the surjectivity of $\sigma_{\mu,\nu}$ is established as follows.
    The $\FF$-algebra $\FF[\zzz_{r \times s}]^{\symm_r \times \symm_s}$ is spanned over $\FF$ by the set 
    \[ \left\{ \sum_{(w_1,w_2) \, \in \, \symm_r \times \symm_s} (w_1,w_2) \cdot \zzz^A \,:\, A \in \Mat_{r \times s}(\ZZ_{\geq 0})\right\}.\]
    If $A$ has a row sum or column sum $> 1$, we have $\zzz^A \in J_{r,s}$ and therefore one has the membership $\sum_{(w_1,w_2) \, \in \, \symm_r \times \symm_s} (w_1,w_2) \cdot \zzz^A \in J_{r,s}$. If every row sum and column sum of $A$ is $\leq 1$ and $\sum_{i,j} a_{i,j} = m$, one checks that 
    \begin{equation}
         \sum_{(w_1,w_2) \, \in \, \symm_r \times \symm_s} (w_1,w_2) \cdot \zzz^A \equiv \gamma \cdot \Sigma^m \mod J_{r,s}
    \end{equation}
    where $\gamma \in \FF$ is a nonzero constant.\footnote{It can be shown that $\gamma = (k-m)! \cdot (p-m)!$.}
    This establishes the statement $(\heartsuit)$. We have
    \[ (\heartsuit) \Rightarrow (\spadesuit) \Rightarrow (\clubsuit) \Leftrightarrow (\diamondsuit)\]
    where the implication $(\heartsuit) \Rightarrow (\spadesuit)$ follows from the decomposition~\eqref{eq:tensor-product-identification}.
    We conclude that $(\diamondsuit)$ holds and the $\FF$-algebra map $\sigma_{\mu,\nu}$ of \eqref{eq:module-identification-five} is a surjection.

    Theorem~\ref{thm:standard-monomial-basis} implies that $\dim_\FF R_{\mu,\nu} = \# \CCC_{\mu,\nu}$.  By \eqref{eq:module-identification-two} we also have $\dim_\FF (R_n)^{\symm_\mu \times \symm_\nu} = \# \CCC_{\mu,\nu}$. This forces the epimorphism $\sigma_{\mu,\nu}: R_{\mu,\nu} \twoheadrightarrow (R_n)^{\symm_\mu \times \symm_\nu}$ of $\FF$-algebras to be an isomorphism. It follows from \eqref{eq:tilde-sigma-definition} that the isomorphism $\sigma_{\mu,\nu}: R_{\mu,\nu} \xrightarrow{ \, \sim \, } (R_n)^{\symm_\mu \times \symm_\nu}$ is $\Stab(\mu) \times \Stab(\nu)$-equivariant and preserves degrees. We may restrict to degree $d$ to get an isomorphism
    \begin{equation}
        \label{eq:module-identification-six}
        (R_{\mu,\nu})_d \cong_{\Stab(\mu) \times \Stab(\nu)} (R_n)^{\symm_{\mu} \times \symm_\nu}_d = (\eta_\mu,\eta_\nu) \cdot (R_n)_d
    \end{equation}
    of $\Stab(\mu) \times \Stab(\nu)$-modules. According to \cite[Thm. 4.2]{Rhoades}, we have
    \begin{equation}
        \label{eq:module-identification-seven}
        (R_n)_d \cong_{\symm_n} \bigoplus_{\substack{\lambda \vdash n \\ \lambda_1 = n-d}} V^\lambda \otimes V^\lambda.
    \end{equation}
    Combining \eqref{eq:module-identification-six} and \eqref{eq:module-identification-seven} proves the theorem.
\end{proof}

\subsection{Symmetric functions} Theorem~\ref{thm:graded-module-structure} gives an expression for the graded $\Stab(\mu) \times \Stab(\nu)$-structure of $R_{\mu,\nu}$. In order to use this result to calculate the graded decomposition of $R_{\mu,\nu}$ into $\Stab(\mu) \times \Stab(\nu)$-irreducibles, we need the $\Stab(\mu)$-isotypic decomposition of $\eta_\mu V^\lambda$ (and the $\Stab(\nu)$-isotypic decomposition of $\eta_\nu V^\lambda$). In this subsection we describe a method for calculating this decomposition.

Let $\mu = (\mu_1,\dots,\mu_k) \vdash n$ be a partition with $k$ parts, so that $\Stab(\mu) \subseteq \symm_k$ is a parabolic subgroup. Write the parabolic factors of $\Stab(\mu)$ as 
\begin{equation}
    \Stab(\mu) = \symm_{\mult(\mu)_1} \times \cdots \times \symm_{\mult(\mu)_r}
\end{equation}
where $\mult(\mu) = (\mult(\mu)_1, \dots, \mult(\mu)_r) \vdash k$ are the part multiplicities of $\mu$ written in weakly decreasing order. If $V$ is any $\symm_n$-module, then $\eta_\mu V$ is a $\Stab(\mu)$-module, and we have a Frobenius image 
\begin{equation}
    \Frob(\eta_\mu V ) \in \Lambda_{\mult(\mu)_1} \otimes \cdots \otimes \Lambda_{\mult(\mu)_r}.
\end{equation}

\begin{defn}
    \label{def:psi-operator}
    Let $\psi_\mu: \Lambda_n \longrightarrow \Lambda_{\mult(\mu)_1} \otimes \cdots \otimes \Lambda_{\mult(\mu)_r}$ be the linear operator characterized by
    \[ \psi_\mu: \Frob(V) \mapsto \Frob(\eta_\mu V)\]
    for all $\symm_n$-modules $V$.
\end{defn}

For any $\symm_n$-modules $V$ and $W$, we have an isomorphism of $\Stab(\mu)$-modules $\eta_\mu(V \oplus W) \cong \eta_\mu V \oplus \eta_\mu W$. This shows that the operator $\psi_\mu$ is well-defined. Finding the isotypic decomposition of $\eta_\mu V^\mu$ as a $\Stab(\mu)$-module amounts to solving the following problem.

\begin{problem}
    \label{prob:psi-on-schur}
    If $\lambda \vdash n$, give a combinatorial rule for the nonnegative integers $a_{\lambda,\mu}^{\nu^{(1)},\dots,\nu^{(r)}}$ appearing in the Schur expansion \[\psi_\mu(s_\lambda) = \sum_{\nu^{(i)} \, \vdash \, \mult(\mu)_i} a_{\lambda,\mu}^{\nu^{(1)},\dots,\nu^{(r)}} \cdot s_{\nu^{(1)}} \otimes \cdots \otimes s_{\nu^{(r)}}.\]
\end{problem}

Although Problem~\ref{prob:psi-on-schur} is open, a combinatorial rule for the $h$-expansion of $\psi_\mu(h_\lambda)$ is available for $\lambda \vdash n$. This facilitates computation of $\psi_\mu(s_\lambda)$ via the change-of-basis $h_\lambda = \sum_{\rho  \, \vdash \, n} K_{\lambda,\rho} \cdot s_\rho$. Describing this rule requires notation related to multisets.

Let $S$ be a multiset with $n$ elements. A {\em multiset partition of $S$} is a multiset $\pi = \{ B_1,\dots,B_k\}$ of multisets such that we have the multiset union $S = B_1 \Cup \cdots \Cup B_k$. The multisets $B_1,\dots,B_k$ are the {\em blocks} of $S$. The {\em shape} of $\pi$ is the partition $\shape(\pi) \vdash n$ obtained by listing the block cardinalities $\# B_1, \dots ,\# B_k$ in weakly decreasing order. We let $\type(\pi) \vdash k$ be the partition encoding the multiset multiplicities among the blocks of $\pi$. Finally, if $i > 0$ we write $\pi \mid_i := \{ B_j \in \pi \,:\, \# B_j = i \}$ for the restriction of $\pi$ to blocks of size $i$. 

An example should clarify the notions in the above paragraph. Let $S = \{1^9,2^5,3^4\}$. Then
\[ \pi = \{ \{ 1,1,2\}, \{1,1,2\}, \{1,3\}, \{1,3\}, \{1,3\}, \{2,3\}, \{1\},\{1\},\{2\},\{2\}\}.\]
is a multiset partition of $S$ with 10 blocks with  $\shape(\pi) = \{3,3,2,2,2,2,1,1,1,1) \vdash 18$. We have the restrictions
\[ \pi \mid_3  \,= \{ \{1,1,2\}, \{1,1,2\}\}, \quad \pi \mid_2 \, = \{ \{1,3\},\{1,3\},\{1,3\},\{2,3\}\}, \quad \pi \mid_1 \, = \{ \{1\}, \{1\},\{2\},\{2\} \}, \]
and $\pi\mid_i \, = \varnothing$ for $i \neq 1,2,3$. The types of these restrictions are 
\[ \type(\pi\mid_3) = (2), \quad \type(\pi\mid_2) = (3,1), \quad \text{and} \quad \type(\pi\mid_1) = (2,2).\]

\begin{proposition}
    \label{prop:psi-on-h}
    Let $\lambda,\mu \vdash n$ be partitions and consider the multiset $S := \{1^{\lambda_1}, 2^{\lambda_2}, \dots \}$. Suppose $\ell(\mu) = k$ and let $\mult(\mu) = (\mult(\mu)_1 \geq \cdots \geq \mult(\mu)_r) \vdash k$ be the partition encoding the part multiplicities of $\mu$. Suppose $\mu$ has $\mult(\mu)_i$ parts of size $a_i$, where the numbers $a_1,\dots,a_r$ are distinct. We have
    \[ \psi_\mu(h_\lambda) = \sum_{\pi} h_{\type(\pi\mid_{a_1})} \otimes \cdots \otimes h_{\type(\pi\mid_{a_r})}\]
    where the sum is over all multiset partitions $\pi$ of $S$ into multisets which satisfy $\shape(\pi) = \mu$.
\end{proposition}

For example, suppose $\lambda = (3,2)$ and $\mu = (2,1,1,1)$. We have $S = \{1,1,1,2,2\}$, and the relevant multiset partitions $\pi$ are 
\[ \{ \{1,1\},\{1\},\{2\}, \{2\} \}, \quad \{ \{1,2\},\{1\},\{1\},\{2\}\}, \quad \text{and} \quad \{\{2,2\},\{1\},\{1\},\{1\}\}.\]
We have $\mult(\mu) = (3,1)$ and Proposition~\ref{prop:psi-on-h} gives
\[ \psi_\mu(h_\lambda) = h_{2,1} \otimes h_1 + h_{2,1} \otimes h_1 + h_3 \otimes h_1\]
where the first tensor factor tracks $\type(\pi\mid_1)$ and the second tensor factor tracks $\type(\pi \mid_2)$.

\begin{proof}
    Recall that $\Frob(M^\lambda) = h_\lambda$ where $M^\lambda$ is the permutation $\symm_n$-module with basis given by length $n$ words $w = [w_1, \dots, w_n]$ with $\lambda_i$ copies of the letter $i$. For example, if $n = 18$ and $\lambda = (9,5,4) \vdash n$ one such word is
    \[w= [2 \,, 1, \, 1, \, 1 , \, 2, \, 1, \, 1, \, 3, \, 3, \, 2, \, 1, \, 3, \, 3, \, 1, \, 2, \, 1, \, 1, \, 2 ]. \]
    The module $\eta_\mu M^\lambda$ therefore has basis given by the collection of {\em ordered} multiset partitions $B_1,  \dots,  B_k$ of $S$ where $\# B_i = \mu_i$. If $\mu = (3,3,2,2,2,1,1,1,1)$ and $w \in M^\lambda$ is as above, we have
    \[ \eta_\mu \cdot w \quad  \leftrightarrow \quad  \{1, \, 1, \, 2\}, \,  \{1, \, 1, \, 2\}, \quad \{1, \, 3\}, \, \{2, \, 3\}, \, \{1, \, 3\}, \, \{1, 3\}, \quad \{2\}, \, \{1\}, \, \{1\}, \, \{2\}.\]
    The group $\Stab(\mu)$ acts on $\eta_\mu M^\lambda$ by permuting blocks of the same size; in our example we have $\Stab(\mu) = \symm_2 \times \symm_4 \times \symm_4$ and the ordered multisetset partition above generates a $\Stab(\mu)$-orbit of type $M^{(2)} \otimes M^{(3,1)} \otimes M^{(2,2)}$. In general one has
    \begin{equation}
        \eta_\mu M^\lambda \cong_{\Stab(\mu)} \bigoplus_\pi M^{\type(\pi \mid_{a_1})} \otimes \cdots \otimes M^{\type(\pi \mid_{a_r})}
    \end{equation}
    and the result follows.
\end{proof}

\section{Future directions}
\label{sec:Future}

\subsection{Log-concavity}
A sequence $a_0, a_1, a_2, \dots$ in $\RR_{> 0}$ is {\em log-concave} if $a_k^2 \geq a_{k-1} \cdot a_{k+1}$ for all $k > 0$.

\begin{conjecture}
    \label{conj:log-concave-hilbert}
    For any $\alpha,\beta \models_0 n$, the coefficient sequence of the polynomial
    \[ \Hilb(R_{\alpha,\beta};q) = \sum_{A \in \CCC_{\alpha,\beta}} q^{n - \zigzag(A)}\]
    is log-concave.
\end{conjecture}

When $\alpha = \beta = (1^n)$, the zigzag number is the longest increasing subsequence statistic on $\symm_n$. In this context,  Conjecture~\ref{conj:log-concave-hilbert} is equivalent to a conjecture of Chen \cite{Chen}. Conjecture~\ref{conj:log-concave-hilbert} is one of a growing list of conjectures on the log-concavity (or at least unimodality) of $\Hilb(\RRR(\Zpoints);q)$ for `interesting' matrix loci $\Zpoints$ \cite{Liu, LMRZ, Rhoades}.

We propose an equivariant strengthening of Conjecture~\ref{conj:log-concave-hilbert} as follows. Let $G$ be a group and let $V,W$ be $\FF[G]$-modules. The tensor product $V \otimes_\FF W$ is an $\FF[G]$-module via $g \cdot (v \otimes w) := (g \cdot v) \otimes (g \cdot w)$ for all $g \in G, v \in V,$ and $w \in W$. In the $G = \symm_n$ context, this operation is usually called {\em Kronecker product}. A sequence $V_0, V_1, V_2, \dots$ of $G$-modules is {\em $G$-log-concave} if for all $k > 0$ we have a $G$-equivariant injection $V_{k-1} \otimes_\FF V_{k+1} \hookrightarrow V_k \otimes_\FF V_k$. Finally, a graded $G$-module $V = \bigoplus_{k \geq 0} V_k$ is $G$-log-concave if the sequence $V_0, V_1, V_2, \dots$ has this property.

\begin{conjecture}
    \label{conj:equivariant-log-concave}
    For any $\alpha,\beta \models_0 n$, the graded module $R_{\alpha,\beta}$ is $\Stab(\alpha) \times \Stab(\beta)$-log-concave.
\end{conjecture}

Conjecture~\ref{conj:equivariant-log-concave} joins the ranks of a family of equivariant log-concavity conjectures \cite{LMRZ, NR, Rhoades} on $\RRR(\Zpoints)$ for matrix loci $\Zpoints$.

\subsection{Top-heavy conjecture}
Let \(\alpha, \beta \models_0 n\) be compositions of lengths \(k\) and \(p\), respectively, and define  
\begin{equation}
    m := \min\{\zigzag(A) : A \in \mathcal{C}_{\alpha,\beta} \}.
\end{equation}  
Computer evidence suggests that the sequence $\dim(R_{\alpha,\beta})_0,\dim(R_{\alpha,\beta})_2,\dots,\dim(R_{\alpha,\beta})_{n-m}$ is {\em top-heavy}, i.e.
\begin{equation}
    \dim (R_{\alpha,\beta})_{k} \leq \dim (R_{\alpha,\beta})_{n-m-k}.
\end{equation}  
We suggest an algebraic lift of this inequality. To that end, we partition the \(n \times n\) matrix into block submatrices, where the row blocks have sizes \(\alpha_1, \dots, \alpha_k\) and the column blocks have sizes \(\beta_1, \dots, \beta_p\). We then define a linear element \( L_{\alpha,\beta} \in R_{\alpha,\beta} \) as the sum of \( x_{i,j} \) over all pairs \((i,j)\) such that the \((i,j)\)-block contains at least one diagonal element of the \(n \times n\) matrix.  

For example, if \(\alpha = (2,2,1)\) and \(\beta = (3,2)\), the block decomposition of the \(5 \times 5\) matrix is  
\[
    \yyy_{5 \times 5} = 
    \left(\begin{array}{c c c | c c }
     \mathbf{a_{1,1}} & a_{1,2} & a_{1,3} & a_{1,4} & a_{1,5} \\
     a_{2,1} & \mathbf{a_{2,2}} & a_{2,3} & a_{2,4} & a_{2,5} \\
     \hline
     a_{3,1} & a_{3,2} & \mathbf{a_{3,3}} & a_{3,4} & a_{3,5} \\
     a_{4,1} & a_{4,2} & a_{4,3} & \mathbf{a_{4,4}} & a_{4,5} \\ 
     \hline
     a_{5,1} & a_{5,2} & a_{5,3} & a_{5,4} & \mathbf{a_{5,5}}
    \end{array}\right).
\]  
In this case, the element \( L_{(2,2,1),(3,2)} \) is given by  
\[
    L_{(2,2,1),(3,2)} = x_{1,1} + x_{2,1} + x_{2,2} + x_{3,2}.
\]  

\begin{conjecture}  
\label{conj:weak-lefschetz}
For any \(\alpha, \beta \models_0 n\), let  
$m = \min\{ \text{zigzag} (A) : A \in \mathcal{C}_{\alpha,\beta} \}.$
For \( k = 0,1,\dots,\lfloor \frac{m}{2} \rfloor \), the map  
\[
    L_{\alpha,\beta}^{n-m-2k} : (R_{\alpha,\beta})_{k} \to (R_{\alpha,\beta})_{n-m-k}
\]  
is injective.  
\end{conjecture}  

Proving Conjecture~\ref{conj:weak-lefschetz} would likely require geometric techniques. Identifying the underlying geometric structure related to this conjecture is an intriguing open problem.

\subsection{Limiting distributions} Suppose that for each $n >0$ we are given weak compositions $\alpha(n),\beta(n) \models_0 n$.  One can ask for the limiting shape of the coefficient sequence of the Hilbert series
\begin{equation}
\label{eq:limiting}
    \Hilb( R_{\alpha(n),\beta(n)};q) = \sum_{A \in \CCC_{\alpha(n),\beta(n)}} q^{n-\zigzag(A)}
\end{equation}
as $n \to \infty$. When $\alpha(n) = \beta(n) = (1^n)$, this amounts to the limiting distribution of the longest increasing subsequence statistic on $\symm_n$. In this case, Baik, Deift, and Johansson famously showed \cite{BDJ} that this limiting distribution is the Tracy-Widom distribution for a Gaussian unitary ensemble. When $\alpha(n),\beta(n)$ are `reasonable' functions of $n$, one might hope for a nice limiting distribution of \eqref{eq:limiting}. One place to look may be the functions \[ \alpha(n) = \beta(n) = (\overbrace{d,\dots,d}^{q},r), \quad \quad n = q\cdot d + r, \quad 0 \leq r < d\]
for fixed $d$. Figure~\ref{fig:coefficient-figure} shows the case $n  = 60$ where $d=1,2,3,4$.

%{\color{red} Jaeseong, could you please fill in one of the Hilbert series for $(2,\dots,2), (3,\dots,3)$ or $(4,\dots,4)$?}

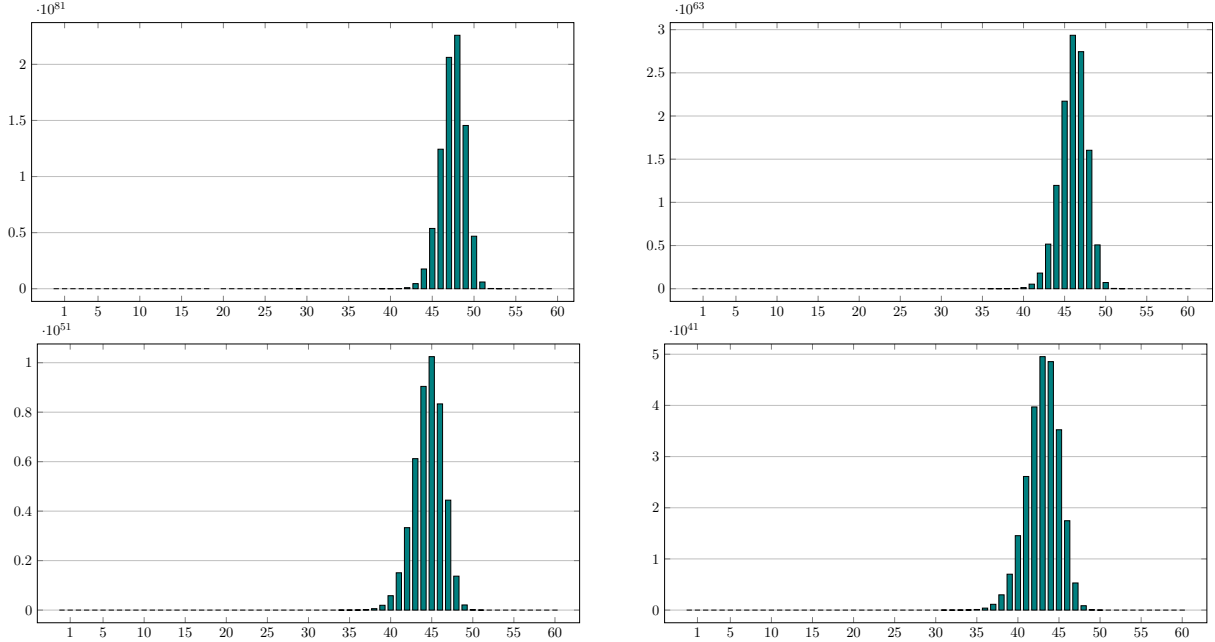
\begin{figure}
\begin{tikzpicture}[scale = 0.5]
    \begin{axis}[
        ymajorgrids=true,
        bar width=4pt,
        width=16cm,
        height=9cm,
        enlargelimits=0.05,
        xtick={1,5,10,15,20,25,30,35,40,45,50,55,60},
        ylabel style={at={(axis description cs:0.02,0.5)}, anchor=south},
        xlabel style={at={(axis description cs:0.5,-0.05)}, anchor=north}
    ]
    \addplot[ybar, fill=teal] coordinates {
        (0,1)
        (1,3481)
        (2,5851621)
        (3,6329639181)
        (4,4953109533951)
        (5,2988846947868807)
        (6,1447687482601462467)
        (7,578499468416768375547)
        (8,194538945880191885164142)
        (9,55882289445190125856537982)
        (10,13871858035569655993122428198)
        (11,3003101053234619836243294988438)
        (12,571194238941869175779849437632858)
        (13,96024776493391284512354802786801738)
        (14,14338575949172527832964867110076216498)
        (15,1909464678419065197802131758896939250754)
        (16,227535457430697745412499435864265131254799)
        (17,24328028687991328153614089674352694270581559)
        (18,2339113811472688502277654778306794059853563499)
        (29,202607444815864518792560988913570051638925224579)
        (20,15831468365324027218299523307120731328900971994505)
        (21,1117112951073704289164060302012753184712282502843905)
        (22,71234520883705192260893127544474396900805731744717605)
        (23,4106553547655147341710844172909909933553257664495111405)
        (24,214047059046253016288888877319022168888880344026378396380)
        (25,10085939380850856133146998090323665735800777575608380654076)
        (26,429447776828374945008891956588902876262471997021459648212556)
        (27,16511387220795567604685869448544444896127174737028045397681196)
        (28,572672495594979262825338794738392020356281425452234267523956756)
        (29,17894607700299703295952617215614009799494649283846631618530060276)
        (30,502968058628572662277191566575373626415624915252154954972665008452)
        (31,12691871855481828626593458025125606857596131143782738903573825981796)
        (32,286869467381919681822222469760492255054776533189183154318863752312230)
        (33,5792237419925383613451898590561388009628831263062366370911868157401964)
        (34,104145900363220144830466866571023152418199800997341553230438935249107258)
        (35,1661393542805513071742417067989822659686364866756920591886636429663178680)
        (36,23413655153323993212944806641432538187723487636946542215904509952446312690)
        (37,290030563932022002118220602447753575011631132138930828877529635970794362700)
        (38,3139198038828528286203710931455906285276708940460977128156446860885650927310)
        (39,29480487670047223803921747890109023556319166001013451281038912984170310753120)
        (40,238188858817559653907757766891040131636069359492222134403147802254653854771728)
        (41,1638759009110121823982506004487303838241549550728662507775364230503063574975316)
        (42,9479603856030503157955685146063700672586357866208188042547738509639642888641224)
        (43,45368617608497201905530039854748875664926717869975676357175086535901817870722812)
        (44,175923103423553571947761906278676245128973950100129233119563346326464104855276860)
        (45,537394830317050100339379519887032754646946119740464857911956705681737098244483360)
        (46,1243711511999821270591207565082889798761871176715300197918122808539228337822802740)
        (47,2062265432178679983886852088922462401452557170316484374161761008379074310593517320)
        (48,2259251055120372007733214696091079754018818083465717757461536975882962682765500625)
        (49,1455327054374385756982545351864306579536481867901002010423776062240740978062678405)
        (50,468104440722126644812839632177556187281953330916322512459026291795529190084140003)
        (51,60002895752771099779779088462943847999099581712023250349374731986619450937660387)
        (52,2233494474948495690243110568745222983262159502283551689273891105099703764639203)
        (53,15180807338873516832021030140438444665815147021460591742801378406314408952231)
        (54,9336151984930708021143911217956813677819162164640452787627883005534760901)
        (55,175243028250079660905018843213615929860825569549681884867765690541701)
        (56,17080691328825216538079811628828842602913045806045692424793199)
        (57,353580101123476924257628603730083960324608410748129)
        (58,1583850964596120042686772779038895)
        (59,1)
    };
    \end{axis}
\end{tikzpicture}
\qquad
\begin{tikzpicture} [scale = 0.5]
    \begin{axis}[
        ymajorgrids=true,
        bar width=4pt,
        width=16cm,
        height=9cm,
        enlargelimits=0.05,
        xtick={1,5,10,15,20,25,30,35,40,45,50,55,60},
        ylabel style={at={(axis description cs:0.02,0.5)}, anchor=south},
        xlabel style={at={(axis description cs:0.5,-0.05)}, anchor=north}
    ]
    \addplot[ybar, fill=teal] coordinates {
(0,1)
(1,841)
(2,354061)
(3,99222341)
(4,20772939251)
(5,3458007009059)
(6,475828838208069)
(7,55564045976391549)
(8,5611018333287998094)
(9,496935309504497821134)
(10,39017166613056105384570)
(11,2738959360028835815258970)
(12,173068583550241728595361670)
(13,9897210449802533763047487870)
(14,514496825851552763509513184220)
(15,24399695274794864043341778595020)
(16,1058723151008385941527839429882045)
(17,42130592518712194574634167971696245)
(18,1540435357083966004343352718195483545)
(19,51826770422356436059715283782188358145)
(20,1606219882438452174069220549834078428315)
(21,45891133089625152132275230356301642831515)
(22,1209281480774777011785556853519706897280365)
(23,29395470060999359243680951352723005700187765)
(24,659098988280330043320712834756778148379701690)
(25,13626596113732229532313689792676233898709629098)
(26,259613425706483043104788442235540980576369716518)
(27,4553991967844100976219068036265395109616043641878)
(28,73466116777585555200594511843283827435533803714518)
(29,1088408131965606035557459998018082880109329560809598)
(30,14782682129883106162213943015657418221127638743455932)
(31,183684609649595503230366295878103948527259339443584971)
(32,2083002272479690967751579066602000077565718769751810385)
(33,21496128445515881309196899412370118146097268916102992419)
(34,201200594443759980701927772128887863030919186730675905468)
(35,1701352237910774390071232219819479474951819147405375482339)
(36,12937699931727989346417815166661396947992944881396337111233)
(37,87995372230147781247130251577798543009614486927930417988962)
(38,531859954907043813965986788749754514715245671026932669583481)
(39,2834623973189549587368550887769351935433310039867357190421575)
(40,13195886342048444503694564022824258794317582006988239993646980)
(41,53029305169019833899879592572179557248073322748308734366601715)
(42,181234291399783693056158673334496665270187863703932880968776225)
(43,516603218064900278054841048006851447806548861637501361044150335)
(44,1196499228751076705991955673140003888237634136126738646583922195)
(45,2171686102231538126189856077902410336630120988024346599102330875)
(46,2934780639366372771663229827437168801136783235884367125678127845)
(47,2745039036174755381953488632582538135056444673013420829528233740)
(48,1603349701374345398418328224704850919114719416338166931192956685)
(49,507404991299033908024119596831662591009934195942523995136739230)
(50,71714788245765505002083415020246664918538888945726723745907426)
(51,3468883477658696252186037463735772651607683764130424039367410)
(52,39237649822015913453480732284994546801651406540229443470697)
(53,58907793637514610453246437257659904056909397360456603257)
(54,4936064135933021600652100275143772209085064006756885)
(55,5786066215471583670425501585570494026524159504)
(56,8539712850280932184223812817790172302)
(57,122238900487877503161968)
(58,1)
(59,0)
(60,0)
 };
    \end{axis}
\end{tikzpicture}

\begin{tikzpicture}[scale = 0.5]
    \begin{axis}[
        ymajorgrids=true,
        bar width=4pt,
        width=16cm,
        height=9cm,
        enlargelimits=0.05,
        xtick={1,5,10,15,20,25,30,35,40,45,50,55,60},
        ylabel style={at={(axis description cs:0.02,0.5)}, anchor=south},
        xlabel style={at={(axis description cs:0.5,-0.05)}, anchor=north}
    ]
    \addplot[ybar, fill=teal] coordinates {
( 0 , 1 )
( 1 , 361 )
( 2 , 65341 )
( 3 , 7906261 )
( 4 , 719177551 )
( 5 , 52420150663 )
( 6 , 3186253516843 )
( 7 , 165937172324083 )
( 8 , 7549585535318818 )
( 9 , 304442275353095538 )
( 10 , 11002911356629149066 )
( 11 , 359502339797239304826 )
( 12 , 10692601851901257055506 )
( 13 , 291111353401225142485026 )
( 14 , 7287442288368036430741866 )
( 15 , 168348758054386263802223130 )
( 16 , 3599501938317551906448310830 )
( 17 , 71400694155227503121926444830 )
( 18 , 1316463513160788347155078419030 )
( 19 , 22594272669245891761704445091430 )
( 20 , 361366358888792400203490780213570 )
( 21 , 5390015693513552124503723086570770 )
( 22 , 75011385852425733304256771365580970 )
( 23 , 974178628447547151665657239402577370 )
( 24 , 11805568238242763023509110354681107920 )
( 25 , 133450748667374662358781696838285801104 )
( 26 , 1406295905961892721849944787967277022544 )
( 27 , 13803036464378347005843004989562474536464 )
( 28 , 126042148841839888316909546350415285483144 )
( 29 , 1069237656094398921558553323118779939031304 )
( 30 , 8411751863701399951293940371840902991479912 )
( 31 , 61240234895708395024428026611851036034472296 )
( 32 , 411565030057431300809077512576531612280501630 )
( 33 , 2545700997499461276618917163199401737726522404 )
( 34 , 14442053802916190103277812580531696925380153858 )
( 35 , 74836378148229915969489975840906737917006938096 )
( 36 , 352478056848683720533482500020140337478492859546 )
( 37 , 1500171866367632534136323256362401215964549432916 )
( 38 , 5728735003535474161092096357452398800324569861926 )
( 39 , 19458271029527958082612857733165567490233937748456 )
( 40 , 58150615371890681335594745300853429406452432984696 )
( 41 , 150792522357529573854183023436473756541907463246468 )
( 42 , 333187561496165252636901270704413441800588042651265 )
( 43 , 612166473960477046827611976025849308658592324136807 )
( 44 , 904277500699544877140038509067464774335266117740889 )
( 45 , 1024541433301562666422540784990007500831895276058563 )
( 46 , 833489451366379070596403003889047730507334962969602 )
( 47 , 444406272725141368894294034430612268950733932234542 )
( 48 , 137191319951850876978531041247878895331374114535689 )
( 49 , 20756013058723437188172758556481217175217117333239 )
( 50 , 1227280757964267076010233553821065909232294936776 )
( 51 , 20649444036486494641029250848791539864003804600 )
( 52 , 62211558490125373466704457387386536200535945 )
( 53 , 16680855907382242757569850146200732004479 )
( 54 , 132849045285751465752727299721346970 )
( 55 , 4950062556686902009127867614 )
( 56 , 26525044132374655 )
( 57 , 1 )
( 58 , 0 )
( 59 , 0 )
( 60 , 0 )

    };
    \end{axis}
\end{tikzpicture}
\qquad
\begin{tikzpicture}[scale = 0.5]
    \begin{axis}[
        ymajorgrids=true,
        bar width=4pt,
        width=16cm,
        height=9cm,
        enlargelimits=0.05,
        xtick={1,5,10,15,20,25,30,35,40,45,50,55,60},
        ylabel style={at={(axis description cs:0.02,0.5)}, anchor=south},
        xlabel style={at={(axis description cs:0.5,-0.05)}, anchor=north}
    ]
    \addplot[ybar, fill=teal] coordinates {
    ( 0 , 1 )
( 1 , 196 )
( 2 , 19306 )
( 3 , 1274196 )
( 4 , 63391251 )
( 5 , 2535393225 )
( 6 , 84896763225 )
( 7 , 2446831329450 )
( 8 , 61925233080075 )
( 9 , 1396956113433700 )
( 10 , 28414967869572250 )
( 11 , 525870540212301250 )
( 12 , 8918533676370381900 )
( 13 , 139409991158971759125 )
( 14 , 2017919366694745786275 )
( 15 , 27149638012670624172075 )
( 16 , 340567938360558222392250 )
( 17 , 3992947128610711116239475 )
( 18 , 43841525884783552036632875 )
( 19 , 451486169059371821226382625 )
( 20 , 4365859060556611154308154750 )
( 21 , 39674746342591342437209454300 )
( 22 , 338997428322842651847153802500 )
( 23 , 2724019061981828396048724686025 )
( 24 , 20584015265255897078496056099625 )
( 25 , 146222775777318397736414883072189 )
( 26 , 975905984859966836273392402069419 )
( 27 , 6114117649397039576669414460541259 )
( 28 , 35916465270206788042340148949701644 )
( 29 , 197540464695914618567770036452810689 )
( 30 , 1015418039808624709931452906778990205 )
( 31 , 4867670790438007612131151031681555769 )
( 32 , 21705028363989010924253096327972385048 )
( 33 , 89747084693482977133257139362462420047 )
( 34 , 342846387431349168208200684273446843106 )
( 35 , 1204701507054345134264862207452783579806 )
( 36 , 3872980956001622739044318711229030038662 )
( 37 , 11318159448434091457556009660754267812138 )
( 38 , 29825221581583490847197483743855401808959 )
( 39 , 70159080085888339307980976816299067095635 )
( 40 , 145427559106496712579395151406452463777445 )
( 41 , 261141292853417120382223824519535244284951 )
( 42 , 397033224329160886423244911509658729802187 )
( 43 , 495320759111482725579642202711929249682713 )
( 44 , 485572274508694498645127586841883681768154 )
( 45 , 352416247536429872370577727415430061608050 )
( 46 , 174638268505609356306354595690355905639890 )
( 47 , 53019232482700194180514846688408102909865 )
( 48 , 8537954804017889441570442855880054392420 )
( 49 , 601148252994082451810621389690349944600 )
( 50 , 14180822366223271213966644985790473253 )
( 51 , 76567378147065948784283500486808772 )
( 52 , 53761677784529238036586921925378 )
( 53 , 2057525347426902264121856845 )
( 54 , 1050352205594063827477 )
( 55 , 657723055999 )
( 56 , 1 )
( 57 , 0 )
( 58 , 0 )
( 59 , 0 )
( 60 , 0 )
    };
    \end{axis}
\end{tikzpicture}
\caption{Upper left: coefficients of $\Hilb(R_{(1^{60}),(1^{60})}; q)$. Upper right: coefficients of $\Hilb(R_{(2^{30}),(2^{30})};q)$. Lower left: coefficients of $\Hilb(R_{(3^{20}),(3^{20})};q)$. Lower right: coefficients of $\Hilb(R_{(4^{15}),(4^{15})};q)$.}
\label{fig:coefficient-figure}
\end{figure}

\subsection{Dual Cauchy}
The RSK correspondence applied to contingency tables yields the fundamental Cauchy kernel identity in symmetric function theory
\begin{equation}
    \prod_{i,j \, \geq \, 1} \frac{1}{1-x_iy_j} = \sum_{A} \xxx^{\row(A)} \yyy^{\col(A)} = \sum_\lambda s_\lambda(\xxx)  s_\lambda(\yyy)
\end{equation}
where the middle sum is over all infinite $\ZZ_{\geq 0}$-matrices $A$ (that is, all contingency tables) and the sum on the right is over all partitions $\lambda$. Many results involving the Cauchy kernel have `dual Cauchy' analogs which reflect the identity
\begin{equation}
    \prod_{i,j \, \geq \, 1} (1 + x_i y_j) = \sum_B \xxx^{\row(B)} \yyy^{\col(B)} = \sum_\lambda s_\lambda(\xxx) s_{\lambda'}(\yyy).
\end{equation}
Here the middle sum is over all infinite $0,1$-matrices $B$, the sum on the right is over all partitions $\lambda$, and $\lambda'$ is the conjugate of $\lambda$.

\begin{problem}
    \label{prob:dual-cauchy}
    Let $\alpha,\beta \models_0 n$ and let $\DDD_{\alpha,\beta}$ be the locus of $0,1$-matrices with row sums $\alpha$ and column sums $\beta$. Study the orbit harmonics quotient ring $\RRR(\DDD_{\alpha,\beta})$.
\end{problem}

Exterior algebras give another possible algebraic model for $\DDD_{\alpha,\beta}$. Suppose $\ell(\alpha) = k, \ell(\beta) = p$, let $\ttheta_{k \times p} = (\theta_{i,j})$ be a $k \times p$ matrix of anticommuting variables, and let $\wedge \{ \ttheta_{k \times p}\}$ be the exterior algebra over these variables. We have an ideal $I_{\alpha,\beta}^\ttheta \subseteq \wedge \{ \ttheta_{k \times p}\}$ defined in the same way as $I_{\alpha,\beta} \subseteq \FF[\xxx_{k \times p}]$. Explicitly, the ideal $I_{\alpha,\beta}^\ttheta \subseteq \wedge \{ \ttheta_{k \times p}\}$ is generated by $\dots$
\begin{itemize}
    \item all row sums $\theta_{i,1} + \cdots + \theta_{i,p}$ for $1 \leq i \leq k$,
    \item all column sums $\theta_{1,j} + \cdots + \theta_{k,j}$ for $1 \leq j \leq p$,
    \item all products $\theta_{i,j_1} \wedge \cdots  \wedge \theta_{i,j_q}$ where $1 \leq j_1 < \cdots < j_q \leq p$ and $q > \alpha_i$, and
    \item all products $\theta_{i_1,j} \wedge \cdots \wedge\theta_{i_q,j}$ where $1 \leq i_1 < \cdots < i_q \leq k$ and $q > \beta_j$.
\end{itemize}The quotient ring $\wedge \{ \ttheta_{k \times p} \}/I^\ttheta_{\alpha,\beta}$ is likely related to $\DDD_{\alpha,\beta}$.

\subsection{Graded Ehrhart theory}
A {\em lattice polytope} $\PPP \subseteq \RR^N$ is the convex hull of a finite set of points in the integer lattice $\ZZ^N$. For an integer $m \geq 0$, one has the {\em dilate}  $m \PPP := \{mp \,:\, p \in \PPP \}$ and the lattice point count $i_\PPP(m) := \#(m \PPP \cap \ZZ^N)$. The {\em Ehrhart series} of $\PPP$ is the formal power series
\begin{equation}
    \Eseries_\PPP(t) := \sum_{m \geq 0} i_\PPP(m) \cdot t^m = \sum_{m \geq 0} \#(m \PPP \cap \ZZ^N) \cdot t^m \in \QQ[[t]].
\end{equation}
Ehrhart proved \cite{Ehrhart} that $i_\PPP(m)$ is a polynomial function of $m$ of degree $\dim(\PPP)$. The polynomial $i_\PPP(m)$ is called the {\em Ehrhart polynomial} of $\PPP$. It follows that $\Eseries_\PPP(t) \in \QQ(t)$ is a rational function of $t$ called the {\em Ehrhart series} of $\PPP$.

Reiner and the second author used orbit harmonics to define \cite{ReinerRhoades} a $q$-analog of the Ehrhart series as follows. If $\PPP \subseteq \RR^N$ is a lattice polytope, the intersection $m \PPP \cap \ZZ^N$ is a finite locus for each $m$, and we have the graded quotient ring $\RRR(m \PPP \cap \ZZ^N)$. Its Hilbert series $\Hilb(\RRR(m \PPP \cap \ZZ^N);q) \in \QQ[q]$ is a polynomial in $q$. The {\em $q$-Ehrhart series} is given by
\begin{equation}
    \Eseries_\PPP(t,q) := \sum_{m \geq 0} \Hilb(\RRR(m \PPP \cap \ZZ^N);q) \cdot t^m \in \QQ[q][[t]].
\end{equation}
One has $\Eseries_\PPP(t,1) = \Eseries_\PPP(t)$. It is conjectured  that $\Eseries_\PPP(t,q)$ is in fact a rational function in $\QQ(t,q)$; this conjecture is proven   for antiblocking and order polytopes \cite{ReinerRhoades}.

Suppose $\alpha, \beta \models_0 n$ satisfy $\ell(\alpha) = k$ and $\ell(\beta) = p$. The {\em transportation polytope} $\PPP_{\alpha,\beta} \subseteq \Mat_{k \times p}(\RR)$ is the convex hull
\begin{equation}
    \PPP_{\alpha,\beta} := \conv( A \,:\, A \in \CCC_{\alpha,\beta})
\end{equation}
of the set $\CCC_{\alpha,\beta}$ of contingency tables inside the space $\Mat_{k \times p}(\RR)$ of matrices. When $\alpha = \beta = (1^n)$, this is also called the {\em Birkhoff polytope}.
The following is immediate from Corollary~\ref{cor:hilbert-series}.

\begin{corollary}
    \label{cor:birkhoff}
    The graded Ehrhart series of $\PPP_{\alpha,\beta}$ is 
    \begin{equation}
    \label{eq:transportiation-ehrhart}
        \Eseries_{\PPP_{\alpha,\beta}}(t,q) = \sum_{m \geq 0} \left( \sum_{\lambda \vdash mn} K_{\lambda,m\alpha} \cdot K_{\lambda,m\beta} \cdot  q^{mn-\lambda_1} \right) \cdot t^m.
    \end{equation}
\end{corollary}

Does this power series live in $\QQ(t,q)$? There is no known formula for the classical Ehrhart polynomial $i_{\PPP_{\alpha,\beta}}(m)$. Beck and Pixton computed \cite{Beck} this polynomial for $\alpha = \beta = (1^n)$ and $n \leq 9$; the resulting expressions are very complicated. Since $q$-Ehrhart series can be significantly more intricate than their classical counterparts, finding an explicit rational expression for $\Eseries_{\PPP_{\alpha,\beta}}(t,q)$ is probably impossible.

Let $\overline{\PPP}$ be the relative interior of a $d$-dimensional lattice polytope $\PPP \subseteq \RR^N$. We have an interior version
\begin{equation}
    \overline{\Eseries}_\PPP(t,q) := \sum_{m \geq 1} \Hilb(\RRR(\overline{m\PPP} \cap \ZZ^N);q) \cdot t^m \in \QQ[q][[t]].
\end{equation}
Reiner and the second author conjecture \cite{ReinerRhoades} the following $q$-analog of Ehrhart reciprocity:
\begin{equation}
    \label{eq:q-reciprocity}
    q^d \cdot \overline{\Eseries}_\PPP(t,q) = (-1)^{d+1} \Eseries(t^{-1}, q^{-1}).
\end{equation}

Let $\alpha,\beta$ be {\em strict} compositions of $n$ (only positive parts) with $\ell(\alpha)=k$ and $\ell(\beta) = p$. If $A  \in \CCC_{\alpha,\beta}$ is a lattice point of $\PPP_{\alpha,\beta}$, one has $A \in \overline{\PPP}_{\alpha,\beta}$ if and only if all entries of $A$ are strictly positive. Let $\one_k$ and $\one_p$ be the all-ones vectors of lengths $k$ and $p$
and let $J$ be the $k \times p$ matrix whose entries are all 1. We have a bijection
\begin{equation}
    \label{eq:transportation-interior}
    \psi: \CCC_{\alpha - p \cdot \one_k, \beta - k \cdot \one_p} \xrightarrow{\, \, \sim \, \, } \CCC_{\alpha,\beta} \cap \overline{\PPP}_{\alpha,\beta}  \quad \quad \psi: B \mapsto B + J.
\end{equation}
Here $\alpha - p \cdot \one_k$ and $\beta - k \cdot \one_p$ are componentwise subtraction.
 The domain of $\psi$ is empty if either $\alpha - p \cdot \one_k$ or $\beta - k \cdot \one_p$ has a negative entry.
Since the orbit harmonics ring $\RRR(-)$ is unchanged by a locus translation such as $\psi$, one has 
\begin{equation}
    \label{eq:transportation-interior-ehrhart}
    \overline{\Eseries}_{\PPP_{\alpha,\beta}}(t,q) = \sum_{m \geq 1} \left( \sum_{\lambda \vdash (mn-kp)} K_{\lambda,m \alpha - p \cdot\one_k} \cdot K_{\lambda,m \beta - k \cdot \one_p} \cdot q^{mn - kp - \lambda_1} \right) \cdot t^m
\end{equation}
where the expression $(\,\cdots)$ is zero if $m\alpha - p \cdot \one_k$ or $m \beta - k \cdot \one_p$ has a negative entry. Can one combine \eqref{eq:transportiation-ehrhart} and \eqref{eq:transportation-interior-ehrhart} to prove the $q$-reciprocity conjecture \eqref{eq:q-reciprocity} in the case of transportation polytopes? Theorem~\ref{thm:zigzag-characterization} gives the expressions
\begin{multline}
\label{eq:transportation-reformulation}
    \Eseries_{\PPP_{\alpha,\beta}}(t,q) = \sum_{m \geq 0} \left(  \sum_{A \in \CCC_{m\alpha,m\beta}} q^{mn - \zigzag(A)} \right) \cdot t^m \quad \text{and}  \\
    \overline{\Eseries}_{\PPP_{\alpha,\beta}}(t,q) = \sum_{m \geq 1} \left( \sum_{B \in \CCC_{m\alpha-p \cdot \one_k, m\beta - k \cdot\one_p}} q^{mn - kp - \zigzag(B)} \right) \cdot t^m.
\end{multline}
Since $\zigzag(\psi(B)) = \zigzag(B)+k+p-1$ and $\dim(\PPP_{\alpha,\beta}) = (k-1)(p-1)$, the formulation \eqref{eq:transportation-reformulation} may be useful for proving $q$-reciprocity.

\section{Acknowledgements}

The authors are grateful to Sara Billey, Carly Klivans, Tanny Libman, and Hanbaek Lyu for helpful conversations. B. Rhoades was partially supported by NSF Grant DMS-2246846. B. Rhoades thanks an anonymous referee of \cite{Rhoades} for the suggestion to look at contingency tables. J. Oh was partially supported by KIAS Individual Grant (HP083401) at  June E Huh Center for Mathematical Challenges in Korea Institute for Advanced Study.

\bibliographystyle{abbrv}
\bibliography{bibliography}

\begin{thebibliography}{10}

\bibitem{ARR}
D.~Armstrong, V.~Reiner, and B.~Rhoades.
\newblock Parking spaces.
\newblock {\em {A}dv. {M}ath.}, 269:647--706, 2015.

\bibitem{BDJ}
J.~Baik, P.~Deift, and K.~Johansson.
\newblock On the distribution of the length of the longest increasing subsequence of random permutations.
\newblock {\em J. {A}mer. {M}ath. {S}oc.}, 12(4):1119--1178, 1999.

\bibitem{Beck}
M.~Beck and D.~Pixton.
\newblock The {E}hrhart polynomial of the {B}irkhoff polytope.
\newblock {\em Discrete {C}omput. {G}eom.}, 30(4):623--637, 2003.

\bibitem{Bergeron}
F.~Bergeron.
\newblock The bosonic-fermionic diagonal coinvariant modules conjecture.
\newblock {\em arXiv preprint arXiv:2005.00924}, 2020.

\bibitem{Chen}
W.~Y. Chen.
\newblock Log-concavity and q-log-convexity conjectures on the longest increasing subsequences of permutations.
\newblock Preprint, 2008. {\tt arXiv:0806.3392}.

\bibitem{CMR}
R.~Chou, T.~Matsumura, and B.~Rhoades.
\newblock Equivariant cohomology and orbit harmonics.
\newblock Preprint, 2024. {\tt arXiv:2410.02105}.

\bibitem{CDHP}
C.~Crowley, G.~Dorphalen-{B}arry, A.~{H}enriques, and N.~Proudfoot.
\newblock The geometry of zonotopal algebras {I}: cohomology of graphical configuration spaces.
\newblock Preprint, 2025. {\tt arXiv:2502.12768}.

\bibitem{Ehrhart}
E.~Ehrhart.
\newblock Sur un probl{\`e}me de g{\'e}om{\'e}trie diophantienne lin{\'e}aire. ii.
\newblock 1967.

\bibitem{Fulton}
W.~Fulton.
\newblock {\em Young tableaux: with applications to representation theory and geometry}.
\newblock Number~35. Cambridge University Press, 1997.

\bibitem{GM}
M.~Goresky and R.~MacPherson.
\newblock On the spectrum of the equivariant cohomology ring.
\newblock {\em {C}anad. {J}. {M}ath.}, 62(2):262--283, 2010.

\bibitem{Griffin}
S.~Griffin.
\newblock Ordered set partitions, {G}arsia-{P}rocesi modules, and rank varieties.
\newblock {\em {T}rans. {A}mer. {M}ath. {S}oc.}, 374(4):2609--2660, 2021.

\bibitem{HRS}
J.~Haglund, B.~Rhoades, and M.~Shimozono.
\newblock Ordered set partitions, generalized coinvariant algebras, and the delta conjecture.
\newblock {\em {A}dv. {M}ath.}, 329:851--915, 2018.

\bibitem{Haiman}
M.~Haiman.
\newblock Vanishing theorems and character formulas for the {H}ilbert scheme of points in the plane.
\newblock {\em Invent. {M}ath.}, 149:371--407, 2002.

\bibitem{KRR}
J.~Kim, J.~M. Rabin, and B.~Rhoades.
\newblock The combinatorics of supertorus sheaf cohomology.
\newblock {\em J. {G}eom. {P}hys.}, 193:104963, 2023.

\bibitem{KR}
J.~Kim and B.~Rhoades.
\newblock Lefschetz theory for exterior algebras and fermionic diagonal coinvariants.
\newblock {\em Int. {M}ath. {R}es. {N}ot. {IMRN}}, 2022(4):2906--2933, 2022.

\bibitem{KRSkein}
J.~Kim and B.~Rhoades.
\newblock Set partitions, fermions, and skein relations.
\newblock {\em Int. {M}ath. {R}es. {N}ot. {IMRN}}, 2023(11):9427--9480, 2023.

\bibitem{Knuth}
D.~Knuth.
\newblock Permutations, matrices, and generalized {Y}oung tableaux.
\newblock {\em Pacific {J}. {M}ath.}, 34(3):709--727, 1970.

\bibitem{Kostant}
B.~Kostant.
\newblock Lie group representations on polynomial rings.
\newblock {\em {A}mer. {J}. {M}ath.x}, 85:327--404, 1963.

\bibitem{Lentfer}
J.~Lentfer.
\newblock The sign character of the triagonal fermionic coinvariant ring.
\newblock Preprint, 2025. {\tt arXiv:2501.09920}.

\bibitem{Liu}
M.~J. Liu.
\newblock Viennot shadows and graded module structure in colored permutation groups.
\newblock Accepted, {\em Comb. Theory}, 2025. {\tt arXiv:2401.07850}.

\bibitem{LMRZ}
M.~J. Liu, Y.~Ma, B.~Rhoades, and H.~Zhu.
\newblock Involution matrix loci and orbit harmonics.
\newblock Accepted, {\em Math. Z.}, 2025. {\tt arXiv:2409.06175}.

\bibitem{NR}
J.~Novak and B.~Rhoades.
\newblock Increasing subsequences and {K}ronecker coefficients.
\newblock In {\em Proc. {S}ympos. {P}ure {M}ath}, volume 110, pages 87--91, 2024.

\bibitem{OR}
J.~Oh and B.~Rhoades.
\newblock Cyclic sieving and orbit harmonics.
\newblock {\em {M}ath. {Z}.}, 300(1):639--660, 2022.

\bibitem{RRT}
M.~Reineke, B.~Rhoades, and V.~Tewari.
\newblock Zonotopal algebras, orbit harmonics, and {D}onaldson--{T}homas invariants of symmetric quivers.
\newblock {\em {I}nt. {M}ath. {R}es. {N}ot. {IMRN}}, 2023(23):20169--20210, 2023.

\bibitem{ReinerRhoades}
V.~Reiner and B.~Rhoades.
\newblock Harmonics and graded {E}hrhart theory.
\newblock Preprint, 2024. {\tt arXiv:2407.06511}.

\bibitem{Rhoades}
B.~Rhoades.
\newblock Increasing subsequences, matrix loci and {V}iennot shadows.
\newblock {\em Forum {M}ath., {S}igma}, 12:e97, 2024.

\bibitem{RW}
B.~Rhoades and A.~T. Wilson.
\newblock The {H}ilbert series of the superspace coinvariant ring.
\newblock {\em Forum {M}ath. {P}i}, 12:e16, 2024.

\bibitem{Sagan}
B.~E. Sagan.
\newblock {\em The symmetric group: representations, combinatorial algorithms, and symmetric functions}, volume 203.
\newblock Springer Science \& Business Media, 2013.

\bibitem{Schensted}
C.~Schensted.
\newblock Longest increasing and decreasing subsequences.
\newblock {\em Canad. {J}. {M}ath.}, 13:179--191, 1961.

\bibitem{Stanley}
R.~Stanley.
\newblock {\em Enumerative {C}ombinatorics, {V}olume 2}.
\newblock Cambridge {U}niversity {P}ress, 1999.

\bibitem{StanleyDifferential}
R.~P. Stanley.
\newblock Differential posets.
\newblock {\em {J}. {A}mer. {M}ath. {S}oc.}, 1(4):919--961, 1988.

\bibitem{Viennot}
G.~Viennot.
\newblock Une forme g{\'e}om{\'e}trique de la correspondance de {R}obinson-{S}chensted.
\newblock In {\em Combinatoire et {R}epr{\'e}sentation du {G}roupe {S}ym{\'e}trique: {A}ctes de la {T}able {R}onde du {CNRS} tenue {\`a} l'{U}niversit{\'e} {L}ouis-{P}asteur de {S}trasbourg, 26 au 30 avril 1976}, pages 29--58. Springer, 1977.

\end{thebibliography}

\end{document}